\UseRawInputEncoding
\documentclass[english,filigrane, couv, hyper, noyhmath,oldmathcal]{smfart}
\usepackage{afterpage, color, smfthm,
enumitem,rotating,booktabs,tabularx,amssymb,mathrsfs}
\usepackage[normalem]{ulem}

\catcode`\@=11
\def\smf@datereception{18 juin 2019}

\def\smf@dateaccepte{19 janvier 2022}

\usepackage[latin1]{inputenc}

\usepackage{tikz}
\usetikzlibrary{cd}

\newcommand\yh{}
\newcommand \tetrecitw{\textit}
\newcommand\tetrecitwi{}
\newcommand\restrict[1]{_{|#1}}

\usepackage{stackengine}

\newcommand\cX{{\mathring{X}}}

\newcommand{\Span}{{\operatorname{Span}}}

\renewcommand{\div}{\operatorname{div}}
\newcommand{\Gmin}{{\operatorname{G}}}

\newcommand\face{{\mathcal F}}
\newcommand\mcD{{\mathcal D}}
\newcommand\mcU{{\mathcal U}}

\newcommand\uo{{\underline{o}}}

\newcommand\cone{{\mathcal C}}
\newcommand\X{{\mathcal X}}
\newcommand\Yc{{\mathcal Y}}
\newcommand\Tau{{\mathcal T}}
\newcommand\lh{{\mathfrak h}}
\renewcommand\lg{{\mathfrak g}}\newcommand\Liel{{\mathfrak l}}
\newcommand\ad{{\operatorname{ad}}}
\newcommand\Image{{\operatorname{Im}}}
\newcommand\Ker{{\operatorname{Ker}}}
\newcommand\Ho{{\operatorname{H}^0}}
\newcommand\Hd{{\operatorname{H}}}
\newcommand\Lie{{\operatorname{Lie}}}

\newcommand\Aut{{\operatorname{Aut}}}
\newcommand\Hom{{\operatorname{Hom}}}\newcommand\Pic{{\operatorname{Pic}}}
\newcommand\lp{{\mathfrak p}}

\newcommand\ZZ{{\mathbb Z}}\newcommand\NN{{\mathbb
N}}\newcommand\QQ{{\mathbb Q}}
\newcommand\CC{\mathbb C}\newcommand\PP{\mathbb P}
\newcommand\longto{\longrightarrow}

\newcommand\Li{{\mathcal L}}\newcommand\Mi{{\mathcal
M}}\newcommand\Ni{{\mathcal N}}
\newcommand\bl{\ell}
\newcommand\bs{\text{\fontencoding{T1}\fontfamily{frc}\fontseries{m}\fontshape{n}\selectfont
s}}
\newcommand\inv{^{-1}}

\newcommand{\Stab}{\operatorname{Stab}}
\newcommand\cChi{{\mathring\Chi}}

\newcommand\cmcX{\mathring{\mathcal{X}}}
\newcommand\Chi{{\mathfrak{X}}}

\newcommand\dcirc[1]{\stackinset{c}{}{t}{-3.5pt}
   {$\mkern2.5mu\scriptscriptstyle\circ\mkern-2mu\circ$}{$#1$}}
\newcommand\ccdX{{\dcirc{\mathcal{X}}}}
\newcommand\ccChi{{\dcirc{\Chi}}}
\newcommand\cbeta{{\mathring{\beta}}}
\newcommand\eps{{\varepsilon}}

\newcommand{\mcX}{\mathcal{X}}

\newtheorem{thm}{Theorem}[section]
\newtheorem{lemma}[thm]{Lemma}
\newtheorem{PROP}[thm]{Proposition}
\newtheorem{CORO}[thm]{Corollary}
\newtheorem{conjecture}[thm]{Conjecture}

\theoremstyle{remark} 
\newtheorem{remark}[thm]{Remark}

\title[On faces of the tensor cone of  KM Lie algebras]{On the faces of
the tensor cone\\ of symmetrizable Kac-Moody Lie algebras}

\author[S. Kumar]{Shrawan Kumar}
\address{Department of Mathematics\\University of North
Carolina\\Chapel Hill, NC 27599-3250, USA}
\email{shrawan@email.unc.edu}
\author[N. Ressayre]{Nicolas Ressayre}
\address{Université Claude Bernard Lyon 1, ICJ UMR5208, CNRS, Ecole
Centrale de Lyon, INSA Lyon, Université Jean Monnet, 69622
Villeurbanne, France.}
\email{ressayre@math.univ-lyon1.fr}

\begin{abstract}
In this paper, we are interested in the decomposition of the tensor
product of two representations of
a symmetrizable Kac-Moody Lie algebra $\lg$, or more precisely in the
tensor cone of~$\lg$.
As usual, we parametrize the integrable, highest weight (irreducible)
representations of~$\lg$ by their highest weights. Then,
the triples of such representations such that the last one is
contained in the tensor product of the  first two is a semigroup.
This semigroup  generates a rational convex cone $ \Gamma(\lg)$ called
tensor cone.
If $\lg$~is finite-dimensional, $\Gamma(\lg)$~is a polyhedral convex
cone. In 2006, Belkale and the first author described this cone by an
explicit finite list of
inequalities.
In 2010, this list of inequalities was proved to be irredundant by  the
second author:
each such inequality corresponds to a codimension one face.
In general, $\Gamma(\lg)$~is neither polyhedral, nor closed.
Brown and the first author obtained a list of inequalities that
describe $\Gamma(\lg)$ conjecturally. Here, we prove that each of
these inequalities corresponds to a codimension one face of~$\Gamma(\lg)$.
\end{abstract}

\alttitle{À propos des faces du cône tensoriel d'une algèbre de Kac-Moody}

\begin{altabstract}
Dans cet article, nous nous int\'eressons à la d\'ecomposition du produit
tensoriel de deux repr\'esentations d'une algèbre de Kac-Moody
sym\'etrisable $\lg$, et plus pr\'ecis\'ement au cône tensoriel de $\lg$.
Comme d'habitude, nous param\'etrons les repr\'esentations
irr\'eductibles int\'egrables et  de plus haut poids par ledit plus haut
poids.
  Alors, les triplets de telles repr\'esentations telles que la troisième
  s'injecte dans le produit tensoriel des deux premières est un
  semi-groupe.
Ces triplets engendrent un cône convexe rationnel $\Gamma(\lg)$ que
nous appelons le  {\it cone tensoriel}.
Lorsque $\lg$ est de dimension finie, $\Gamma(\lg)$ est un cône
convexe poly\'edral. En 2006, Belkale et le premier auteur ont d\'ecrit ce
cône par une liste finie explicite d'in\'egalit\'es lin\'eaires.
En 2010, le second auteur a montr\'e que  cette liste d'in\'egalit\'es
n'est pas redondante~:
chaque in\'egalit\'e correspond à une face de codimension un.
En g\'en\'eral,  $\Gamma(\lg)$ n'est ni ferm\'e, ni poly\'edral.
Brown et le premier auteur  ont obtenu une liste
d'in\'egalit\'es qui d\'ecrit conjecturalement le cône $\Gamma(\lg)$.
Nous montrons ici que chacune de ces in\'egalit\'es correspond à une face
de codimension un de $\Gamma(\lg)$.
\end{altabstract}

\begin{document}
\maketitle
\lefthyphenmin4
\righthyphenmin4
\frenchspacing
\lefthyphenmin4
\righthyphenmin4

\section{Introduction}\label{sec1}

Let $A$~be a symmetrizable irreducible GCM (generalized Cartan matrix) of size~$l+1$.
Let~$\lh\supset\{\alpha_0^\vee,\dots,\alpha_l^\vee\}$ and
$\lh^*\supset\{\alpha_0,\dots,\alpha_l\}=:\Delta$ be a realization of~$A$ over the complex numbers~$\mathbb{C}$. 
We fix an integral form~$\lh_\ZZ\subset\lh$ containing each~$\alpha_i^\vee$, such that ${\lh^*_\ZZ:=\Hom(\lh_\ZZ,\ZZ)}$ contains $\Delta$
and such that $\lh_\ZZ/\bigoplus_i\ZZ\alpha_i^\vee$~is torsion-free.

Set ${\lh_\QQ^*=\lh_\ZZ^*\otimes\QQ\subset\lh^*}$,  
${P_{+,\QQ}:=\{\lambda\in\lh_\QQ^*\,:\,\langle\alpha_i^\vee,\lambda\rangle\geq
0\quad\forall i\}}$, and $P_+:=\lh_\ZZ^*\cap P_{+,\QQ}$.

Let $\lg=\lg(A)$~be the associated Kac-Moody (KM) Lie algebra over~$\mathbb{C}$ with Cartan
subalgebra~$\lh$. 
For~$\lambda\in P_+$,  $L(\lambda)$ denotes the (irreducible) integrable,
highest weight representation of~$\lg$ with highest weight~$\lambda$. 
Define the (rational) {\it tensor cone} as~$$
\Gamma(\lg):=\{(\lambda_1,\lambda_2,\mu)\in P_{+,\QQ}^3:\exists N\geq 1
\,\text{such that}\, L(N\mu)\subset L(N\lambda_1)\otimes L(N\lambda_2)\}.
$$ 

The aim of this paper is to describe  facets (codimension one faces) of this cone. 
Before describing our result, we
recall from \cite{BrownKumar} a conjectural
description of~$\Gamma(\lg)$, due to Brown and the first author.
We need some more notation. 

Fix $\{x_0,\dots,x_l\}\in \lh$ to be dual of the simple roots: $\langle
\alpha_j, x_i\rangle=\delta_i^j$.
Let~$Q=\bigoplus_{i=0}^l\ZZ\alpha_i$ denote the root lattice.
Let~$X =G/B$  be the standard full KM-flag variety
associated to~$\lg$, where $G$~is the `minimal' Kac-Moody group with Lie algebra $\lg$ and $B$~is the standard Borel subgroup of~$G$.  For~$w$ in the Weyl group  $W$ of~$G$, let $X_w = \overline{BwB/B}\subset X$~be the corresponding Schubert variety. Let $\{\eps^w\}_{w\in W} \subset \Hd^*(X,\ZZ)$~be the 
(Schubert) basis dual (with respect to  the standard pairing) to the
basis of the singular homology of~$X$ given by the fundamental classes
of~$X_w$.

Let $P\supset B$~be a (standard) parabolic subgroup and let $X_P:=G/P$~be the corresponding partial flag variety. Let $W_P$~be the Weyl group of~$P$ (which is, by definition, the Weyl group of the Levi $L$ of~$P$) and let $W^P$~be the set of minimal length  representatives 
of cosets in~$W/W_P$. 
The projection map~$X\to X_P$ induces an injective homomorphism $\Hd^*(X_P, \ZZ) \to \Hd^*(X, \ZZ)$ and $\Hd^*(X_P, \ZZ) $~has the Schubert basis
$\{\eps^w_P\}_{w\in W^P}$ such that~$\eps^w_P$ goes to~$\eps^w$ for
any~$w\in W^P$. As defined by Belkale and the first author \cite[$\S$6]{BK} in the finite-dimensional case and extended by the first author in \cite{Kumar:CDSWconj} for any symmetrizable Kac-Moody case (see \cite[\S,7]{BrownKumar} for more details), there is a new deformed product $\odot_0$ in~$\Hd^*(X_P, \ZZ)$, which is commutative and associative. Now, we are ready to state Brown-Kumar's conjecture \cite{BrownKumar}.

\begin{conjecture} \label{conj}
Let $\lg$~be any indecomposable symmetrizable Kac-Moody  Lie algebra
and let~$(\lambda_1,\lambda_2, \mu)\in P_+^{3}$. Assume further that none of~$\lambda_j$  and $\mu$ are $W$\yh-invariant and\break $\mu-\sum_{j=1}^2 \lambda_j\in Q$.
Then, the following are equivalent:

{\rm(a)}  $(\lambda_1,\lambda_2, \mu)\in \Gamma(\lg)$.

{\rm(b)} For every standard maximal parabolic subgroup~$P$ in~$G$ and every choice of
triples  $(w_1,w_2, v)\in (W^P)^{3}$ such that~$\eps_P^v$ occurs with coefficient~$1$ in 
the deformed product
$$\eps_P^{w_1}\odot_0 \eps_P^{w_2}
\in \bigl(\Hd^*(X_P,\ZZ), \odot_0\bigr),$$
  the following inequality  holds:
    \[
\lambda_1(w_1x_{P})+ \lambda_2(w_2x_{P})-\mu(vx_P)\geq 0, \tag{$I^P_{(w_1, w_2,v)}$}
\]
where $\alpha_{i_P}$~is the (unique) simple root not in the Levi of~$P$
and $x_P:=x_{i_P}$.
\end{conjecture}
 
Note that if $\lambda_1$~is $W$\yh-invariant, $L(\lambda_1)$~is one-dimensional and hence $L(\lambda_1)\otimes L(\lambda_2)$~is
irreducible. 

In the case where~$\lg$~is a semisimple Lie algebra,
Conjecture~\ref{conj} was proved by Belkale and the first author in \cite{BK}. The
following result is due to the second author.
\begin{thm}[\cite{R:KM1}]\label{thm0.1} 
In the case where~$\lg$~is affine untwisted, Conjecture~\ref{conj} holds. 
\end{thm}
The
conjecture in the general symmetrizable case is still open.  But it is
conceivable that the inductive proof in the case of affine
$\mathfrak{g}$ obtained by the second author might be amenable to handle the general symmetrizable case. 

Let us come back to the case where~$\lg$~is semisimple. 
Then, $\Gamma(\lg)$~is a closed convex polyhedral cone, and
Conjecture~\ref{conj} (Belkale-Kumar's theorem) describes $\Gamma(\lg)$ in~$(\lh_\QQ^*)^3$
by (finitely many) explicit  inequalities. (Recall that a rational cone $\mathcal{C}$~is called {\it convex} if for~$x, y\in \mathcal{C}$ and $0<\alpha <1, \alpha\in \mathbb{Q}$,  $\alpha x+ (1-\alpha) y\in \mathcal{C}$.) In the case of~$\lg =
{\mathfrak{sl}}_n$, a larger set of inequalities describing 
$\Gamma(\lg)$ was conjectured by Horn \cite{Horn:conj} and proved by Klyachko \cite{Kly} (combining the saturation result of Knutson-Tao \cite{KT:saturation}).
A  larger set of   inequalities describing $\Gamma(\lg)$ for any semisimple $\lg$ was known earlier  (see
\cite{BerSja:Mumford}). 
The  irredundancy of the above  set of inequalities~${I^P_{(w_1, w_2,v)}}$ was proved by Knutson-Tao-Woodward in type A
\cite{KTW} and by the second author in general \cite{GITEigen}. (See \cite[$\S1$] {Kumar:survey} for more details on the history.) 
The irredundancy  assertion is the statement  that each inequality~$I^P_{(w_1, w_2,v)}$ in
Conjecture~\ref{conj} corresponds to a face of~$\Gamma(\lg)$ of
codimension one. The aim of this paper is to extend this
result to any symmetrizable Kac-Moody Lie algebra. We, in fact, prove the following (stronger) result for any (not necessarily maximal) standard parabolic subgroup~$P$. 

\begin{thm}\label{th:codimface}
  Let $\lg$~be any indecomposable symmetrizable Kac-Moody  Lie
  algebra. 
Let $P$~be a  standard parabolic subgroup in~$G$ and let $(w_1,w_2, v)\in (W^P)^{3}$~be a triple such that $\eps_P^v$~occurs with coefficient~$1$ in 
the deformed product
$$\eps_P^{w_1}\odot_0 \eps_P^{w_2}
\in \bigl(\Hd^*(X_P,\ZZ), \odot_0\bigr).$$

Then, the  set of~$(\lambda_1,\lambda_2, \mu)\in \Gamma(\lg)$  such that for all $\alpha_j\not\in\Delta(P)$,
\[
\lambda_1(w_1x_{j})+ \lambda_2(w_2x_{j})-\mu(vx_j)=0 \tag{$I^j_{(w_1, w_2,v)}$}
\]
has codimension~$\sharp (\Delta\backslash \Delta(P))$ in~$\Gamma(\lg)$, where $\Delta(P) \subset \Delta$~is the set of simple roots of the Levi subgroup~$L$ of~$P$.
\end{thm}

Let~$\cone$ denote the cone determined by the inequalities in
Conjecture~\ref{conj}. 
For~$P$ maximal, Theorem~\ref{th:codimface} implies
that if one removes any of the inequalities~$I^P_{(w_1, w_2,v)}$,  the cone thus obtained is strictly larger than~$\cone$.

Theorem~\ref{th:codimface} implies that~$\cone$~is locally
polyhedral. This property of~$\cone$ plays an important role in the
inductive proof of Theorem~1 from \cite{R:KM1}. (Note that in
\cite{R:KM1}, the local polyhedrality is proved in a totally different
way.)
As a consequence, one can hopefully think about
Theorem~\ref{th:codimface} as a first step towards a proof of Conjecture~\ref{conj}.

Combining Theorems~\ref{thm0.1} and \ref{th:codimface}, we get the following.
\begin{CORO} For any untwisted affine Kac-Moody Lie algebra $\mathfrak{g}$, the inequalities~${I^P_{(w_1, w_2,v)}}$ in Conjecture~\ref{conj} give an irredundent and complete set of inequalities  determining the cone $\Gamma (\mathfrak{g})$.
\end{CORO}

To prove Theorem~\ref{th:codimface} we will use  (geometric)
Theorem~\ref{th:restCisom} below. Let us introduce some more notation.

Fix a standard parabolic subgroup~$P$ of~$G$. For~$w\in W^P$, we set~$$
\Delta^-(w)=\{\alpha\in \Delta\,:\,\ell(s_\alpha w)=\ell(w)-1\},
$$
and 
$$
\Delta^+(w)=\{\alpha\in \Delta\,:\,\ell(s_\alpha w)=\ell(w)+1\ {\rm and }\
s_\alpha w\in W^P\},
$$
where $s_\alpha$~is the (simple) reflection corresponding to the (simple) root $\alpha$. It is easy to see that for any~$\alpha\in \Delta^-(w), 
s_\alpha w\in W^P$.

Let~$B^-$ denote the Borel subgroup of~$G$ opposite to~$B$.
Consider the flag ind-variety $\mathcal{X} :=(\Gmin/B^-)^2\times \Gmin/B$ and
$\Pic^\Gmin(\mathcal{X})$ the group of~$G$\yh-linearized line bundles on~$\mathcal{X}$.
For~$\lambda \in \lh^*_\ZZ$, denote the line bundle~$\Li^-(\lambda) :=G \times^{B^-}\,\mathbb{C}_\lambda$ over~$G/B^-$ (resp.\ $\Li (\lambda) :=G \times^{B}\,\mathbb{C}_{-\lambda}$ over~$G/B$) associated to the principal $B^-$\yh-bundle $G \to G/B^-$ (resp. \ the $B$\yh-bundle $G \to G/B$) via the one-dimensional representation $\mathbb{C}_{\lambda}$ of~$B^-$ given by the character $e^\lambda$ uniquely extended to a character of~$B^-$  (resp. \ the representation $\mathbb{C}_{-\lambda}$ of~$B$ given by the character $e^{-\lambda}$). 

Fix  $(\lambda_1,\lambda_2, \mu)\in P_+^{3}$.
By an analogue of the Borel-Weil theorem for any Kac-Moody \tetrecitw{group $G$ (cf. \cite[Corollary~8.3.12]{Kumar:KacMoody}),  the  $G$\yh-linearized line bundle ${\Li :=\Li^-(\lambda_1)\boxtimes \Li^-(\lambda_2)\boxtimes \Li(\mu)}$}
 on~$\mathcal{X}$~is
such that the dimension of the space~$\Ho(\mathcal{X},\Li)^G$ of~$G$\yh-invariant
sections is the multiplicity of~$L(\mu)$ in~$L(\lambda_1)\otimes
L(\lambda_2)$ (cf. \cite[Proof of Theorem~3.2]{BrownKumar}). From this  we see that~$\Gamma(\lg)$~is a convex subset of~$P_{+, \QQ}^{3}$.

Fix $(w_1,w_2, v)\in (W^P)^{3}$ as in
Theorem~\ref{th:codimface} and let~$L\supset T $ denote the standard Levi
subgroup of~$P$, where $T$  is the standard maximal torus of~$G$ with Lie algebra $\mathfrak{h}$.
The base point~$B/B$ in~$G/B$  is denoted by~$\uo$.
Similarly,  $\uo^-=B^-/B^-$.
Set~$$x_0=(w_1\inv \uo^-,w_2\inv \uo^-,v\inv \uo)\in \mathcal{X}.$$
For~$\alpha\in \Delta^+(w_1)$, we set~$$
x_{\alpha,1}=(w_1\inv s_\alpha \uo^-,w_2\inv \uo^-,v\inv \uo)\in \mathcal{X}.
$$
Similarly, we define $x_{\alpha,2}$ associated to~$\alpha\in
\Delta^+(w_2)$. 
For~$\alpha\in \Delta^-(v)$, we  set~$$
x_{\alpha,3}=(w_1\inv \uo^-,w_2\inv \uo^-,v\inv s_\alpha \uo)\in \mathcal{X}.
$$
For any~$(\alpha,i)$ as above, we denote by~$\bl_{\alpha,i}$ the
unique $T$\yh-stable curve in~$\mathcal{X}$ containing $x_0$ and $x_{\alpha,i}$;
then $\bl_{\alpha,i}\simeq\PP^1$ and $x_0$ and $x_{\alpha,i}$ are the
two $T$\yh-fixed points in~$\bl_{\alpha,i}$. Explicitly, 
$$\bl_{\alpha,1} = \left(w_1^{-1}P_\alpha^-\uo^-, w_2^{-1}\uo^-, v^{-1}\uo\right) \subset \mathcal{X},$$
where $P_\alpha^-$~is the minimal (opposite) parabolic subgroup containing $B^-$ and $s_\alpha$. Similarly, $\bl_{\alpha,2}$ and $\bl_{\alpha,3}$  can be described explicitly. 

Consider now
$$
C=Lw_1\inv \uo^-\times Lw_2\inv \uo^-\times Lv\inv \uo,
$$
acted on by~$L$ diagonally.

\begin{thm}\label{th:restCisom}
Let~$P$ and $(w_1,w_2, v)\in (W^P)^{3}$ be  as in
Theorem~\ref{th:codimface}. 

Fix $(\lambda_1,\lambda_2, \mu)\in
(\lh_\ZZ^*)^{3}$ such that~$$
\forall \alpha_j\not\in\Delta(P), \qquad 
\lambda_1(w_1x_{j})+ \lambda_2(w_2x_{j})-\mu(vx_j)=0.
$$ 
Let~$\Li :=\Li^-(\lambda_1)\boxtimes \Li^-(\lambda_2)\boxtimes \Li(\mu)$ denote the associated line bundle on~$\mathcal{X}$. We assume that,
for any~$i=1,2$ and $\alpha\in \Delta^+(w_i)$, 
the restriction of~$\Li$ to~$\bl_{\alpha,i}$~is
nonnegative. Similarly, we assume that for any~$\alpha\in \Delta^-(v)$
the restriction of~$\Li$ to~$\bl_{\alpha,3}$~is
nonnegative. 

Then, the restriction map induces an isomorphism:
$$
\Ho(\mathcal{X},\Li)^G\simeq \Ho(C,\Li)^L.
$$
\end{thm}

To prove Theorem~\ref{th:codimface}, we have to produce line bundles~$\Li$ on~$\mathcal{X}$ having nonzero $G$\yh-invariant sections and
satisfying the equalities $(I^j_{(w_1,w_2,v)})$. 
To do this we start with a line bundle~$\Mi$ on~$\mathcal{X}$ whose
  restriction $\Mi\restrict{C}$ admits an $L$\yh-invariant section~$\sigma$.
Now, we want to extend $\sigma$ to a regular $G$\yh-invariant section on~$\mathcal{X}$. 
The first step is to extend $\sigma$ to a rational $G$\yh-invariant
section. 
Even though  this rational section can have poles, we are able to kill them by
adding an explicit   line bundle~$\Li'$ to~$\Mi$. 
An informed reader will notice  that the strategy is similar to the one used by the second author in \cite{GITEigen}. 
Nevertheless, there are numerous  difficulties because of infinite-dimensional phenomena. 
For example, we have  no abstract construction of line bundles arising from
divisors; the order of a pole along a divisor is not so easy to define (and
even if it is defined, such an order could be infinite) etc. 
In this paper, we overcome these difficulties by making various constructions more explicit which extend to our infinite-dimensional situation.

\subsubsection*{Acknowledgements} The first author is supported by NSF grants.
The second author is supported by the French ANR project ANR-15-CE40-0012.

\section{Zariski's main theorem}

We recall a  consequence of the Zariski's main theorem for our later use. 

\begin{PROP}\label{prop:Zariski}
  Let $f\,:\,Y\longto Z$~be a proper birational morphism between two
  quasiprojective irreducible varieties. 
We assume that we have an open subset~$\tilde Y$ of~$Y$ such that~$f(Y\backslash \tilde Y)$  has 
codimension at least two in~$Z$ and that~$Z$~is normal. Let $\Li$~be a line bundle over~$Z$.

\tetrecitwi{Then,  $f^*\,:\,\Ho(Z,\Li)\to  \Ho(Y,f^*(\Li))$ and the restriction
map~$r: \Ho(Y,f^*(\Li))\to  \Ho(\tilde Y,f^*(\Li))$} are both  isomorphisms. 
\end{PROP}

\begin{proof}
 To prove that~$f^*$~is an isomorphism, use the proof of Zariski's main theorem as in \cite[Chap. III, Corollary~11.4]{Hart}.

To prove that~$r$~is an isomorphism, consider the following commutative diagram:

\begin{center}
 \begin{tikzpicture}
\matrix (m) [matrix of math nodes,row sep=.8cm,column sep=.8cm] {
\Ho(Z, \Li)&& \Ho(Z\backslash f(Y\backslash \tilde{Y}), \Li)\\
\Ho(Y,f^* \Li)&& \Ho(Y\backslash f^{-1}(f(Y\backslash \tilde{Y})), f^*\Li)\\
&\Ho(\tilde{Y}, f^*\Li).\\
};
\path [->]     (m-1-1) edge node[above]     {$\beta$}  node[below,inner sep=3pt]     {$\sim$}        (m-1-3);
 \path [->]     (m-1-1) edge node[left]      {$f^*$}    node[right]      {$\wr$}    (m-2-1);
\path [->]     (m-1-3) edge node[right]      {$f^*$}     node[left,inner sep=2pt]      {$\wr$}    (m-2-3);
\path [->]     (m-2-1) edge node[below]     {$r_1$}        (m-2-3);
\path [->]     (m-2-1) edge node[above,inner sep=0.2pt]     {$\sim$}        (m-2-3);
\path [->]     (m-2-1) edge node[below]      {$r$}        (m-3-2);
\path [->]     (m-3-2) edge node[below]      {$r_2$}        (m-2-3);
\end{tikzpicture} 
\end{center}

   In the above diagram, $\beta$~is an isomorphism since $f(Y\backslash \tilde{Y})$~is of codimension~$\geq 2$ and $Z$~is normal. Thus, $r_1$~is an isomorphism. Further,  since $r_1$~is an isomorphism and $r$ and $r_2$ are injective, $r$~is an isomorphism as well. 
\end{proof}

\section{The span of the cone}

Before being interested in the faces of~$\Gamma(\lg)$, we describe the
span of it. 

 \begin{PROP} 
\label{prop:span}
The tensor cone $\Gamma(\lg)$ (which is, by definition, a rational cone) has nonempty interior in
  the following rational vector space~$$
E=E_\mathfrak{g}:=\{(\lambda_1,\lambda_2,\mu)\in (\lh_\QQ^*)^3\,:\,\lambda_1+\lambda_2-\mu\in\Span_\QQ(\Delta)\}.
$$
Observe that $E$~has dimension~$2\dim \lh + \sharp \Delta.$
\end{PROP}

\begin{proof}
If $(\lambda_1,\lambda_2,\mu)\in\Gamma(\lg)$ then
some integral multiple $N(\lambda_1+\lambda_2-\mu)$ belongs to the root lattice. 
Hence, 
\begin{equation}
  \label{eqn3.0}
\Gamma(\lg) \subset E.
 \end{equation}
Note that, for~$\lambda,\mu$ in~$P^+$, the point
\begin{equation}
  \label{eq:26}
  (\lambda,\mu,\lambda+\mu)\in \Gamma(\lg).
\end{equation}

We claim that for any simple root $\alpha_i\in \Delta$,
\begin{equation}
  \label{eq:22}
  (\rho,\rho,2\rho-\alpha_i)\in \Gamma(\lg),
\end{equation}
where $\rho\in \lh_{\mathbb{Z}}^*$~is any element satisfying $\rho(\alpha_i^\vee)=1$ for all the simple coroots $\alpha_i^\vee$. 
Indeed, fix a highest weight vector~$v_+$ in~$L(\rho)$ and a nonzero
$e_j$ (resp.\ $f_j$) in~$\lg_{\alpha_j}$ (resp.\ $\lg_{-\alpha_j}$) for any simple root $\alpha_j$ with $[e_j,f_j]=\alpha_j^\vee$, where $\lg_{\alpha}$ denotes the corresponding root space. 
Consider  the element in~$L(\rho)\otimes L(\rho)$:
$$
v=f_iv_+\otimes v_+-v_+\otimes f_i v_+.
$$
Clearly, $e_jv=0$ for any~$j\neq i$. Also,
$$
\begin{array}{ll}
  e_iv&=(e_if_i v_+)\otimes v_+-v_+\otimes(e_if_iv_+)\\
&=\alpha_i^\vee v_+\otimes v_+-v_+\otimes\alpha_i^\vee v_+\\
&=0.
\end{array}
$$
It follows that~$v$~is a highest weight vector. But its weight is
$2\rho-\alpha_i$, proving \eqref{eq:22}. Combined
with \eqref{eq:26}, we get 
\begin{equation}
  \label{eq:27}
  (0,0,\alpha_i)\in \langle \Gamma(\lg)\rangle, \qquad\forall \alpha_i\in\Delta,
\end{equation}
where $ \langle \Gamma(\lg)\rangle$~is the $\mathbb{Q}$\yh-span of~$ \Gamma(\lg)$ in~$ (\lh_\QQ^*)^3$ . 
Now, by \eqref{eq:26} and \eqref{eq:27}, $\Gamma(\lg)$ spans $E$.
\end{proof}

\section{On some translated Richardson varieties}

Fix a standard parabolic subgroup~$P$ of~$G$  with Levi subgroup~$L \supset T$, where $T$~is the (standard) maximal torus of~$G$ with Lie algebra $\mathfrak{h}$. For~$w\in W^P$, let~$$X^w_P:= \overline{B^-wP/P}\subset X_P\,\,\,\text{and}\,\,X_w^P:= \overline{BwP/P}\subset X_P$$
be respectively the {\it opposite Schubert variety} and the   {\it  Schubert variety} associated to~$w$.

\begin{prop}\label{prop:interSchub}
  Fix $(w, v)\in (W^P)^{2}$ and $\alpha\in\Delta$ such that
  \begin{enumerate}
  \item $v\geq w$;
  \item $\bar v:=s_\alpha v\leq v$;
  \item $\bar v\not\geq w$.
  \end{enumerate}

  Then, $X_v^w(P) :=X_v^P\cap X^w_P=s_\alpha X^P_{\bar v}\cap X^w_P$.
  
  In particular,   $\cX^P_{\bar v}\cap s_\alpha\cX^w_P$ is nonempty.
\end{prop}

\begin{proof}
  The inclusion $s_\alpha X^P_{\bar v}\cap X^w_P\subset X_v^w(P)$ is clear. 
  Moreover, $X_v^w(P)$ is an irreducible closed subvariety of $X_v^P$ of dimension
  $\ell(v)-\ell(w)$ (cf. \cite[Proposition~6.6]{Kumar:positivity} and use the surjectivity of $X_v^w(B)$ onto $X_v^w(P)$).   
   Since   $s_\alpha X^P_{\bar v}\cap X^w_P$ is closed
  in $X_v^P$, it is sufficient to prove that
  \begin{equation}
    \label{eq:dimXsvw}
    \dim(X^P_{\bar v}\cap s_\alpha X^w_P)=\ell(v)-\ell(w).
  \end{equation}
  Consider the incidence variety ${\mathcal Y}$ with projection on the second factor:
    $$
\pi\,:\,{\mathcal Y}:=\{(x,g\uo^-)\in X_v^P\times G/B^-\,|\, x\in gX^w_P\}\longto G/B^-.
$$
Set $\bar {\mathcal Y}:={\mathcal Y}\cap (X^P_{\bar v}\times G/B^-)$ and
$\bar \pi$ the restriction of $\pi$ to $\bar {\mathcal Y}$.\\
Observe first that $ {\mathcal Y}$ and $\bar {\mathcal Y}$ are
respectively $P_\alpha$-stable and $B$-stable closed subsets of $X_v^P\times G/B^-$ and 
$X^P_{\bar v}\times G/B^-$ respectively 
and that $\pi$ and
$\bar\pi$ are equivariant, where $P_\alpha$ is the minimal parabolic subgroup of $G$ containing $B$ and $s_\alpha$.
Moreover, $G\times_B  {\mathcal Y}\simeq
G\times_P(\overline{Pv\inv\uo}\times \overline{Pw\inv\uo^-}))$ is
ind-irreducible (i.e., admits a filtration by finite-dimensional irreducible closed subsets). Hence ${\mathcal Y}$ is ind-irreducible. Similarly, $\bar
{\mathcal Y}$ is ind-irreducible.
Since $X_v^P$ is projective, $\pi$ is proper. Similarly, $\bar \pi$ is proper. Hence
their images are closed.

Since $v\geq w$, $\uo^-\in\Image\,\pi$. Hence, $\pi$ is surjective.

Since $\bar v\not\geq w$, $\uo^-\not\in\Image\,\bar\pi$. Hence, $\Image(\bar
\pi)\subset \bigcup_{\beta\in\Delta}\overline{Bs_\beta\uo^-}$.
But, $(\bar vP/P,s_\alpha\uo^-)\in \bar {\mathcal Y}$. Indeed $vP/P\in
X^w_P$ and $s_\alpha vP/P=\bar vP/P\in s_\alpha X^w_P$.
Hence, $\Image(\bar\pi)$ contains $Bs_\alpha\uo^-$ by $B$-equivariance.
Since $\Image(\bar\pi)$ is closed and irreducible, we get
\begin{equation}\label{image}
\Image(\bar\pi)=\overline{Bs_\alpha\uo^-}.
\end{equation}
We now restrict ${\mathcal Y}$ over $P_\alpha\uo^-$. Consider the
action map
$$
\rho\,:\,P_\alpha\longto \Aut(X_v^P).
$$
The image $P_v$ of $\rho$ is a finite dimensional connected algebraic group
of semi-simple rank one. 
Consider
$$
\Yc^\circ_{\rm red}:=\{(x,p)\in X_v^P\times P_v\,|\, p\inv x\in X^w_P\}
$$
with its two projections $p_1$ and $p_2$ on $X_v^P$ and $P_v$ respectively. Moreover, $p_2$ is 
$P_v$-equivariant and hence surjective. Further, $\Yc^\circ_{\rm red}$ is irreducible since so is $p_2\inv(e)$ and $P_v$. 

Note that $p_2\inv(e)\simeq X_v^w(P)$ is of dimension $\ell(v)-\ell(w)$.
Hence,
\begin{equation}
  \label{eq:dimY0}
  \dim(\Yc^\circ_{\rm red})=\dim(P_v)+\ell(v)-\ell(w).
\end{equation}
We already observed that $(vP/P,e)\in \Yc^\circ_{\rm red}$.
Since $P_v\cdot v$ is dense in $X_v^P$ and $p_1$ is $P_v$-equivariant, we
conclude that $p_1$ is dominant.
Then, still using the $P_v$-equivariance, we get
\begin{equation}
  \label{eq:dimfiberp1}
  \dim(p_1\inv(vP/P))=\dim(\Yc^\circ_{\rm red})-\ell(v)=  
  \dim(P_v)-\ell(w),\,\text{by equation \eqref{eq:dimY0}} 
\end{equation}
Set
$$
{\bar \Yc}^\circ_{\rm red}=\Yc^\circ_{\rm red}\cap (X^P_{\bar v}\times P_v).
$$
As observed above,  ${\bar \Yc}\neq \Yc$. Similarly ${\bar \Yc}^\circ_{\rm red}\neq
\Yc^\circ_{\rm red}$.
Since $X^P_{\bar v}\times P_v$ is a hypersurface in $X_{ v}^P\times P_v$,
we deduce that
\begin{equation}
  \label{eq:dimXred}
  \dim({\bar \Yc}^\circ_{\rm red})=\dim(\Yc^\circ_{\rm red})-1=\dim(P_v)+\ell(v)-\ell(w)-1,\,\text{by equation \eqref{eq:dimY0}}.
\end{equation}
Write now $P_\alpha=BU_{-\alpha}(\CC)\sqcup Bs_\alpha$ and
$P_v=\rho(B)U_{-\alpha}(\CC)\sqcup \rho(B)s_\alpha$, where $U_{-\alpha}(\CC)$ is the root subgroup of $G$ 
corresponding to the negative root $-\alpha$.
Let $\bar p_2$ denote the restriction of $p_2$ to ${\bar
  \Yc}^\circ_{\rm red}$, which is $\rho(B)$-equivariant.
The discussion about the image of $\bar \pi$ at the beginning of the
proof implies that $\Image(\bar p_2)=\rho(B)s_\alpha$.
Hence,
\begin{equation}
  \label{eq:dimexpected}
  \dim({\bar p_2)}\inv(s_\alpha)=\dim(X^P_{\bar v}\cap s_\alpha X^w_P)=\dim({\bar \Yc}^\circ_{\rm red})-\dim(\rho(B)s_\alpha).
\end{equation}
The expected equality~\eqref{eq:dimXsvw} follows from
\eqref{eq:dimexpected}, \eqref{eq:dimXred} and the identity $\dim(\rho(B)s_\alpha)=\dim(P_v)-1$.

To prove the `In particular' statement of the proposition, observe that 
$$\cX^P_{\bar v}\cap s_\alpha\cX^w_P= (X^P_{\bar v}\cap s_\alpha\cX^w_P)\cap (\cX^P_{\bar v}\cap s_\alpha X^w_P)
$$
and the last two open subsets are nonempty in the irreducible variety $X^P_{\bar v}\cap s_\alpha X^w_P$
(by the first part of the lemma).
\end{proof}
\begin{coro}\label{prop:interSchub2}
Fix $(w, v)\in (W^P)^{2}$ and $\alpha\in\Delta$ such that
  \begin{enumerate}
  \item $v\geq w$;
  \item $\bar w:=s_\alpha w\geq w\in W^P$;
  \item $v\not\geq \bar w$.
  \end{enumerate}

  Then, $X_v^w(P)=X_{v}^P\cap s_\alpha X^{\bar w}_P$.
\end{coro}
\begin{proof} Set $\bar v=s_\alpha v$.
  Then, by [12, Lemma 1.3.18 and Corollary 1.3.19], 
   $\bar v\in W^P$;
   $\bar v\geq \bar w$;
  $v< \bar v$; and
  $v\not\geq\bar w$.
  Hence, we can apply Proposition~\ref{prop:interSchub} to the pair
 $(\bar w,\bar v)$ to get
 \begin{equation}
   \label{eq:2}
   X^P_{\bar v}\cap X^{\bar w}_P=s_\alpha X_{v}^P\cap  X^{\bar w}_P.
 \end{equation}

 In particular, we have
 $$
 \begin{array}{rcl}
   \dim(X_v^w(P))&=&\ell(v)-\ell(w)\\
              &=&\ell(\bar v)-\ell(\bar w)\\
              &=&\dim(X_{\bar v}^{\bar w}(P))\\
              &=&\dim(s_\alpha X_{v}^P\cap  X^{\bar w}_P)\\
    &=&\dim(X_{v}^P\cap  s_\alpha X^{\bar w}_P).\\
 \end{array}
 $$
 Moreover, by [12, Theorem 5.1.3(d)], $s_\alpha X^{\bar w}_P\subset X^w_P$. Hence,
$X_{v}^P\cap s_\alpha X^{\bar w}_P\subset X_v^w(P)$.
 This proves the corollary. 
\end{proof}

\section{Construction of line bundles}\label{section5}

Consider a subvariety~$Z\subset {\mathcal X}$. 
If $G$ and so ${\mathcal X}$~is finite-dimensional, $Z$ can be
realized as the zero set of a section of some  line bundle on~${\mathcal X}$ 
if and only if $Z$~has codimension one. 
If $G$~is not finite-dimensional, then ${\mathcal X}$~is only an ind-variety and the
codimension is not so easy to define. 
Moreover, even if there exists a filtration ${\mathcal X}=\bigcup_n
{\mathcal X}_n$ by finite-dimensional closed subvarieties such that~$Z\cap {\mathcal X}_n$~has codimension one in~${\mathcal X}_n$, $Z$~is
not necessarily the zero locus of a section of some line bundle on~${\mathcal X}$. 

Nevertheless, if $Z=F_{\alpha,i}$ (resp. $Z=E_{\bar w_1, \bar w_2,
  \bar v}$)    as defined by
Formula~\eqref{eq:defF} (resp. \ \eqref{eq:defEbar}) below, we prove in this
section that~$Z$~is the zero locus of a section of some line bundle.

\subsection{First divisors}\label{section4.1}

Fix once and  for all  fundamental weights~$\varpi_{\alpha_0},\dots,
\varpi_{\alpha_l}$ in~$\lh_\ZZ^*$ such that~$\langle
\varpi_{\alpha_i},\alpha_j^\vee\rangle=\delta_i^j$.

Let $M$~be a $\lg$\yh-module such that, under the action of~$\lh$, $M$
decomposes as~$\bigoplus_{\mu\in\lh^*} M_\mu$ with finite-dimensional weight spaces~$M_\mu$. Set~$M^\vee=\bigoplus_\mu M_\mu^*$: it is a $\lg$\yh-submodule of
the full dual space~$M^*$. 

Recall that~${\mathcal X}=(G/B^-)^2\times G/B$ and
$\uo^\pm=B^\pm/B^\pm$. 
Consider, for~$\alpha\in \Delta$ and $i=1,2$, 
\begin{equation}
  \label{eq:defF}
F_{\alpha,i}=\{(x_1,x_2,g\uo)\in \mathcal{X}\,:\,g\inv x_i\in
\overline{Bs_\alpha \uo^-}\}
\end{equation}
with the reduced ind-scheme structure. It is easy to see that~$F_{\alpha,i}$~is ind-irreducible.
Let~$p_1,p_2$ and $p_3$ denote the projections from~$\mathcal{X}$ to the
corresponding factor. 
Set, for~$i=1,2$ and $\alpha\in \Delta$, 
$$
\Mi_{\alpha,i}=p_i^*(\Li^-_{\varpi_\alpha})\otimes p_3^*(\Li_{\varpi_\alpha}).
$$

\begin{lemma}
  \label{lem:linebundlebord}
The space~$\Ho(\mathcal{X}, \Mi_{\alpha,i})$ contains a unique (up to scalar
multiples)  nonzero $G$\yh-invariant section~$\sigma = \sigma_{\alpha, i}$. Moreover,
scheme-theoretically, 
$$
F_{\alpha,i}=\{x\in \mathcal{X}\,:\,\sigma(x)=0\}.
$$ 
\end{lemma}

\begin{proof}
Our construction of~$\Mi_{\alpha,i}$ and $\sigma_{\alpha, i}$~is completely explicit.

By the analogue of the Borel-Weil theorem for Kac-Moody groups (cf. \cite[Corollary~8.3.12]{Kumar:KacMoody}), we have (cf. \cite[Proof of Theorem~3.2]{BrownKumar}):
\begin{equation}\label{eqn3.1.1}
\Ho(\mathcal{X}, \Mi_{\alpha,i})\simeq \Hom_\CC(L({\varpi_\alpha})^\vee\otimes L({\varpi_\alpha}),\CC).
\end{equation}
Observe that  
\begin{equation}\label{eqn3.1.2}
 \Hom_\CC(L({\varpi_\alpha})^\vee\otimes L({\varpi_\alpha}),\CC)\simeq \Hom_\CC(L({\varpi_\alpha})^\vee,  L({\varpi_\alpha})^*),
\end{equation}
since  $\Hom_\CC(V\otimes W, \CC)\simeq  \Hom_\CC(V, W^*)$ for any~$\CC$\yh-vector spaces~$V$ and $W$. From the Equations~\eqref{eqn3.1.1} and 
 \eqref{eqn3.1.2} it is easy to see that~$\Ho(\mathcal{X}, \Mi_{\alpha,i})^G$~is one-dimensional spanned by the inclusion of~$L({\varpi_\alpha})^\vee$ in~$L({\varpi_\alpha})^*$ under the identifications  \eqref{eqn3.1.1} and 
 \eqref{eqn3.1.2}. 
We now identify  the zero locus of nonzero $\sigma \in \Ho(\mathcal{X}, \Mi_{\alpha,i})^G$:

Consider the isomorphism 
$$\psi: G \times ^{B^-}\, G/B \simeq G/B^-\times G/B, \,\,\,[g,h\uo] \mapsto (g\uo^-, gh\uo), \,\,\text{for} \,\,g,h\in G,
$$
where $[g,h\uo] $ denotes the $B^-$\yh-orbit of~$(g,h\uo)$. Consider the $B^-$\yh-equivariant line bundle ${\CC_{\varpi_\alpha}\otimes \Li_{\varpi_\alpha}}$ over~$G/B$, where $\CC_{\varpi_\alpha}$ denotes the trivial line bundle over~$G/B$ with the $B^-$\yh-action given by the character $e^{\varpi_\alpha}$. It is easy to see that 
\begin{equation}\label{eqn3.1.3}
 \psi^*(\Li^-_{\varpi_\alpha}\boxtimes \Li_{\varpi_\alpha}) = G \times ^{B^-}\,(\CC_{\varpi_\alpha}\otimes \Li_{\varpi_\alpha}).    
\end{equation}
Let $v_-$~be  a fixed nonzero vector of~$\CC_{-\varpi_\alpha}$. Consider the section~$\sigma_o$ of~$\Li_{\varpi_\alpha}$ over~$G/B$ given by
\begin{equation}\label{eqn3.1.4}
 \sigma_o (g\uo) = [g, v_+^*(gv_+)v_-], \,\,\,\text{for}\,\, g\in G,
\end{equation}
where $v_+$~is a nonzero highest weight vector of~$L(\varpi_\alpha)$ and $v_+^* \in L(\varpi_\alpha)^*$~is given by~$$v_+^* (v_+)=1\,\text{and}\, v_+^*(v)=0,\,\text{for any weight vector~$v$  of~$L(\varpi_\alpha)$ of weight~$\neq \varpi_\alpha$}.$$
By the definition of~$\sigma_o$, it is a character of~$B^-$ of weight~$\yh- \varpi_\alpha$ and hence $1\otimes \sigma_o$ thought of as a section of~$\CC_{\varpi_\alpha}\otimes \Li_{\varpi_\alpha}$~is $B^-$\yh-invariant. Thus, it canonically gives rise to a $G$\yh-invariant section~$\hat{\sigma}_o$ of~$G \times ^{B^-}\,(\CC_{\varpi_\alpha}\otimes \Li_{\varpi_\alpha}).$

We next claim that the zero set~$Z(\sigma_o)$ of~$\sigma_o$~is given by 
\begin{equation}\label{eqn3.1.5}
 Z(\sigma_o)=\overline{B^-s_\alpha \uo}\subset G/B.
\end{equation}
By the definition of~$\sigma_o$, $ Z(\sigma_o)$~is left $B^-$\yh-stable (since $v_+^*\in  L(\varpi_\alpha)^*$~is an eigenvector for the action of~$B^-$).
Take $w\in W$. Then,
\begin{align*} w\uo\in Z(\sigma_o) &\Leftrightarrow v_+^*(wv_+)=0\\
 &\Leftrightarrow w\varpi_\alpha \neq \varpi_\alpha\\
 &\Leftrightarrow w\notin \langle s_\beta\rangle_{\beta \in \Delta \backslash \{\alpha\}},\,\,\text{by \cite[Proposition~1.4.2 (a)]{Kumar:KacMoody}}\\
 &\Leftrightarrow w\geq s_\alpha,
\end{align*}
where $\langle s_\beta\rangle \subset W$ denotes the subgroup
generated by the elements $s_\beta$. This proves the Equation~\eqref{eqn3.1.5} by the Birkhoff decomposition \cite[Theorem~6.2.8]{Kumar:KacMoody}. Thus, the zero set~$Z(\hat{\sigma}_o)$ of~$\hat{\sigma}_o$~is given by:
$$Z(\hat{\sigma}_o) = G\times^{B^-}\,\left(\overline{B^-s_\alpha \uo}\right).$$
Moreover,
$$\psi \left(G\times^{B^-}\,\left(\overline{B^-s_\alpha \uo}\right)\right)=\{(x, g\uo) \in G/B^-\times G/B: g^{-1}x\in \overline{Bs_\alpha \uo^-}\}.$$
From this we obtain that~$Z(\sigma)=F_{\alpha,i}$ set theoretically.

To prove that~$Z(\sigma)=F_{\alpha,i}$ scheme-theoretically, it suffices to show that~$Z(\sigma_o)$ (which is set theoretically $X^{s_\alpha}= \overline{B^-s_\alpha \uo}\subset G/B$) is reduced. 

For any~$v\in W$, consider  $Z(\sigma_o)\cap X_v = X^{s_\alpha}\cap X_v$, which is an irreducible subset of codimension one in~$X_v$. The Chern class of the line bundle~$ {\Li_{\varpi_\alpha}}\restrict{X_v}$~is the Schubert class ${\eps^{s_\alpha}\in \Hd^2(X_v, \ZZ)}$. If  $Z(\sigma_o)\cap X_v$ were not reduced, say
$$ Z(\sigma_o)\cap X_v = d (X^{s_\alpha}\cap X_v)\,\,\,\text{(scheme-theoretically) for some~$d>1$},$$
then $\frac{1}{d} \eps^{s_\alpha}\in \Hd^2(X_v, \ZZ)$, which is a contradiction. Hence $d=1$, proving that~$ Z(\sigma_o)\cap X_v $~is reduced for any~$v\in W$. Thus, $ Z(\sigma_o)$~is reduced, proving the lemma.
\end{proof}

\subsection{Subvarieties of~$\mathcal{X}$ from Schubert varieties}

Fix a standard parabolic subgroup~$P$ of~$G$  with Levi subgroup~$L \supset T$, where $T$~is the (standard) maximal torus of~$G$ with Lie algebra $\mathfrak{h}$.
For any triple $(w_1, w_2, v)\in (W^P)^3$, set~$$
{\bar C}^+_{w_1,w_2,v}=\overline{Pw_1\inv \uo^-}\times
\overline{Pw_2\inv \uo^-}\times \overline{Pv\inv \uo}\subset \mathcal{X},
$$
and 
\begin{equation}
  \label{eq:defE}
E_{w_1,w_2,v}=G. {\bar C}^+_{w_1,w_2,v}\subset \mathcal{X}\,\,\,\text{under the diagonal action of~$G$}.
 \end{equation}

\begin{lemma}\label{Eisclosed} For any triple $(w_1, w_2, v)\in (W^P)^3$,
the set~$E_{w_1,w_2,v}$~is closed and ind-irreducible in~${\mathcal X}$.  
\end{lemma}

\begin{proof}
Since $G$ and ${\bar C}^+_{w_1,w_2,v}$ are ind-irreducible (see
\cite[before Lemma~3]{R:KM1} and the argument in the proof of Lemma~5.4), so
is $E_{w_1,w_2,v}$. 
Note that 
\begin{equation}\label{neweqn2.1}
E_{w_1,w_2,v}=\{(g_1\uo^-,g_2\uo^-,g_3\uo)\in{\mathcal X}\,:\, g_1X^{w_1}_P\cap g_2X^{w_2}_P\cap g_3X_{v}^P\neq\emptyset\}. 
\end{equation}
Observe that $E_{w_1,w_2,v} =\mathcal{X}$ if $\varepsilon_P^v$ occurs in $\varepsilon_P^{w_1}\cdot \varepsilon_P^{w_2}$ with nonzero coefficient (cf. \cite[Porposition~3.5]{BrownKumar}).

  By the following isomorphism
$$
G\times_B(G/B^-)^2\longto {\mathcal X},\,[g, x]\mapsto (gx,gB/B),
$$ 
it is sufficient to prove that~$$
\tilde E=\{(g_1\uo^-,g_2\uo^-)\,:\, g_1X^{w_1}_P\cap g_2X^{w_2}_P\cap X_{v}^P\neq\emptyset\}
$$
is closed in~$\mathcal{X}_\bs:= (G/B^-)^2\simeq  (G/B^-)^2\times \uo$.
Consider 
$$
\pi_\bs\,:\,\Chi_\bs\longto \mathcal{X}_\bs,
$$
where $$\Chi_\bs:= \{(y, g_1\uo^-, g_2\uo^-, \uo)\in G/P\times \mathcal{X}: y\in  g_1X^{w_1}_P\cap g_2X^{w_2}_P\cap X_{v}^P\}$$
and $\pi_\bs$~is the projection to the last three factors. 
Note that~$\tilde E$~is the image of~$\Chi_\bs$ and $\Chi_\bs$~is closed in $X_v^P\times\mathcal{X}_s$. 
Consider a filtration ${\mathcal X}_\bs=\bigcup_n {\mathcal X}_\bs^n$ by
closed finite-dimensional subvarieties. 
Then, $\pi_\bs\inv({\mathcal X}_\bs^n)$~is closed in~$X_v^P\times {\mathcal
  X}_s^n$. Since $X_v^P$~is projective, it follows that~$\pi_\bs(\pi_\bs\inv({\mathcal X}_\bs^n))$~is closed in~${\mathcal X}_\bs^n$.
This concludes the proof since $\pi_\bs(\pi_\bs\inv({\mathcal
  X}_\bs^n))=\tilde E\cap {\mathcal X}_\bs^n$.
\end{proof}

For~$w\in W^P$, we set~$\cX^w_P=B^-wP/P$ and $\cX_w^P=BwP/P$.
Consider,  for any triple $(w_1, w_2, v)\in (W^P)^3$,
\begin{align}\label{eqn4.1}
  \Chi&:=\{(gP/P,x)\in G/P\times \mathcal{X}\,:\, g\inv x\in \bar C^+\}\notag\\ 
&=\{(y,g_1\uo^-,g_2\uo^-,g_3\uo)\in G/P\times \mathcal{X}\,:\,y\in
  g_1X_P^{w_1}\cap g_2X^{w_2}_P\cap g_3X_{v}^P\}
\end{align}
and 
$$
  \cChi :=\{(y,g_1\uo^-,g_2\uo^-,g_3\uo)\in G/P\times \mathcal{X}\,:\,y\in
  g_1\cX_P^{w_1}\cap g_2\cX^{w_2}_P\cap g_3\cX_{v}^P\},$$
where $ \bar{C}^+ = \bar C^+_{w_1, w_2, v}$. Observe that~$\Chi$~is closed in~$G/P\times \mathcal{X}$ and it  is  ind-irreducible since $\Chi = G\cdot (P/P,  \bar C^+).$

Consider also the set~$\cChi^+$ of points~$(y,g_1\uo^-,g_2\uo^-,g_3\uo)\in\cChi$ such that the linear
  map~$$
\textstyle\Tau_y (g_3\cX_v^P) \longto\frac{\Tau_y (G/P)}{\Tau_y (g_1\cX^{w_1}_P)}\oplus\frac{\Tau_y (G/P)}{\Tau_y (g_2\cX^{w_2}_P)}
$$ 
is injective, i.e.,
$$\Tau_y (g_1\cX^{w_1}_P)\cap \Tau_y (g_2\cX^{w_2}_P)\cap  \Tau_y  (g_3\cX_v^P) =(0),$$
where $\Tau$ denotes the Zariski tangent space. 

For~$v\in W^P$, we denote $v'\to v$ if $v'\in W^P$, $\ell(v')=\ell(v)-1$ and $v'\leq v$. 

\begin{lemma}
  \label{lem:Chiopen}
The subsets~$\cChi$ and $\cChi^+$ are open in~$\Chi$  for any triple $(w_1, w_2, v)\in (W^P)^3$. 

 In the definition of~$\cChi$ and $\cChi^+$  if we replace  $\cX^{w_i}_P$ (for any~$i=1,2$) by any~$B^-$\yh-stable open subset of~$\cX^{w_i}_P \cup \,(\bigcup_{w_i\to w_i'\in W^P}\,\cX^{w_i'}_P)$ and  $ \cX_v^P$ by any~$B$\yh-stable open subset of\break ${\cX_v^P\cup\,(\bigcup_{v'\to v, v'\in W^P}\, \cX_{v'}^P)}$, then  the lemma still remains true.
 
 \end{lemma}

\begin{proof}
  Consider the projection~$$\pi: G^{\times 4} \to G/P \times \mathcal{X}, \,\,\,(g, g_1, g_2, g_3) \mapsto (gP/P, g_1\uo^-, g_2\uo^-, g_3\uo),$$
and define $\tilde{\Chi}:= \pi^{-1}(\Chi)$ and  $\tilde{\cChi}:= \pi^{-1}(\cChi)$. Then,
\begin{equation} \label{eqn 5.1}
\tilde{\Chi}= \{(g, g_1, g_2, g_3)\in G^{\times 4}: gP/P\in g_1X_P^{w_1}\cap g_2 X^{w_2}_P\cap g_3 X_{v}^P\},
\end{equation}
and
\begin{equation}\label{eqn5.2}
\tilde{\cChi}= \{(g, g_1, g_2, g_3)\in G^{\times 4}: gP/P\in g_1\cX_P^{w_1}\cap g_2\cX^{w_2}_P\cap g_3\cX_{v}^P\}.
\end{equation}
Define the morphism~$$\beta: \tilde{\Chi} \to X_P^{w_1}\times X^{w_2}_P\times X_{v}^P,\,\,\,(g, g_1, g_2, g_3) \mapsto (g_1^{-1}gP/P, g_2^{-1} gP/P, g_3^{-1}gP/P).$$
Then, 
$$\tilde{\cChi} = \beta^{-1} \left(\cX_P^{w_1}\times \cX^{w_2}_P\times \cX_{v}^P\right)$$
and hence $\tilde{\cChi}$~is open in~$\tilde{\Chi}$. Thus, $\pi$ being an open map, ${\cChi}$~is open in~${\Chi}$. 

We now prove that 
\begin{center}
$\cChi^+$~is open in~$\cChi$ (and hence in~$\Chi$).
\end{center}
By the Equation~\eqref{eqn5.2}
\begin{equation}
\label{eqn5.4}
\begin{array}{l@{\,}l@{\,}l}
\pi^{-1}(\cChi)&= \tilde{\cChi} &= \{(g, g_1, g_2, g_3)\in G^{\times      4}: \\ 
&&g^{-1}g_1\in Pw_1^{-1}U^-,\,  g^{-1}g_2\in Pw_2^{-1}U^-,\,
 g^{-1}g_3\in Pv^{-1}U\}, 
\end{array}
\end{equation}
and 
\begin{equation}\label{eqn5.5}
  \begin{array}{l@{\,}l}
\pi^{-1}(\cChi^+)= & \{(g, g_1, g_2, g_3)\in \pi^{-1}(\cChi) : \\
&\Tau_{\dot{e}} (g^{-1}g_1\cX_P^{w_1})\cap \Tau_{\dot{e}} (g^{-1}g_2\cX_P^{w_2})\cap
\Tau_{\dot{e}}(g^{-1}g_3\cX_v^P)= (0)\},
 \end{array}
\end{equation}
where $\dot{e}:=P/P\in G/P$.
Consider the morphism~$$\tilde{\cbeta}:\tilde{\cChi} \to \tilde{\cX}_{w_1, w_2, v}:= \tilde{\cX}_P^{w_1}\times \tilde{\cX}^{w_2}_P\times \tilde{\cX}_{v}^P,\,\,\,(g, g_1, g_2, g_3) \mapsto (g_1^{-1}g, g_2^{-1} g, g_3^{-1}g),$$
where $ \tilde{\cX}_P^{w_i}:=B^-w_iP\subset G$ and similarly $\tilde{\cX}_{v}^P:=BvP\subset G$. Define the finite rank vector bundle~$\mathcal{E}_i$ over~$ \tilde{\cX}_P^{w_i}$ ($i=1,2$) by~$$\bigcup_{h_i\in \tilde{\cX}_P^{w_i}}\, \Tau_{\dot{e}}(G/P)/\Tau_{\dot{e}} (h_i^{-1}\cX^{w_i}_P)\to  \tilde{\cX}_P^{w_i}, $$ 
and similarly  the finite rank vector bundle~$\mathcal{E}_3$ over~$ \tilde{\cX}^P_{v}$  by~$$\bigcup_{h\in \tilde{\cX}^P_{v}}\, \Tau_{\dot{e}} (h^{-1}\cX_{v}^P)\to  \tilde{\cX}_P^{v},$$
and  a morphism over~$ \tilde{\cX}_{w_1, w_2, v}$:
$$\varphi: \pi_3^*(\mathcal{E}_3)\to  \pi_1^*(\mathcal{E}_1)\oplus \pi_2^*(\mathcal{E}_2)$$
induced by the canonical inclusion of~$ \Tau_{\dot{e}} (h^{-1}\cX_{v}^P)\hookrightarrow \Tau_{\dot{e}}(G/P)$, where $\pi_i$~is the projection from~$ \tilde{\cX}_{w_1, w_2, v}$ to the $i$\yh-th factor.. The set of points~$Z\subset  \tilde{\cX}_{w_1, w_2, v}$ where $\varphi$~is injective is clearly open.
But, it is easy to see that~$(\tilde{\cbeta})^{-1}(Z)= \pi^{-1}(\cChi^+)$, and hence $ \pi^{-1}(\cChi^+)$~is open in~$\tilde{\cChi}$ and thus $\cChi^+$~is open in~$\cChi$. This proves the first part of the lemma.

The proof for the second (stronger) part of  the lemma  is identical.
\end{proof}

\subsection{Divisors from Schubert varieties}
In the remaining part of this Section \ref{section5}, $P$~is still a standard parabolic subgroup (and not necessarily maximal). We fix $(w_1,w_2, v)\in
(W^P)^{3}$  such that~$\eps_P^v$ occurs with coefficient~$1$ in 
the deformed product
$$\eps_P^{w_1}\odot_0 \eps_P^{w_2}
\in \bigl(\Hd^*(X_P,\ZZ), \odot_0\bigr).$$
In particular, $w_1, w_2\leq v$.
\subsection{The setting} \label{subsection5.4}
By a {\it descent} we mean a pair $(\alpha, i)\in \Delta \times \{1, 2, 3\}$ such that 
\begin{itemize}
\item $i=1, 2$ and $\alpha\in \Delta^+(w_i)$, i.e., $s_\alpha w_i\in W^P$ and $\ell(s_\alpha w_i)= \ell(w_i)+1$. 

In this case we set $\bar{w}_i= s_\alpha w_i, \bar{w}_{3-i}= w_{3-i}$ and $\bar{v}=v$. Or
\item $i=3$ and $\alpha\in \Delta^-(v)$, i.e., $\ell(s_\alpha v)= \ell(v)-1$. 

In this case we set $\bar{w}_i= w_i$ (for $i=1, 2$)  and $\bar{v}=s_\alpha v$. 
\end{itemize}
The set of all descents is denoted by $\mcD$.

We now aim to prove that

\begin{equation}
E_{\alpha, i}:= E_{\bar{w}_1,\bar{w}_2,\bar{v}} \,\,\,\text{defined by
  equation \eqref{eq:defE}}
\label{eq:defEbar}
\end{equation}

is the zero set of a section of some $G$-linearized line bundle on $\mathcal{X}$. To do this, we consider three different situations of descents:

\vskip1ex
{\bf Type $A$:} Descent $(\alpha, i)$ with $i=1$ or $2$ such that  $v=\bar{v}\not\geq \bar{w}_i$.

\vskip1ex
{\bf Type $B$:} Descent $(\alpha, i)$ with $i=3$ such that $\bar{v}\not\geq w_j=\bar{w}_j$ for at least one $j=1, 2$.

\vskip1ex
{\bf Type $C$:} Descent $(\alpha, i)$ with $1\leq i\leq 3$ such that  the relation $\bar{v}\geq \bar{w}_j$ holds for both $j=1, 2$.

\subsection{Type $A$ descents} 
\begin{lemma} \label{typealemma} For a descent $(\alpha, i)$ of type $A$, 
$$E_{\alpha, i}:= E_{\bar{w}_1,\bar{w}_2,\bar{v}} = F_{\alpha, i},$$
where $ F_{\alpha, i}$ is defined by equation \eqref{eq:defF}.
\end{lemma}
\begin{proof}
Assume that~$i=2$. (The proof for $i=1$ is identical.) 
Recall from the Equation~\eqref{eqn4.1}:
$$
\Chi :=\{(y, g_1\uo^-,g_2\uo^-,g\uo)\in G/P \times \mcX\,:\,y\in
g_1X^{w_1}_P\cap g_2X^{w_2}_P\cap gX_v^P\},
$$
for the triple $(w_1,  w_2, v)$. Consider 
 its analogue for~$w_2$ replaced by~$\bar{w}_2:=s_\alpha w_2$:
$$
\Chi':=\Chi'_{\alpha,2}:=\{(y, g_1\uo^-,g_2\uo^-,g\uo)\in G/P \times \mcX\,:\,y\in
g_1X^{w_1}_P\cap g_2X^{\bar{w}_2}_P\cap gX_v^P\}, 
$$
and $\Chi'_{\alpha, i}$~has  a similar meaning, where we place $s_\alpha$ in the $i$\yh-th factor.   

Let $\eta\,:\, G/P \times \mcX\longto\mcX$~be the projection.
By Lemma~\ref{Eisclosed}, $\eta(\Chi')=E_{\alpha,2}$ (cf. the identity \eqref{neweqn2.1}) is
  closed in~$\mcX$ and ind-irreducible.
Define the open subset of~$\mcX$ :
$$
\cmcX:=\{(x_1,x_2,x)\in\mcX\,:\, (x_1,x)\in G.(\uo^-,\uo)\}.
$$
To prove that $\cmcX$ is open in $\mathcal{X}$, use the standard isomorphism $G\times_{B^-}G/B\simeq G/B^-\times G/B$.
Since $(\uo^-,s_\alpha\uo^-,\uo)\in \cmcX\cap F_{\alpha,2}$ and
$F_{\alpha,2}$~is ind-irreducible (cf. $\S$\ref{section4.1}), we have
\begin{equation}
  \label{eq:Fbar}
  \overline{\cmcX\cap F_{\alpha,2}}=F_{\alpha,2}.
\end{equation}
Since $w_1\leq v$, the Richardson
variety~$X_v^{w_1}(P):= X_v^P\cap X_P^{w_1}$~is nonempty.  Take $x \in X_v^{w_1}(P)$.
There exists $g\in G$ such that~$g\inv x\in X^{s_\alpha w_2}_P$. 
Then, $(\uo^-,g\uo^-,\uo)$ belongs to~$\cmcX\cap\eta(\Chi')$. 
Since $\eta(\Chi')$~is ind-irreducible, we deduce that
\begin{equation}
  \label{eq:Ebar}
  \overline{\cmcX\cap\eta(\Chi')}=\eta(\Chi').
\end{equation}
By \eqref{eq:Fbar} and \eqref{eq:Ebar}, it is sufficient to prove that
\begin{equation}
  \label{eq:Xcirc}
  \cmcX\cap\eta(\Chi')=\cmcX\cap F_{\alpha,2}.
\end{equation}

But $G\times_T G/B^-\longto \cmcX, [g:x]\longmapsto
(g\uo^-,gx,g\uo)$~is an isomorphism. Consider the intersection of~$\Chi$ with $G/P\times \uo^-\times G/B^-\times \uo$: 
$$
\Chi_{\bs\bs} =\{(x, g\uo^-)\in  X_v^{w_1}(P)\times G/B^-\,:\,x\in gX^{w_2}_P\}
$$
and
$$
{\Chi'}_{\bs\bs}=\{(x, g\uo^-)\in  X_v^{w_1}(P)\times G/B^-\,:\,x\in
gX^{\bar{w}_2}_P\}.
$$
Since $\Chi$~is closed in~$G/P \times \mcX$ (see above Lemma~\ref{lem:Chiopen}), $\Chi_{\bs\bs}$ and $\Chi'_{\bs\bs}$ are closed in ${X_v^{w_1}(P)\times G/B^-}$. Note that
\begin{equation}\label{eqn71'} \Chi\cap (G/P \times \cmcX)\simeq G\times_T \Chi_{\bs\bs},\,\,\,
\Chi'\cap (G/P \times \cmcX)\simeq G\times_T {\Chi'}_{\bs\bs}
\end{equation} under the maps~$$\delta: [g: (x, h\uo^-)]\mapsto (gx, g\uo^-, gh\uo^-, g\uo)$$ 
and
$\cmcX \cap F_{\alpha,2} \simeq G\times_T
\overline{Bs_\alpha \uo^-}$.
Thus, to prove \eqref{eq:Xcirc}, it is sufficient to prove that 
\begin{equation}
  \label{eq:Ess}
  \hat{\mcX}_{\bs\bs}=\overline{Bs_\alpha \uo^-},
\end{equation}
where $\hat{\mcX}_{\bs\bs} :=\{g\uo^-\in G/B^-\,:\, X_v^{w_1}(P)\cap g X^{s_\alpha w_2}_P\neq
\emptyset\}$. By Lemma~\ref{Eisclosed},  $\hat{\mcX}_{\bs\bs}$~is closed in~$G/B^-$.

By the identity \eqref{eq:Ess}, it suffices to prove that $\hat{\mcX}_{ss}= \overline{Bs_\alpha \uo^-}$. By equation \eqref{image} applied in the setting of Corollary  \ref{prop:interSchub2} for $w=w_2$,
 we get $\hat{\mcX}_{ss}\subset \overline{Bs_\alpha \uo^-}$. Moreover, by Corollary \ref{prop:interSchub2}, 
$Bs_\alpha \uo^- \subset \hat{\mcX}_{ss}$ (since $X_v^P\cap X_P^{w_2}\cap g X^{w_1}_P\neq \emptyset$ for any$g\in G$ due to the fact that $\eps^v_P$ occurs in $\eps^{w_1}_P\cdot \eps^{w_1}_P$ with nonzero coefficient \cite[Proposition 3.5]{BrownKumar}).
 But since $\hat{\mcX}_{ss}$ is closed in $G/B^-$, we get $\hat{\mcX}_{ss}= \overline{Bs_\alpha \uo^-}$.
 \end{proof}

\subsection{Type $B$ descents} 
\begin{lemma} \label{typeblemma} For a descent $(\alpha, 3)$ of type $B$ such that $w_j=\bar{w}_j \not\leq \bar{v}$ (for some $1\leq j\leq 2$),
$$E_{\alpha, 3}:= E_{\bar{w}_1,\bar{w}_2,\bar{v}} = F_{\alpha, j},$$
where $ F_{\alpha, j}$ is defined by equation \eqref{eq:defF}.
\end{lemma}
\begin{proof}
Without loss of generality take $j=2$. By Lemma~\ref{Eisclosed}, $E_{\alpha,3}$  is
  closed and
 ind-irreducible.
Define the open subset of~$\mcX$ :
$$
\cmcX:=\{(x_1,x_2,x)\in\mcX\,:\, (x_1,x)\in G.(\uo^-,\uo)\}.
$$
Since $(\uo^-,s_\alpha\uo^-,\uo)\in \cmcX\cap F_{\alpha,2}$ and
$F_{\alpha,2}$~is ind-irreducible (cf. $\S$\ref{section4.1}), we have
\begin{equation}
  \label{eq1:Fbar}
  \overline{\cmcX\cap F_{\alpha,2}}=F_{\alpha,2}.
\end{equation}
Since $w_1\leq \bar{v}:= s_\alpha v$, the Richardson
variety~$X_{\bar{v}}^{w_1}(P):= X_{\bar{v}}^P\cap X_P^{w_1}$~is nonempty.  Take $x \in X_{\bar{v}}^{w_1}(P)$.
There exists $g\in G$ such that~$g\inv x\in X^{ w_2}_P$. 
Then, $(\uo^-,g\uo^-,\uo)$ belongs to~$\cmcX\cap\eta(\Chi')$, where $\Chi':= \Chi'_{\alpha, 3}$. 
Since $\eta(\Chi')$~is ind-irreducible, we deduce that
\begin{equation}
  \label{eq1:Ebar}
  \overline{\cmcX\cap\eta(\Chi')}=\eta(\Chi').
\end{equation}
By \eqref{eq1:Fbar} and \eqref{eq1:Ebar}, it is sufficient to prove that
\begin{equation}
  \label{eq1:Xcirc}
  \cmcX\cap\eta(\Chi')=\cmcX\cap F_{\alpha,2}.
\end{equation}

But $G\times_T G/B^-\longto \cmcX, [g:x]\longmapsto
(g\uo^-,gx,g\uo)$~is an isomorphism. Consider the intersection of~$\Chi$ with $G/P\times \uo^-\times G/B^-\times \uo$:
$$
\Chi_{\bs\bs} =\{(x, g\uo^-)\in  X_v^{w_1}(P)\times G/B^-\,:\,x\in gX^{w_2}_P\}
$$
and its closed subset~$$
{\Chi'}_{\bs\bs}=\{(x, g\uo^-)\in  X_{\bar{v}}^{w_1}(P)\times G/B^-\,:\,x\in
gX^{w_2}_P\}.
$$
Since $\Chi$~is closed in~$G/P \times \mcX$ (see above Lemma~\ref{lem:Chiopen}), $\Chi_{\bs\bs}$ and $\Chi'_{\bs\bs}$ are closed in ${X_v^{w_1}(P)\times G/B^-}$. Note that
\begin{equation}\label{eqn171'} \Chi\cap (G/P \times \cmcX)\simeq G\times_T \Chi_{\bs\bs},\,\,\,
\Chi'\cap (G/P \times \cmcX)\simeq G\times_T {\Chi'}_{\bs\bs}
\end{equation} under the maps~$$\delta: [g: (x, h\uo^-)]\mapsto (gx, g\uo^-, gh\uo^-, g\uo)$$ 
and
$\cmcX \cap F_{\alpha,2} \simeq G\times_T
\overline{Bs_\alpha \uo^-}$.
Thus, to prove \eqref{eq1:Xcirc}, it is sufficient to prove that 
\begin{equation}
  \label{eq1:Ess}
  \hat{\mcX}_{\bs\bs}=\overline{Bs_\alpha \uo^-},
\end{equation}
where $\hat{\mcX}_{\bs\bs} :=\{g\uo^-\in G/B^-\,:\, X_{s_\alpha v}^{w_1}(P)\cap g X^{ w_2}_P\neq
\emptyset\}$. By Lemma~\ref{Eisclosed},  $\hat{\mcX}_{\bs\bs}$~is closed in~$G/B^-$.

By the identity \eqref{eq1:Ess}, it suffices to prove that $\hat{\mcX}_{ss}= \overline{Bs_\alpha \uo^-}$. By equation \eqref{image} applied in the setting of Proposition \ref{prop:interSchub} for $w=w_2$,
we get $\hat{\mcX}_{ss}\subset \overline{Bs_\alpha \uo^-}$. Moreover, by Proposition \ref{prop:interSchub}, 
$Bs_\alpha \uo^- \subset \hat{\mcX}_{ss}$. But since $\hat{\mcX}_{ss}$ is closed in $G/B^-$, we get $\hat{\mcX}_{ss}= \overline{Bs_\alpha \uo^-}$.
\end{proof}

\subsection{Type $C$ descents} 

Take a descent $(\alpha, i), 1 \leq i \leq 3$,  of type $C$. Thus, 

\label{sec:wcd}
\begin{enumerate}[label={\rm(\roman*)}]
\item $\bar{w}_1 \leq \bar{v}$ and $\bar{w}_2\leq \bar{v}$;
\item $\ell(\bar{v})=\ell(\bar{w}_1)+\ell(\bar{w}_2)-1$;
\item there exist $l_1,\,l_2$ and $l_3$ in~$L$ such that the linear
  map~$$
\textstyle l_3\Tau_{\bar{v}}\longto\frac{\Tau}{l_1\Tau^{\bar{w}_1}}\oplus\frac{\Tau}{l_2\Tau^{\bar{w}_2}}
$$ 
is injective, where the Zariski tangent spaces
\end{enumerate}
$$
  \Tau=T_{\dot{e}}(G/P), \,\,\,
\Tau^{\bar{w}_i}=T_{\dot{e}}({\bar{w}_i}\inv X^{\bar{w}_i}_P),\,\,\,\text{and}\,\, 
\Tau_{\bar{v}}=T_{\dot{e}}(\bar{v}\inv X_{\bar{v}}^P).
$$
The above  condition  {\rm(iii)} follows from the following lemma.
\begin{lemma}\label{lem:wbdwc}
   For any descent ~$(\alpha,i)$ of type $C$, the triple $(\bar{w}_1,\bar{w}_2,
  \bar{v})$ satisfies the above condition (iii).   
\end{lemma}

\begin{proof} We first prove the lemma  for a descent $(\alpha, 1)$ of type $C$. By the proof of 
\cite[Lemma~19]{R:KM1}, there exists $l_1, l_2, l_3\in L$ such that~$$l_3\Tau_v\cap l_1\Tau^{w_1}\cap l_2\Tau^{w_2}= (0).$$
Now, $\Tau^{w_1}\supset \Tau^{\bar{w}_1}$, since
$$\Tau^{w_1}=\bigoplus_{\beta\in \Phi_P^+\cap w_1^{-1}\Phi^+}\, \lg_{-\beta}\,\,\,\text{and}\,\, \Tau^{\bar{w}_1}=\bigoplus_{\beta\in \Phi_P^+\cap \bar{w}_1^{-1}\Phi^+}\, \lg_{-\beta},$$
where $\Phi^+$~is the set of positive roots of the Kac-Moody Lie algebra $\lg$ and $\Phi_P^+:=\Phi^+\setminus\Phi^+(L)$ ($\Phi^+(L)$ being the set of positive roots of~$L$). Thus,
$$l_3\Tau_v\cap l_1\Tau^{\bar{w}_1}\cap l_2\Tau^{w_2}= (0).$$

The proof of the lemma for any descent $(\alpha, i)$ of type $C$ for $i=2$ or $3$ is identical.
\end{proof}

\begin{PROP}
\label{prop:linebundleSchubdiv} Let $(\alpha, i)$ be any descent of type $C$. Then, 
there exist a $G$\yh-linearized line bundle~$\Li_{\bar{w}_1,\bar{w}_2,\bar{v}}$ over~$\mathcal{X}$ of the form~$\Li_{\bar{w}_1,\bar{w}_2,\bar{v}}=\Li^-(\lambda_1)\boxtimes \Li^-(\lambda_2)\boxtimes \Li(\mu)$ for some~$ 
(\lambda_1, \lambda_2, \mu)\in P_+^{3}$ and
a nonzero $G$\yh-invariant section~$\sigma_{\bar{w}_1,\bar{w}_2,\bar{v}}$ of~$\Li_{\bar{w}_1,\bar{w}_2,\bar{v}}$ such that~$$
E_{\bar{w}_1,\bar{w}_2,\bar{v}}=\{x\in \mathcal{X}\,:\,\sigma_{\bar{w}_1,\bar{w}_2,\bar{v}}(x)=0\}.
$$   
\end{PROP}

Before we come to the proof of the proposition, we need to prove some preparatory results.

Let $U$~be the commutator subgroup~$[B,B]$ of~$B$ and $U\uo^-$ the open cell in~$G/B^-$. 
Set~$$
\Omega=\{(x_1,x_2,g_3\uo)\in \mathcal{X}\,:\, g_3\inv x_i\in U\uo^-\mbox{ for }i=1,2\}.
$$
It is easy to see that~$\Omega$~is open in~$\mathcal{X}$.

The construction of~$\Li_{\bar{w}_1,\bar{w}_2,\bar{v}}$ and $\sigma_{\bar{w}_1,\bar{w}_2,\bar{v}}$~is
made in two steps:

\medskip

(1)  construct  their restrictions to~$\Omega$ by using a slice
  technique to reduce to the case of finite-dimensional varieties 
  (see Lemma~\ref{lem:LsOmega} below). 
Now,  $E _{\bar{w}_1,\bar{w}_2,\bar{v}}$ corresponds to  the subvariety~$\hat E$
(see \eqref{eq:defhatE} below) of an affine space.
Lemma~\ref{lemma6} proves that~$\hat E$~is a closed divisor using Lemma~\ref{lem:Chiopen}. 

(2) Twist the restriction $(\Li_{\bar{w}_1,\bar{w}_2,\bar{v}})\restrict{\Omega}$ to avoid
  components of the zero locus of~$\sigma_{\bar{w}_1,\bar{w}_2, \bar{v}}$ in the boundary~${\mathcal X}-\Omega$. This step uses Lemmas~\ref{lem:LsOmega} and
  \ref{lem:vFfini} below.  

\medskip

Observe that, by the Birkhoff decomposition \cite[Theorem~6.2.8]{Kumar:KacMoody},
\begin{equation}\label{eqn3.2.1}
\mathcal{X}=\Omega\sqcup\left (\bigcup_{\alpha\in\Delta, \,
  i=1,2}F_{\alpha,i}\right ).
\end{equation}

Consider the group homomorphism $\theta\,:\,U\longto\Aut(X_{\bar{v}}^P)$ given by the action and let $U_{\bar{v}}$~be  its image. 
Note that~$U_{\bar{v}}$~is a finite-dimensional unipotent group. 
Set 
\begin{equation}
  \label{eq:defhatE}
\hat E:=\{u\in U_{\bar{v}}\,:\,
\left(uX^{\bar{w}_1}_{\bar{v}}(P)\right)\cap
X_{\bar{v}}^{\bar{w}_2}(P)\neq\emptyset\}.
\end{equation}

\begin{lemma}\label{lemma6}
  The subset~$\hat E$ of~$U_{\bar{v}}$~is a closed irreducible divisor of~$U_{\bar{v}}$.
\end{lemma}

\begin{proof}
  Consider the closed subset of~$ U_{\bar{v}}\times X_{\bar{v}}^{\bar{w}_2}(P)$:
$$
\hat\Chi:=\{(u,x)\in U_{\bar{v}}\times X_{\bar{v}}^{\bar{w}_2}(P)\,:\,u\inv x\in X_{\bar{v}}^{\bar{w}_1}(P)\},
$$
with its two projections~$p_1$ and $p_2$ on~$U_{\bar{v}}$ and $X_{\bar{v}}^{\bar{w}_2}(P)$ respectively.
Since $X_{\bar{v}}^{\bar{w}_2}(P)$~is projective, $p_1$~is proper. In particular, $\hat
E =p_1(\hat\Chi)$~is closed in~$U_{\bar{v}}$. 

Recall the definition of~$\Chi$ from the Equation~\eqref{eqn4.1} and as defined earlier in the proof of Lemma~\ref{Eisclosed},
\begin{align*}\Chi_\bs &:=\Chi\cap \left(G/P\times G/B^-\times G/B^-\times \{\uo\}\right)\\
&=\{(y, g_1\uo^-, g_2\uo^-, \uo)\in G/P\times \mathcal{X}: y\in (g_1X^{\bar{w}_1}_P)\cap  (g_2X^{\bar{w}_2}_P)\cap X^P_{\bar{v}}\},
\end{align*}
 its open subset~$$\cChi_1:=\Chi_\bs\cap \left(G/P\times (U\cdot\uo^-)\times  (U\cdot\uo^-)\times \{\uo\}\right),$$
and
$$\hat{\Chi}_\bs:=\pi_1^{-1}(\Chi_\bs),\,\,\,\text{where $\pi_1: G \times \mathcal{X}\to G/P \times \mathcal{X}$~is the projection}.$$
Then,
$$(\overline{B\bar{v}P})\times (\overline{P\bar{w}_1^{-1}B^-/B^-})\times (\overline{P\bar{w}_2^{-1}B^-/B^-})\simeq \hat{\Chi}_\bs, \, (g, x_1, x_2)\mapsto (g, gx_1, gx_2, \uo).$$
Hence, $\hat{\Chi}_\bs$~is irreducible and thus  so is its quotient~$\Chi_\bs$. By the condition (i) at the beginning of subsection ~\ref{sec:wcd}, $\cChi_1$~is
nonempty. By the condition (iii) at the beginning of subsection ~\ref{sec:wcd} and Lemma~\ref{lem:Chiopen}, $\Chi_\bs\cap \cChi^+$~is a nonempty open subset of~$\Chi_\bs$. Since $\Chi_\bs$~is irreducible and  $\Chi_\bs\cap \cChi^+$  and $\cChi_1$ are nonempty open subsets of irreducible~$\Chi_\bs$, their intersection~$$\cChi_1^+:=\cChi_1\cap \cChi^+\,\,\,\text{is nonempty}.$$

Consider the ind-variety $Y=G/P\times U\times U$ and the morphism~$$\alpha: Y \to G/P^{\times 3},\,\,\, (y, u_1, u_2) \mapsto (u_1^{-1}y, u_2^{-1}y, y).$$
Let~$Y' = Y_{(\bar{w}_1, \bar{w}_2, \bar{v})}\subset Y$  be the  closed ind-subvariety
$$ Y':=\alpha^{-1}\left(X^{\bar{w}_1}_P\times X^{\bar{w}_2}_P \times X_{\bar{v}}^P\right).$$
Then, there is an isomorphism
$$\hat{\beta}: \cChi_1\simeq Y',\,\,\,(y, u_1\uo^-, u_2\uo^-, \uo) \mapsto (y, u_1, u_2).$$
In particular, $Y'$~is also irreducible. Let~$$Y'_+:=\hat{\beta} (\cChi_1^+)\subset Y'\,\,\,\text{be the nonempty open subset}.$$
Consider the morphism~$$q:Y'\to \hat{\Chi},\,\,\,(y, u_1, u_2)\mapsto (\theta(u_2^{-1}u_1), u_2^{-1}y).$$
Clearly, $q$~is surjective. In particular, we obtain that~$ \hat{\Chi}$~is irreducible and hence so is ${\hat{E}=p_1(\hat{\Chi})}$.

We now determine the  image of~$p_2$: Let~$x\in X^{\bar{w}_2}_{\bar{v}}$(P) and let $v'\leq \bar{v}$~be such that~$v'\in W^P$ and 
$x\in \cX_{v'}^P$. Then, $x\in \Image (p_2)$ if and only if $Ux\cap
X^{\bar{w}_1}_P\neq\emptyset$ if and only if $\bar{w}_1\leq v'$ (cf. \cite[Lemma~7.1.22]{Kumar:KacMoody}). We deduce that
\begin{equation}\label{neweqn6.1}
\Image (p_2)=X^{\bar{w}_2}_P\cap\left(\bigcup_{\bar{w}_1\leq v'\leq \bar{v}; v'\in W^P}\,\cX_{v'}^P\right).
\end{equation}
In particular, it is open in~$X^{\bar{w}_2}_{\bar{v}}(P)$.

We now analyze the fibers of~$p_2$: Let~$x\in\Image (p_2)$ and $v'$ be as above.
Then, $p_2\inv(x)$~is the set of points~$u\in U_{\bar{v}}$ such that~$u^{-1}x\in
X^{\bar{w}_1}_P$. It is the pullback of~$\cX_{v'}^P\cap X^{\bar{w}_1}_P$ by the orbit
map~$u\mapsto u^{-1}x$. Since  $\cX_{v'}^P\cap X^{\bar{w}_1}_P$~is irreducible (cf. \cite[Proposition~6.6]{Kumar:positivity}) and the
stabilizer of~$x$ in~$U_{\bar{v}}$~is, of course,  irreducible (being a closed subgroup of a finite-dimensional unipotent group), so is $p_2\inv(x)$. 
Moreover,
\begin{equation}\label{eqn6.1}
  \begin{array}{ll}
    \dim(p_2\inv(x))&=\ell(v')+\dim(\Stab_{U_{\bar{v}}}(v'P/P)) -\ell(\bar{w}_1)\\
&=\ell(\bar{v})+\dim(\Stab_{U_{\bar{v}}}(\bar{v}P/P)) -\ell(\bar{w}_1),
  \end{array}
\end{equation}
where $ \Stab_{U_{\bar{v}}}(v'P/P)$ denotes the stabilizer of~$v'P/P$ in~$U_{\bar{v}}$.

Further, by Equations~\eqref{neweqn6.1} and  \eqref{eqn6.1},
\begin{align}\label{eqn6.2}
\dim \hat\Chi &=\ell(\bar{v})+\dim(\Stab_{U_{\bar{v}}}(\bar{v}P/P)) -\ell(\bar{w}_1)+\ell(\bar{v})-\ell(\bar{w}_2)\\ \notag
&= \dim U_{\bar{v}} -1,\,\,\,\text{by the condition (ii) at the beginning of subsection ~\ref{sec:wcd}}.
\end{align}

We return to the surjective map~$q:Y'\twoheadrightarrow \hat{\Chi}$ defined above. By Chevalley's theorem (cf. [Har77, Chap. II, Exercise 3.19(b)]),
$q(Y'_+)$ contains a nonempty open subset (denoted by~$\hat{\Chi}^+$) of~$\hat{\Chi}$.
By the definition of~$ \cChi^+_1$, 
we get  the following:
\begin{equation} \label{eqn6.4}
\Tau_x(u\cX^{\bar{w}_1}_{\bar{v}}(P))\cap \Tau_x(\cX^{\bar{w}_2}_{\bar{v}}(P)) =(0),\,\,\,\text{for any~$(u,x)\in  \hat\Chi^+\subset U_{\bar{v}}\times \cX^{\bar{w}_2}_{\bar{v}}(P)$},
\end{equation}
where 
$$\cX^{\bar{w}}_{\bar{v}}(P):= \cX^{\bar{w}}_P\cap \cX^P_{\bar{v}}.$$
Observe that~$\cX^{\bar{w}_i}_{\bar{v}}(P)$~is smooth (which follows from \cite[Lemma~7.3.10]{Kumar:KacMoody}). Consider the projection map~$$p_1^+: \hat\Chi^+ \to U_{\bar{v}},\,\,\,\text{where $p_1^+:={p_1}\restrict{\hat{\Chi}^+}$}.$$
From the above Equation~\eqref{eqn6.4}, we conclude that~$$
(p_1^+)^{-1}(p_1^+(u,x))\subset \{u\}\times \left((u\cX^{\bar{w}_1}_{\bar{v}}(P))\cap \cX^{\bar{w}_2}_{\bar{v}}(P)\right)$$
 is a finite set for any~$(u,x)\in \hat{\Chi}^+$. In particular, $\hat{E}$ being irreducible,
$$\dim (\hat{E})=\dim (\overline{\Image\, p_1^+})= \dim (\hat{\Chi}^+)= \dim (\hat{\Chi})= \dim (U_{\bar{v}})-1,$$
where the last equality follows from the Equation~\eqref{eqn6.2}.
This proves that~$\hat E$~is a divisor, proving the lemma. 
\end{proof}

\begin{lemma}
  \label{lem:LsOmega}
There exist a $G$\yh-equivariant  line bundle~\mbox{$\Mi\in\Pic(\Omega)$} and
nonzero\break ${\tau\in \Ho(\Omega,\Mi)^G}$ such that~$$\Omega\cap E=\{x\in\Omega\,:\,\tau(x)=0\},
$$
where $E=E_{\bar{w}_1, \bar{w}_2, \bar{v}}$ (defined by the Identity {\rm(11)--(12)}). In fact, we can take $\Mi = ({p_3}\restrict{\Omega})^*\Li_\chi$ for a character $\chi$ of~$B$.

In particular, $E\cap\Omega$~is closed in~$\Omega$.
\end{lemma}

\begin{proof}
By definition,
$$
\begin{array}{rl}
  E&=\{(g_1\uo^-,g_2\uo^-,g_3\uo)\in \mathcal{X}\,:\, g_1 X_P^{\bar{w}_1}\cap g_2 X_P^{\bar{w}_2}\cap g_3 X^P_{\bar{v}}\neq\emptyset\}\text{, by (12)}\\
&=\{(g_1\uo^-,g_2\uo^-,g_3\uo)\,:\, (g_3\inv g_1 X_P^{\bar{w}_1})\cap (g_3\inv g_2 X_P^{\bar{w}_2})\cap  X^P_{\bar{v}}\neq\emptyset\}.
\end{array}
$$  
Consider the isomorphism $\iota\,:\, U\uo^-\longto U, u\uo^-\mapsto u$.
Then,
$$
\begin{array}{rl}
  E\cap\Omega&=\{(x_1,x_2,g_3\uo)\in\Omega\,:\, \iota(g_3\inv x_1) X_P^{\bar{w}_1}\cap \iota(g_3\inv x_2) X_P^{\bar{w}_2}\cap  X^P_{\bar{v}}\neq\emptyset\}\\
&=\{(x_1,x_2,g_3\uo)\in \Omega\,:\, \left(\iota(g_3\inv x_1) X^{\bar{w}_1}_{\bar{v}}(P)\right)\cap \left(\iota(g_3\inv x_2) X_{\bar{v}}^{\bar{w}_2}(P)\right)\neq\emptyset\},
\end{array}
$$
since $X^P_{\bar{v}}$~is $U$\yh-stable. Here (as earlier) $X^{\bar{w}_1}_{\bar{v}}(P) :=X_P^{\bar{w}_1}\cap X^P_{\bar{v}}$.
Thus,
\begin{equation}
  \label{eq:1}
  E\cap\Omega=\{(x_1,x_2,g_3\uo)\in \Omega\,:\,\left( [\iota(g_3\inv x_2) \inv \iota(g_3\inv x_1)] X^{\bar{w}_1}_{\bar{v}}(P)\right)\cap X_{\bar{v}}^{\bar{w}_2}(P)\neq\emptyset\}.
\end{equation}
As earlier, consider the group homomorphism
$\theta\,:\,U\longto\Aut(X^P_{\bar{v}})$ given by the action, and denote by~$U_{\bar{v}}$ its image (which is a finite-dimensional unipotent group). 
Recall (cf. (19)) that~$$
\hat E:=\{u\in U_{\bar{v}}\,:\, \left(uX^{\bar{w}_1}_{\bar{v}}(P)\right)\cap X_{\bar{v}}^{\bar{w}_2}(P)\neq\emptyset\}.
$$
Note that the torus~$T$ acts by conjugation on~$U_{\bar{v}}$ and that~$\hat E$~is $T$\yh-stable.
Being  a finite-dimensional unipotent group, $U_{\bar{v}}$~is isomorphic as a
variety to an affine space. 
In particular, there exists $\hat f\in\CC[U_{\bar{v}}]$, unique up to scalar
multiplication,  such that~$\div(\hat f)=\hat E$ (since $\hat{E}$~is an irreducible divisor by Lemma~\ref{lemma6}).
Moreover, since $\hat E$~is $T$\yh-stable, $\hat f$~is an eigenvector of~$T$; denote by~$\chi$ the corresponding character.
We extend $\chi$ uniquely to a character of~$B$.

Set~$\tilde E=\tilde{\pi}\inv(E)$ and  $\tilde \Omega :=\tilde{\pi}\inv(\Omega)$, where $\tilde{\pi}: \tilde{\mathcal{X}}:=G/B^-\times G/B^-\times G\to \mathcal{X}$~is the projection.
Then, $\tilde \Omega$ and $\tilde E$ are stable by the following action of~$G\times B$:
$$
(g,b).(x_1,x_2,g'):=(gx_1,gx_2,gg'b\inv).
$$
Consider $\tilde f\,:\,\tilde\Omega\longto\CC$ defined by~$$
\tilde f(x_1,x_2,g)=\hat f\circ\theta(\iota(g\inv x_2)\inv\iota(g\inv x_1)).
$$
Then, by the Equation~\eqref{eq:1},  $\tilde E\cap\tilde \Omega$~is
the zero locus $Z(\tilde f)$ of~$\tilde f$ and for~$b=ut\in B$ (where $u\in U, t\in T$):
\begin{equation} 
\label{eqn4.2}
\begin{array}{l@{\,}l}
\tilde f(x_1,x_2,gb) :&=\hat f\circ\theta(t\inv [\iota(g\inv
                        x_2)\inv\iota(g\inv x_1)]t)\\
&=\chi(t)\tilde f(x_1,x_2,g) = \chi(b)\tilde f(x_1,x_2,g).
\end{array}
\end{equation}
We claim that~$\tilde f$ induces a section~$\tau_{\tilde{f}}$ of~$({p_3}\restrict{\Omega})^*(\Li_\chi)$, where $p_3: \mathcal{X}\to G/B$~is the projection onto the third factor. 

By the Equation~\eqref{eqn4.2}, $\tilde{f}$ gives rise to a section~$\tau_{\tilde{f}}$ of the line bundle~$\Li_\Omega (\chi)$ associated to the principal $B$\yh-bundle $\tilde{\Omega} \to \Omega$ (induced from the right $\cdot$ action of~$B$ on~$\tilde{\Omega}$)  via the character~$\chi^{-1}$ of~$B$. Clearly,
$$\Li_\Omega (\chi) = ({p_3}\restrict{\Omega})^*(\Li_\chi).$$

By construction, the zero set~$Z(\tau_{\tilde{f}})=E\cap\Omega$. By the definition of~$\tau_{\tilde{f}}$, it is easy to see that it is a $G$\yh-invariant section. Taking $\tau=\tau_{\tilde{f}}$, we get the lemma.
\end{proof}

\bigskip
We now have a line bundle and a section~$\tau$ on~$\Omega$ with the
expected zero locus. To avoid extra zero locus in the boundary~${\mathcal X}\backslash\Omega$ we need to twist by some line bundles given by
Lemma~\ref{lem:linebundlebord}. The key point to do this is the
following  finiteness result:

\begin{lemma}
\label{lem:vFfini}
  The valuation $v_{F_{\beta,i}}(\tau)$~is finite for any~$\beta\in\Delta$ and $i=1,2$, where $\tau$~is the section taken from Lemma~\ref{lem:LsOmega}. (In the proof below we see that~$F_{\beta,i}$~is ind-irreducible.)
\end{lemma}

\begin{proof}
We are going to prove that~$v_{F_{\beta,i}}(\tau)$ can be computed
in some finite-dimensional variety after taking a quotient by a
unipotent group. 
 
 Fix a simple root $\beta \in \Delta$ and $i=1$ and consider 
$$
F=F_{\beta, 1}=\{(x_1,x_2,g_3\uo)\in \mathcal{X}\,:\,g_3\inv x_1\in \overline{Bs_\beta \uo^-}\}.
$$ 	

Consider the isomorphism
	$$
	\varphi:\tilde{\mathcal{X}}\to \tilde{\mathcal{X}},\,\,\,\,
	(x_1,x_2,g)\mapsto (gx_1,gx_2,g).$$
	Endow $\tilde{\mathcal{X}}$ with the following two right actions of~$B$:
	$$ (x_1, x_2, g) \odot b= (b^{-1}x_1, b^{-1}x_2, gb)
	$$
	and
	$$  (x_1, x_2, g) \cdot b= (x_1, x_2, gb).
	$$
	Then, the morphism~$\varphi$~is $B$\yh-equivariant with respect
        to the action $\odot$ on the domain and the action $\cdot$ on
        the range. 

Clearly, $\tilde{\pi}: \tilde{\mathcal{X}}\to \mathcal{X}$~is a principal $B$\yh-bundle with respect to the action $\cdot$. Define
$$\tilde{\Omega}':=\varphi^{-1}(\tilde{\Omega}).$$
By the definition of~$\Omega$,
\begin{equation}\label{eqn2.3} 
\tilde{\Omega}'= U\uo^-\times U\uo^-\times G.
\end{equation}

Let~$\hat f$ and $\tilde f$ be as in the proof of
  Lemma~\ref{lem:LsOmega}. 
Set~$\tilde f'=\tilde f\circ\varphi: \tilde{\Omega}' \to \CC$. Thus, 
\begin{equation}
\label{eqn5.1}
\tilde f'(u_1\uo^-,u_2\uo^-,g)=\hat f\circ \theta(u_2\inv u_1),
\qquad \text{for} \,\,u_1,u_2\in U\,\,\text{and} \,\, g\in G.
\end{equation}
Set~$F':=(\tilde{\pi}\circ\varphi)\inv(F)=\overline{Us_\beta \uo^-}\times
G/B^-\times G$.
Consider $V^\beta:=U\uo^-\cup Us_\beta \uo^-$. It is an open
subset of~$G/B^-$ (containing $\overline{Us_\beta \uo^-}$).
By \cite[Lemma~6.1]{Kumar:positivity} and \cite[Proposition~3.2]{KS:Gpol}, there exists a closed normal subgroup~$\mcU$ of~$U$
such that~$V^\beta\longto \mcU\backslash V^\beta=:Y^\beta$~is a
principal $\mcU$\yh-bundle and $Y^\beta$~is a smooth finite-dimensional
variety. 
Moreover, by  intersecting with $\Ker\, \theta$, one can assume that~$\mcU$ acts trivially on~$X_{\bar{v}}$.

Let~$h_1, h_2\in \mcU$. We have, for any~$u_1,u_2\in U$ and $g\in G$,
\begin{align}\label{neweqn5.2}
  \tilde f'(h_1u_1\uo^-,h_2u_2\uo^-,g) & =\hat f\circ\theta(u_2\inv h_2\inv
                                h_1u_1),\,\,\,\text{by Equation~\eqref{eqn5.1}}\notag\\
&=\hat f\circ\theta(u_2\inv
                                u_1), \,\,\,\text{since $\theta$~is a group homomorphism}\notag\\
                                &\,\,\,\,\,\,\,	\text{and $h_1,h_2 \in \mcU\subset \Ker\, \theta$}\notag\\
&=\tilde f'(u_1\uo^-,u_2\uo^-,g).
\end{align}
Since the line bundle~$p_3^*(\Li_\chi)$ over~$\mathcal{X}$ pulled to the principal $B$\yh-bundle $\pi: \tilde{\mathcal{X}}\to \mathcal{X}$~is trivialized, to prove the finiteness of~$v_F(\tau)$, it suffices to show that the function ${\tilde{f}: \tilde{\Omega} \to \CC}$ has  a pole of finite order along $\pi^{-1}(F)$. 
Equivalently, considering the isomorphism ${\varphi: \tilde{\mathcal{X}}\to \tilde{\mathcal{X}}}$, it suffices to show that the function~$$\tilde{f}': \tilde{\Omega}'= U\uo^-\times U\uo^-\times G \to \CC$$ has  a pole of finite order along $F'= \overline{Us_\beta\uo^-}\times G/B^-\times G$, since $F'=\varphi^{-1}(F)$; in particular, $F$~is ind\yh-irreducible.

The diagonal action of~$G$ on~$\tilde{\mathcal{X}}$ pulled back via $\varphi$ to the action $\odot$ of~$G$ on~$\tilde{\mathcal{X}}$~is given by:
$$ g\odot (x_1, x_2, h) =  (x_1, x_2, gh), \,\,\,\text{for $x_1, x_2 \in G/B^-$ and $g,h\in G$}.$$
The function~$\tilde{f}':  U\uo^-\times U\uo^-\times G \to \CC$ descends to a function~$\hat{f}'$ on~$U\uo^-\times U\uo^-$ by Equation~\eqref{eqn5.1}. So, to prove that the function~$\tilde{f}'$~has  a pole of finite order along $F'$, it suffices to show that the function~$\hat{f}':  U\uo^-\times U\uo^- \to \CC$~has a pole of finite order along $\left(\overline{Us_\beta\uo^-}\right) \times G/B^-$. Consider the open embedding~$$\left(\mcU\backslash U\uo^-\right)\times \left(\mcU\backslash U\uo^-\right)\hookrightarrow \left(\mcU\backslash V^\beta\right) \times \left(\mcU\backslash U\uo^-\right).$$
By the Equation~\eqref{neweqn5.2}, the function~$\hat{f}'$ descends to a function~$\hat{\phi}'$ on~$\left(\mcU\backslash U\uo^-\right)\times \left(\mcU\backslash U\uo^-\right)$. Since
$ \left(\mcU\backslash V^\beta\right) \times \left(\mcU\backslash U\uo^-\right)$~is a (smooth) scheme of finite type over~$\CC$, the function~$\hat{\phi}'$~has  a pole of finite order 
along the divisor~$\left(\mcU\backslash (Us_\beta\uo^-)\right) \times \left(\mcU\backslash U\uo^-\right)$ and hence $\hat{f}'$~has a pole of finite order along the divisor~$(Us_\beta\uo^-)  \times U\uo^-$. Since $Us_\beta\uo^-$~is an open subset of~$\overline{Us_\beta\uo^-}$, we get that~$\hat{f}'$~has a pole of finite order along 
$(\overline{Us_\beta\uo^-})  \times U\underline{o}^-$.  This proves the finiteness of~$v_{F_{\beta, 1}}(\tau)$ for any~$\beta\in \Delta$. The proof of   the finiteness of~$v_{F_{\beta, 2}}(\tau)$~is identical. 
\end{proof}

\begin{proof}[Proof of Proposition~\ref{prop:linebundleSchubdiv}]
Observe that $E\ne\mathcal{X}$ by Lemma~4.5. By Lemma~\ref{lem:LsOmega}, there exist a $G$\yh-equivariant line bundle~$\Mi$ over~$\Omega$ and a nonzero $G$\yh-invariant section~$\tau$ over~$\Omega$ with its zero set $Z(\tau)= E\cap \Omega$.
Moreover, the line bundle~$\Mi$~is the
restriction of the  line bundle~$p_3^*(\Li_\chi)$ over~$\mathcal{X}$. 
Then,  $\tau$~is a (rational) section of~$\Mi':=p_3^*(\Li_\chi)$ regular over~$\Omega$.

Lemma~\ref{lem:vFfini} allows to consider the $G$\yh-linearized line bundle~$$\Li_{\bar{w}_1, \bar{w}_2, \bar{v}}:=\Mi' \otimes \left (\bigotimes_{\beta \in \Delta,
    \, i=1,2}\, \Mi_{\beta, i}^{v_{F_{\beta, i}}(\tau)}\right )\,\,\,\text{over $\mathcal{X}$},$$
where the line bundles~$ \Mi_{\beta, i}$ are as in Lemma~\ref{lem:linebundlebord}. In particular, $\Li_{\bar{w}_1,\bar{w}_2,\bar{v}}$~is of the form\break$\Li^-(\lambda_1)\boxtimes \Li^-(\lambda_2)\boxtimes \Li(\mu)$ for some~$ 
\lambda_1, \lambda_2, \mu\in \lh_\ZZ^*$. 

By Lemmas~\ref{lem:linebundlebord}, 4.6, 4.7 and the decomposition \eqref{eqn3.2.1}, it has a
nonzero $G$\yh-invariant section 
\begin{equation} \label{eqn3.2.2}\sigma_{\bar{w}_1, \bar{w}_2, \bar{v}}=\tau\otimes
  \left (\bigotimes_{\beta \in \Delta, \, i=1,2}\,\sigma_{\beta,
      i}^{v_{F_{\beta, i}}(\tau)}\right ).\end{equation}
Thus, by \cite[Corollary~8.3.12]{Kumar:KacMoody}, $(\lambda_1,
\lambda_2, \mu) \in P_+^{3}$.  This proves the proposition by using
the following Lemma~\ref{lemma3.3}.
\end{proof}

Observe that~$\overline{E\cap \Omega}\subset E$ (since $E$~is closed by Lemma~\ref{Eisclosed}). Moreover, since $E$~is irreducible and $E\cap \Omega \neq \emptyset$  (as $(\uo^-, \uo^-, \uo)\in E\cap \Omega$),
\begin{equation} \label{eqn3.2.3}   \overline{E\cap \Omega}= E.
\end{equation}
\begin{lemma} \label{lemma3.3}
 The zero set~$Z(\sigma_{\bar{w}_1, \bar{w}_2, \bar{v}}):= \{x\in \mathcal{X}: \sigma_{\bar{w}_1, \bar{w}_2, \bar{v}}(x)=0\}$~is equal to~$E$.
\end{lemma}
\begin{proof} Consider the map~$$\psi: \tilde{\mathcal{X}}:=(G/B^-)^2\times G \to  \mathcal{X}:=(G/B^-)^2\times G/B, \,\,\,(x_1, x_2, g) \mapsto (gx_1, gx_2, g\uo).$$
For any subset~$Y\subset \mathcal{X}$, we set~$\hat{Y}':= \psi^{-1}(Y)$. Then,
$$\hat{F}_{\beta,1}'= \overline{Bs_\beta\uo^-}\times G/B^-\times G.$$
Take an increasing cofinal sequence~$w_n\in W$ (i.e., $\bar{w}_1< \bar{w}_2< w_3
<\cdots$ and for any~$w\in W$ there exists a $w_n$ such that~$w\leq
w_n$). Take  a filtration $(G_n)_{n \geq 0}$ of~$G$ by finite-dimensional irreducible subvarieties compatible with its ind-variety
structure (cf. \cite[above Lemma~2.3]{R:KM1}). Now, define the increasing filtration
$$\tilde{\mathcal{X}}_n:= X_{w_n}^-\times  X_{w_n}^-\times G_n\,\,\, \text{of $\tilde{\mathcal{X}}$, where $X_w^-:=\overline{B^-w\uo^-}$}.$$
Then, 
\begin{equation} \label{eqn3.4.1} \tilde{\mathcal{X}}_n\cap \hat{F}_{\beta,1}'= ( X_{w_n}^-\cap \overline{Bs_\beta\uo^-})
 \times  X_{w_n}^-\times G_n,
\end{equation}
and a similar expression for~$\tilde{\mathcal{X}}_n\cap \hat{F}_{\beta,2}'$. Thus, $\tilde{\mathcal{X}}_n\cap \hat{F}_{\beta,i}'$~is irreducible. Abbreviate\break
${Z=Z(\sigma_{\bar{w}_1, \bar{w}_2, \bar{v}})}$. Then,  by Lemmas~\ref{lem:linebundlebord} and \ref{lem:LsOmega} and the identity \eqref{eqn3.2.1}, $Z\cap \Omega =E\cap \Omega$ and hence $Z\supset E $ by the identity \eqref{eqn3.2.3}. Write
$$\hat{Z}'=\hat{E}'\cup \left (\bigcup_{(\beta, i)\in \Delta\times \{1,2\}}(\hat{Z}'\cap \hat{F}_{\beta, i}')\right),\,\,\,\text{by the identity  \eqref{eqn3.2.1}}.$$
Thus, for any~$n\geq 0$, 
\begin{equation}\label{eqn3.4.2}\hat{Z}'\cap \tilde{\mathcal{X}}_n
  =(\hat{E}' \cap\tilde{\mathcal{X}}_n)\cup \left (\bigcup_{(\beta,
      i)\in \Delta\times \{1,2\}}
(\hat{Z}'\cap \hat{F}_{\beta, i}'\cap  \tilde{\mathcal{X}}_n)\right).
\end{equation}
But, being the zero set of a section of  a line bundle, $\hat{Z}'\cap \tilde{\mathcal{X}}_n $~is a divisor in~$\tilde{\mathcal{X}}_n$ for large~$n$ and so is  
$ \hat{F}_{\beta, i}'\cap  \tilde{\mathcal{X}}_n$ and the latter is irreducible (divisor of~$ \tilde{\mathcal{X}}_n$) by the Equation~\eqref{eqn3.4.1}. From the definition of~$\sigma$ given by the Equation~\eqref{eqn3.2.2}, we get (for any~$(\beta, i)\in \Delta\times \{1,2\}$)
\begin{equation}\label{eqn3.4.3}\hat{Z}'\cap \hat{F}_{\beta, i}'\cap  \tilde{\mathcal{X}}_n \subsetneq \hat{F}_{\beta, i}'\cap  \tilde{\mathcal{X}}_n, \,\,\,\text{for large enough $n$}.
\end{equation}
Thus, $\hat{Z}'\cap \hat{F}_{\beta, i}'\cap  \tilde{\mathcal{X}}_n$  is of codimension~$\geq 2$ in~$\tilde{\mathcal{X}}_n$ for large enough $n$. But, since  $\hat{Z}'\cap \tilde{\mathcal{X}}_n $~is a divisor in~$\tilde{\mathcal{X}}_n$, we get from the Equation~\eqref{eqn3.4.2} that~$$\hat{Z}'\cap \hat{F}_{\beta, i}'\cap  \tilde{\mathcal{X}}_n \subset \hat{E}'\cap  \tilde{\mathcal{X}}_n, \,\,\,\text{for large enough $n$}.$$
Thus, 
$$\hat{Z}'\cap  \tilde{\mathcal{X}}_n = \hat{E}'\cap  \tilde{\mathcal{X}}_n, \,\,\,\text{for large enough $n$ which gives $\hat{Z}'=\hat{E}'$}.$$
Hence, $Z=E$ proving the lemma. 
\end{proof}

Combining Lemmas~\ref{lem:linebundlebord}, \ref{typealemma} and \ref{typeblemma} and Proposition \ref{prop:linebundleSchubdiv}, we get the following.

\begin{CORO} \label{coro6.6.1} For any $(\alpha, i) \in \mcD$, there exists a $G$-linearized line bundle $\Ni_{\alpha,i}$ of the form $\Li^-(\lambda_1)\otimes\Li^-(\lambda_2)\otimes\Li(\mu)$, for some $\lambda_1, \lambda_2, \mu\in P^3_+$, together with a nonzero $G$-invariant section $\mu_{\alpha, i}$  of  $\Ni_{\alpha,i}$ such that the zero set 
$$Z(\mu_{\alpha, i})= E_{\alpha, i},\,\,\text{where $E_{\alpha, i}:= E_{\bar{w}_1, \bar{w}_2, \bar{v}}$}.$$
\end{CORO}
\begin{remark} Observe that for a descent $(\alpha, 3)$ of type $B$, there exists exactly one $1\leq j\leq 2$ such that $w_j=\bar{w}_j\not\leq \bar{v}$. To show this, assume for contradiction, that $w_i\not\leq \bar{v}$ for both $i=1,2$.Then, by Lemma \ref{typeblemma}, $E_{\alpha, 3}= F_{\alpha, 1}= F_{\alpha, 2}$. This is a contradiction since $F_{\alpha, 1}\neq F_{\alpha, 2}$. This proves the claim.
\end{remark}

 \section{Proof of Theorem~\ref{th:restCisom}}
\label{sec:pfrestCisom}

{\it In this section,  we fix $P$, $(w_1,w_2,v)$ and $\Li$ as in the theorem. }

As earlier, let~$\mcD$ denote the set of descents~$(\alpha,i)\in\Delta\times
\{1,2,3\}$ coming from~$\Delta^+(w_1)$, $\Delta^+(w_2)$ and
$\Delta^-(v)$, i.e., 
$$\mcD \cap (\Delta\times \{i\}) = \Delta^+(w_i)\,\,\,\text{for $i=1,2$
  and \, $\mcD \cap (\Delta\times \{3\}) = \Delta^-(v)$},
$$
where $ \Delta^+(w_i)$ and $\Delta^-(v)$ are defined in the Introduction.
\subsection{Strategy}\label{subsec4.1}

 We set
\begin{align*}
C&=Lw_1\inv \uo^-\times Lw_2\inv \uo^-\times Lv\inv \uo,\\
C^+&=Pw_1\inv \uo^-\times Pw_2\inv \uo^-\times Pv\inv \uo,\\
\intertext{and (as earlier)}
{\bar C}^+&={\bar C}^+_{w_1, w_2, v}:=\overline{Pw_1\inv \uo^-}\times \overline{Pw_2\inv \uo^-}\times \overline{Pv\inv \uo}.
\end{align*}
Recall from Equation~\eqref{eqn4.1}:
$$
\begin{array}{ll}
  \Chi&:=\{(gP/P,x)\in G/P\times \X\,:\, g\inv x\in \bar C^+\}\\
&=\{(y,g_1\uo^-,g_2\uo^-,g_3\uo)\in G/P\times \X\,:\,y\in
  g_1X_P^{w_1}\cap g_2X^{w_2}_P\cap g_3X_{v}^P\}. 
\end{array}
$$
As a closed subset of~$G/P\times \X$, it is a $G$\yh-ind-variety with the diagonal action of~$G$.
Consider the projection~$$
  \eta:\Chi \to \X, \,\,\,
(y,x)\mapsto x.
$$ 
For each~$(\alpha,i)\in \mcD$, consider the associated $P^3$\yh-orbit
$\partial C^+_{\alpha,i}$ in~$\X$, where   $$\partial
C^+_{\alpha,1} :=Pw_1\inv s_\alpha \uo^-\times Pw_2\inv \uo^-\times
Pv\inv \uo$$ and   $\partial
C^+_{\alpha,i}$ ($i=2,3$) are defined similarly. Then,  $\partial
C^+_{\alpha,i}$~is open in an irreducible component of~$\bar
C^+\backslash C^+$. 
Set~$$
{\tilde C}^+=\tilde{Y}^{w_1}\times\tilde{Y}^{w_2}\times\tilde{Y}_v,
$$
where
\begin{align*}
\tilde{Y}^{w_i}&:=(Pw_i^{-1}\underline{o}^-)\cup\left(\bigcup_{(\alpha,i)\in\mathcal{D}}Pw_i^{-1}s_{\alpha}\underline{o}^-\right)\\
\intertext{and}
\tilde{Y}_{v}&:=(Pv^{-1}\underline{o})\cup\left(\bigcup_{(\alpha,3)\in\mathcal{D}}Pv^{-1}s_{\alpha}\underline{o}\right).
\end{align*}
It is open in~${\bar C}^+$.
Similarly,  we define the open subset of~$X_P^{w_i}$:
$${\tilde
  X}^{w_i}_P
:=(B^-w_iP/P)\cup\left(\bigcup_{\alpha\in\Delta^+(w_i)}B^-s_\alpha w_iP/P\right)
\,\,\, (\text{for } i=1,2)
$$ 
and the open subset of~$X_v^P$:
$$
{\tilde
  X}_{v}^P:=(BvP/P)\cup\left(\bigcup_{\alpha\in\Delta^-(v)}Bs_\alpha vP/P\right).
$$
We  also set~$$
\begin{array}{ll}
  \tilde\Chi' &:=\{(gP/P,x)\in G/P\times \X\,:\, g\inv x\in {\tilde C}^+\}\\
&=\{(y,g_1\uo^-,g_2\uo^-,g_3\uo)\in G/P\times \X\,:\,y\in
  g_1\tilde X_P^{w_1}\cap g_2\tilde X^{w_2}_P\cap g_3\tilde X_{v}^P\},
\end{array}
$$
which is an open subset of~$\Chi$ and hence ind\yh-irreducible (since so is $\Chi$ as observed earlier below the Equation~\eqref{eqn4.1}). 
We make use of a slice by setting
$$
\X_\bs :=(G/B^-)^2\times\{\uo\}\subset \X,$$
and its $B$\yh-stable open subset~$$ \ccdX_\bs :=(B\uo^-\cup \bigcup_{\alpha\in\Delta} s_\alpha B\uo^-)^2\times\{\uo\} = \left(\bigcup_{\ell(w)\leq 1}  Bw\uo^-\right)^2\times\{\uo\}.
$$
Then, we have a $G$\yh-equivariant isomorphism:
\begin{equation} \label{eqn6.1.1}
G\times^B\X_\bs\simeq \X,\,\,\,[g, x]\mapsto gx.
\end{equation}
As defined in the proof of Lemma~\ref{Eisclosed}, 
$$
\Chi_\bs :=\{(y,g_1\uo^-,g_2\uo^-, \uo)\in G/P \times  \X_\bs\,:\,y\in
  g_1X_P^{w_1}\cap g_2X^{w_2}_P\cap X_v^P\} \subset \Chi . $$
We also set~$$
\ccChi_\bs :=\Chi_\bs\cap (G/P \times \ccdX_\bs)
$$
and
$$
\tilde \Chi_\bs :=\{(y,g_1\uo^-,g_2\uo^-, \uo)\in G/P\times \X_\bs\,:\,y\in
  g_1\tilde X_P^{w_1}\cap g_2 \tilde X_P^{w_2}\cap  \tilde X_v^P\}.
$$
Then,
\begin{equation} \label{eqn6.1.2}
G\times^B\Chi_\bs\simeq \Chi,\,\,\,[g, x]\mapsto gx.
\end{equation}
In particular, $\Chi_\bs$~is irreducible since so is $\Chi$. Hence, $\tilde \Chi_\bs$ and $\ccChi_\bs $  (being  open subsets of~$\Chi_\bs$) are irreducible.

We now consider the following commutative diagram $(\diamond)$ for any~$G$\yh-equivariant line bundle~$\Li$ over~$\X$ as in Theorem~\ref{th:restCisom}:

\begin{center}
\begin{tikzpicture}
\matrix (m) [matrix of math nodes,row sep=1.2cm,column sep=0.9cm] {
\Ho(\X,\Li)^G&\Ho(\Chi,\Li)^G&\Ho(G\times^P {\tilde C}^+,\Li)^G&\Ho(C,\Li)^L\\
\Ho(\X_\bs,\Li)^B&\Ho(\Chi_\bs,\Li)^B&\Ho(\tilde\Chi_\bs,\Li)^B\\
\Ho(\ccdX_\bs,\Li)^B&\Ho(\ccChi_\bs,\Li)^B&\Ho(\tilde\Chi_\bs\cap \ccChi_\bs,\Li)^B,\\
};
\path [->]     (m-1-1) edge node[above]     {$\eta^*$} (m-1-2);
\path [->]     (m-1-2) edge node[above]     {$\alpha^*$}(m-1-3);
\path [->]     (m-1-3) edge node[above]     {$\beta^*$}(m-1-4);
\path [->]     (m-2-1) edge node[above]     {$\eta_1^*$} (m-2-2);
\path [->]     (m-2-2) edge node[above]     {$i_3^*$} (m-2-3);
\path [->]     (m-3-1) edge node[above]     {$\eta_2^*$} (m-3-2);\path [->]     (m-3-2) edge node[above]     {$i_7^*$} (m-3-3);
\path [->]     (m-1-1) edge node[left]     {$i_1^*$}  (m-2-1);
\path [->]     (m-1-2) edge node[left]     {$i_2^*$}  (m-2-2);
\path [->]     (m-2-1) edge node[left]     {$i_4^*$}  (m-3-1);
\path [->]     (m-2-2) edge node[left]     {$i_5^*$}  (m-3-2);
\path [->]     (m-1-3) edge node[right]     {$\gamma^*$}  (m-2-3);
\path [->]     (m-2-3) edge node[right]     {$i_6^*$}  (m-3-3);
\end{tikzpicture} 
\end{center} 
where $$\alpha: G\times^P \tilde{C}^+ \to \Chi,\,\,\,[g, (x_1, x_2, x_3)]\mapsto (gP, gx_1, gx_2, gx_3)$$
is a $G$\yh-equivariant  open embedding with image $\tilde\Chi'$,
\begin{align*}
\beta: C&\hookrightarrow  G\times^P \tilde{C}^+ \,\,\,\text{is the $L$\yh-equivariant morphism~$x \mapsto [1, x]$},\\
\gamma\,: \tilde \Chi_\bs &\longto  G\times^P \tilde{C}^+, \,\,\,
(gP, g_1\uo^-, g_2\uo^-, \uo) \mapsto [g,
      (g^{-1}g_1\uo^-, g^{-1}g_2\uo^-, g^{-1}\uo) ],
\end{align*}
is the morphism
(which is $\alpha^{-1}\restrict{\tilde\Chi_\bs}$),
$\eta_1, \eta_2$ are restrictions of~$\eta$ to~$\Chi_\bs$ and $\ccChi_\bs$ respectively. All the maps~$i_j$ are appropriate inclusion maps. 
In the above diagram $\Li$ also denotes the  induced line
bundle on each of the above   ind-varieties by pullback. 
Note that the ind-varieties with $\bs$ as subscript are 
$B$\yh-ind-varieties with the $B$\yh-action induced from the $G$\yh-action of the ambient $G$\yh-ind-varieties; in particular, the line bundle~$\Li$ over them  is
endowed with a natural $B$\yh-action. 
 
We now  prove that all the maps in the above commutative
diagram are isomorphisms.

\subsection{Various isomorphisms}\label{subsec4.2} We first prove the following lemma for its use in the proof of Lemma~\ref{lemma6.2.1}.
\begin{lemma}\label{lemma6.2.0} Let $U_P$~be the unipotent radical of~$P$. Then,

{\rm(a)} Any regular map~$U_P\to \mathbb{C}^*$~is constant.

{\rm(b)} $\Pic (U_P)= (0)$.
\end{lemma}
\begin{proof} (a) Consider the parabolic subgroup~$P^-$ opposite to~$P$ and the homogeneous space~$G/P^-$. Then $U_P$ can be seen as an
  open subset of~$G/P^-$. 
For any Schubert variety~$X_w^-=X_w^-(P):=\overline{B^-wP^-/P^-}\subset G/P^-$ (with
$w\in W^P$), 
$X_w^-\cap U_P$~is contractible in the analytic topology (cf. \cite[Proposition~7.4.17 and its proof]{Kumar:KacMoody}). Now, by \cite[Lemma~2.5]{KNR}, we get that any regular map~$X_w^-\cap U_P\to \mathbb{C}^*$~is a constant. From this (a) follows.

(b) By induction on~$\ell(w)$, we show that the group of~$k$\yh-cycles modulo rational equivalence  $A_k(X_w^-\cap U_P)$~is a finitely generated group. 
By \cite[Proposition~1.8]{Fulton:IT}, we have an exact sequence:
\[A_k((\partial X_w^-)\cap U_P)\to A_k(X_w^-\cap U_P)\to A_k\left((B^-wP^-/P^-)\cap U_P\right)\to 0.\]
Writing $\partial X_w^-$ as a union  $\bigcup_{\ell(v)=\ell(w)-1}\,X_v^-$ and applying \cite[Example~1.3.1(c)]{Fulton:IT} and the induction hypothesis, we get that~$A_k(\partial X_w^-\cap U_P)$~is finitely generated. Also, applying  \cite[Proposition~1.8]{Fulton:IT} again to the open subset~$(B^-wP^-/P^-)\cap U_P$ of the affine space~$B^-wP^-/P^-$, we get that~$A_k\left((B^-wP^-/P^-)\cap U_P\right)$~is finitely generated since so is $A_k(B^-wP^-/P^-)$ (cf. \cite[Proposition~1.9]{Fulton:IT}). Thus, from the above exact sequence, we get that~$A_k(X_w^-\cap U_P)$~is finitely generated, completing the induction. 

Consider the cohomology exact sequence (since $X_w^-\cap U_P$~is contractible in the analytic topology)
\begin{align*} \Hd^1(X_w^-\cap U_P, \mathbb{Z}_m)=0 &\to \Hd^1(X_w^-\cap U_P, \mathscr{O}^*)
=\Pic(X_w^-\cap U_P)  \\
&\to \Hd^1(X_w^-\cap U_P,  \mathscr{O}^*)
=\Pic(X_w^-\cap U_P) \to \Hd^2(X_w^-\cap U_P, \mathbb{Z}_m)=0,
\end{align*}
 induced from the sheaf exact sequence:
\[\mathbb{Z}_m\to  \mathscr{O}^*\to  \mathscr{O}^*\to 0,\]
where the map~$\mathscr{O}^*\to  \mathscr{O}^*$ takes $f\mapsto f^m$.
From the above cohomology exact sequence we see that~$\Pic(X_w^-\cap
U_P)$~is a divisible group. But, since it is also a finitely generated
abelian group (by \cite[Example~2.1.1]{Fulton:IT}), it must be trivial. From this, taking limit,  we obtain (b).
\end{proof}

Since $\Chi$~is irreducible and $\Image\, \alpha=\tilde\Chi'$~is open in~$\Chi$, the
restriction map~$$
 \Ho(\Chi,\Li)\longto \Ho(G\times^P {\tilde C}^+,\Li)\,\,\,\text{is injective and hence so is $\alpha^*$.}
$$
\begin{lemma} \label{lemma6.2.1}
 {\rm(a)} The pullback induces an isomorphism:
$$
\eta^*:\Ho(\X,\Li)^G\simeq\Ho(\Chi,\Li)^G.
$$

{\rm(b)} The restriction map~$$
\Ho({\tilde C}^+,\Li)^P\longto \Ho( C^+,\Li)^P
$$
is an isomorphism.

{\rm(c)}  The restriction map~$$
 \Ho( C^+,\Li)^P \to  \Ho( C,\Li)^L 
$$
is an isomorphism.

\end{lemma}
\begin{proof} (a) follows by \cite[Lemma~11]{R:KM1}.

The proof of (b) is analogous to the proof of
\cite[Lemma~13]{R:KM1}. We sketch the proof:
the map~$
\Ho({\tilde C}^+,\Li)^P\longto \Ho( C^+,\Li)^P
$  is obviously injective. 
Hence, it remains to prove that any~$P$\yh-invariant section~$\sigma$ of~$\Li$ on~$C^+$ extends to~${\tilde C}^+$.  

For~$x\in W^P$, $Px\inv \uo^-$~is contained in~$\overline{Pw_i^{-1}\uo^-}$ if and only if $x\geq w_i$.
Moreover,\break $\{z\in W\,:\,z\uo^-\in Px\inv \uo^-\}$~is the set of~$z\in
W$ that can be written as~$z=yx\inv$ for some~$y\in W_P$. Since 
$xy\inv\geq x$, such a point~$z\uo^-$ belongs to~$\overline{Bw_i^{-1}\uo^-}$. 
Then, $\overline{Pw_i^{-1}\uo^-}$ and $\overline{Bw_i^{-1}\uo^-}$ are
$B$\yh-stable and contain the same $T$\yh-fixed points. We deduce that 
\begin{equation}
  \label{eq:89}
  \overline{Pw_i^{-1}\uo^-}=\overline{Bw_i^{-1}\uo^-}.
\end{equation}
Yet, $\overline{Pv^{-1}\uo} =
\bigcup_{v_n\in W_P}\,X_{v_nv^{-1}}$, where $v_n$  is an increasing cofinal sequence in~$W_P$.
We now construct an increasing filtration of~$\bar C^+$ by products of finite-dimensional Richardson varieties:
$$
\bar C^+=\bigcup_{n\in \NN}\bar C^+_n.$$
Explicitly
$$\bar C^+_n:=(X_-^{w_1}\cap X^-_{w_n})\times (X_-^{ w_2}\cap X^-_{w_n})\times X_{v_nv^{-1}},$$
where $\{w_n\}$~is a cofinal increasing sequence in~$W$ and $\overline{Pw_i^{-1}\uo^-} = X_-^{w_i}$ by the Equation~\eqref{eq:89}, where 
$X_-^{w} :=\overline{Bw^{-1}\uo^-}$ and $X^-_{w} :=\overline{B^-w\uo^-}$. 
In particular, $\bar{C}^+_n$ are irreducible and normal (cf. \cite[Proposition~6.6]{Kumar:positivity}). Of course,   $\bar C_n^+\cap C^+$~is open in~$\bar C_n^+$ and
nonempty for large enough~$n$. 
It remains to prove that~$\sigma\restrict{\bar C_n^+\cap C^+}$ extends to a regular section
on~$\bar C^+_n\cap\tilde C^+$, for any~$n$.

Fix $(\alpha,i)\in{\mathcal D}$. The irreducibility  of the Richardson
varieties implies that the intersection~$\bar C^+_n\cap \overline{\partial
C^+_{\alpha,i}}$~is either empty or irreducible.
Since $\bar C^+_n$~is normal, to prove that~$\sigma_{|\bar C^+_n\cap C^+}$ extends 
to~$\bar C^+_n\cap\tilde C^+$, it is sufficient to prove that~$\sigma_{ |\bar C^+_n\cap C^+}$~has no pole along $\bar C^+_n\cap \overline{\partial
C^+_{\alpha,i}}$ if  $\bar C^+_n\cap \overline{\partial
C^+_{\alpha,i}}$~has codimension~$1$ in~$\bar C^+_n$ .

Assume that~$D_n:=\bar C^+_n\cap \overline{\partial
C^+_{\alpha,i}}$~has codimension~$1$ in~$\bar C^+_n$.
Then, $D_n$~is equal to either
\begin{enumerate}
\item[$(\alpha)$] $(X_-^{\bar u_1}\cap X^-_{w_n})\times (X_-^{w_2}\cap X^-_{w_n})\times X_{v_nv^{-1}},$
 for some~$\bar u_1\geq w_1\in W^P$ and $\ell(\bar u_1)=\ell(w_1)+1$; or 
\item[$(\alpha')$] $(X_-^{w_1}\cap X^-_{w_n})\times (X_-^{\bar u_2}\cap X^-_{w_n})\times X_{v_nv^{-1}},$
for some~$\bar u_2\geq w_2\in W^P$ and $\ell(\bar u_2)=\ell( w_2)+1$; or 
\item[$(\beta)$] $(X_-^{w_1}\cap X^-_{w_n})\times (X_-^{ w_2}\cap X^-_{w_n})\times X_{v_nv^{-1}s_\alpha}.$
\end{enumerate}

Now, we  construct an explicit affine open
subset~$\Omega_n$ in~$\bar C^+_n$ that intersects  
$D_n$. 

In case~$(\alpha)$,
set~$$
\Omega_n=(X_-^{w_1}\cap X^-_{w_n}\cap ((\bar u_1)^{-1}B\uo^-))\times (X_-^{w_2}\cap \cX^-_{w_n})\times \cX_{v_nv^{-1}},
$$
where $\cX^-_{w}:= B^-w\uo^-$ and  $\cX_{w}:= Bw\uo$
and similarly for the case~$(\alpha')$.
In case~$(\beta)$,
$$
\Omega_n=(X_-^{w_1}\cap \cX^-_{w_n})\times (X_-^{w_2}\cap \cX^-_{w_n})\times (X_{v_nv^{-1}}\cap (v_nv^{-1}s_\alpha B^-\uo)).
$$

Fix $\tau=z^{\sum_{\alpha_i\not\in \Delta(P)}d_ix_i}\,:\,\CC^*\longto T$, where $d_i >0$~is an integer such that~$d_ix_i\in\mathfrak{h}_{\mathbb{Z}}$.
We now apply \cite[Lemma~33]{R:KM1} to~$\Omega_n$ endowed with the
action of~$\CC^*$ induced by~$\tau$. The checking of the assumptions
(i)--(iv) of \cite[Lemma~33]{R:KM1} are done in the proof of
\cite[Lemma~13]{R:KM1}. 
The only remaining point, with the notation of
\cite[Lemma~33]{R:KM1}, is to prove that~$k\geq 0$.
This is done as in \cite[Proof of Lemma~13, specifically the part `The line bundle on the affine
subvarieties']{R:KM1}. Here, the non-negativity of~$k$~is due to the fact
that~$\Li$~is nonnegative restricted to the projective lines $\bl_{\alpha,i}$  for any~$(\alpha, i)\in  {\mathcal D}$, which is our assumption (cf. Theorem~\ref{th:restCisom}). This proves (b).

\bigskip
We now come to the proof of (c). Since $\Ho( C^+,\Li)^P$~is contained in~$\Ho( C^+,\Li)^\tau$,
\cite[Lemma~14]{R:KM1} implies that the map~$(c)$ of the lemma is injective. We now prove its surjectivity:

Consider the map~$\theta: P\longto L,\,p\longmapsto \lim_{t\to
  0}\tau(t)p\tau(t\inv)$, which  is a surjective group homomorphism. This provides 
an action of~$P^3$ on~$C$ through the homomorphism $\theta$.  
Then, the regular map~$\gamma:\,C^+\longto C,\,x\longmapsto \lim_{t\to
  0}\tau(t)x$~is $P^3$\yh-equivariant. 

Take the canonical $G^3$\yh-equivariant structure on~$\Li$ over~$\mathcal{X}$ under the componentwise action of~$G^3$ on~$\mathcal{X}$. Thus, we will
think of~$\Li$ as a $G^3$\yh-equivariant line bundle over~$\mathcal{X}$. Denote
$$x=(w_1^{-1}\uo^-, w_2^{-1}\uo^-, v^{-1}\uo)\in C.$$
Then, $C=L^3\cdot x$ and $C^+=P^3\cdot x$. Thus,  
\begin{equation}\label{eqn4.2.1} 
\Pic^{P^3}(C^+)\simeq X(P^3_x)\,\,\,\text{and}\,\,\, \Pic^{L^3}(C)\simeq X(L^3_x),
\end{equation}
where $X(\ )$ denotes the character group and $P^3_x$ (resp.\ $L^3_x$) denotes the isotropy subgroup of~$P^3$ (resp.\ $L^3$) at~$x$. Now, it is easy to see that (by considering $P\cap w_i^{-1}B^-{w_i}$ and $P\cap v^{-1}Bv$)
\begin{equation}\label{eqn4.2.2} 
P^3_x=L^3_x\cdot \left(U_{w_1}\times U_{w_2}\times U_v'\right),
\end{equation}
where $U_w$ (resp.\ $U_v'$) is the finite-dimensional (resp. \ finite-codimensional) subgroup of the \tetrecitw{unipotent radical $U_P$ of~$P$ with Lie algebra
$\bigoplus_{\beta\in \Phi^+\cap w^{-1}\Phi^-}\, \mathfrak{g}_\beta$ (resp.\ $\bigoplus_{\beta\in (\Phi^+\backslash \Phi_L^+)\cap v^{-1}\Phi^+}\, \mathfrak{g}_\beta$),} where $\Phi^+$ (resp.\ $\Phi_L^+$) is the set of positive roots of~$G$ (resp.\ $L$). Moreover, since $L^3$ normalizes~$U_P^3$,  $L_x^3$ normalizes $U_{w_1}\times U_{w_2}\times U_v'$. Now, for a finite-dimensional unipotent group, any character is trivial and similarly $U_v'$~has no nontrivial characters by the same proof as that of Lemma~\ref{lemma6.2.0}(a). Thus,
$$ X(P^3_x)= X(L^3_x).$$
Hence, by combining the Equations~\eqref{eqn4.2.1} and \eqref{eqn4.2.2}, we get
\begin{equation}\label{eqn4.2.3} 
\Pic^{P^3}(C^+)\simeq  \Pic^{L^3}(C).
\end{equation}
We define the $P^3$\yh-action on~$\Li\restrict{C}$ compatible with the action of~$P^3$ on~$C$ by demanding that~$U_P^3$ acts trivially on~$\Li\restrict{C}$. Thus, we
get a $P^3$\yh-equivariant line bundle~$\gamma^*(\Li\restrict{C})$  over~$C^+$. We also have a $P^3$\yh-equivariant line bundle~$\Li\restrict{C^+}$. By the Equation~\eqref{eqn4.2.3}, we readily see that~$$\Li\restrict{C^+} \simeq \gamma^*(\Li\restrict{C}),\,\,\,\text{as $P^3$\yh-equivariant line bundles};$$
in particular, as diagonal  $P$\yh-equivariant line bundles.

Thus, for~$\sigma\in \Ho( C,\Li)^L$,
$\gamma^*(\sigma)\in \Ho( C^+,\Li)^P$ and
$\gamma^*(\sigma)\restrict{C}=\sigma$. 
We deduce thus that  the restriction map~$\Ho( C^+,\Li)^P \to  \Ho(
C,\Li)^L $~is surjective. This proves (c).
\end{proof}

We thus conclude that the first horizontal line in the above diagram $(\diamond)$ satisfies:

\begin{center}
\begin{tikzpicture}
\matrix (m) [matrix of math nodes,row sep=1cm,column sep=1cm] {
\Ho(\X,\Li)^G&\Ho(\Chi,\Li)^G&\Ho(G\times^P {\tilde C}^+,\Li)^G&\Ho({\tilde C}^+,\Li)^P\\
&&&\Ho(C,\Li)^L,\\
};
\path [->]     (m-1-1) edge node[above,inner sep=0.2pt]     {$\sim$}  node[below] {$\eta^*$}(m-1-2);
\path [->]     (m-1-2) edge node[above]     {$\alpha^*$}(m-1-3);
\path [->]     (m-1-4) edge node[right,inner sep=2pt]     {$\wr$}(m-2-4);
\path [->]     (m-1-3) edge node[above,sloped,inner sep=0.2pt]     {$\sim$}(m-2-4);
\path [->]     (m-1-3) edge node[below]     {$\beta^*$}(m-2-4);
\path [commutative diagrams/.cd, every arrow, every label]     (m-1-2) edge [commutative diagrams/hook] (m-1-3);
\path [->]     (m-1-3) edge node[above,inner sep=0.2pt]     {$\sim$}(m-1-4);
\end{tikzpicture} 
 \end{center}
 where $\eta^*$~is  an isomorphism and the last vertical map is an
 isomorphism (which follows from Lemma~\ref{lemma6.2.1}).   

\subsection{Isomorphisms induced from slice}\label{subsec4.3}

Since $G\times^B\X_\bs \simeq \X$ (cf. Equation~\eqref{eqn6.1.1}), we get that~$i_1^*:\Ho(\X,\Li)^G\longto
\Ho(\X_\bs,\Li)^B$~is an isomorphism. Similarly,  $i_2^*$  is an isomorphism by using Equation~\eqref{eqn6.1.2}.
Further, $\gamma^*$~is an isomorphism since 
$\alpha:G\times^P {\tilde C}^+\to \tilde\Chi'$~is a $G$\yh-equivariant isomorphism and so is
\begin{equation} \label{eqn4.3.1} G\times^B\tilde\Chi_\bs \simeq \tilde\Chi',\,\,\,[g, x] \mapsto gx.
\end{equation}

\subsection{Isomorphisms obtained from restriction to some  open subsets}
\label{subsec4.4}

\begin{lemma}
  The restriction map~$\Ho(\X_\bs,\Li)\longto \Ho(\ccdX_\bs,\Li)$~is an
  isomorphism and hence $i_4^*$~is an isomorphism. 
\end{lemma}

\begin{proof} For any~$w\in W$, consider the Schubert variety~$$X_w^-:=\overline{B^-wB^-/B^-}\subset G/B^-.$$
For any~$w_1, w_2\in W$, consider the open embedding~$$i_{w_1, w_2}: \ccdX_\bs \cap \left(X_{w_1}^-\times X_{w_2}^-\times\{\uo\}\right) \hookrightarrow X_{w_1}^-\times X_{w_2}^-\times\{\uo\}.$$
The complement 
$$Y_{w_1, w_2}:= \left(X_{w_1}^-\times X_{w_2}^-\times\{\uo\}\right) \backslash \Image (i_{w_1, w_2})$$
has its irreducible components of the form~$$(X_{w_1}^-\cap \overline{BuB^-/B^-})\times X_{w_2}^-\times\{\uo\} \,\,\,\text{or}\,\, X_{w_1}^-\times (X_{w_2}^-\cap \overline{BuB^-/B^-})\times\{\uo\}$$
for some~$\ell (u) = 2$. But, by \cite[Lemma~7.3.10]{Kumar:KacMoody}, each of these irreducible components have codimension~$2$ in (the finite-dimensional) $X_{w_1}^-\times X_{w_2}^-\times\{\uo\}$.
Thus, by the normality of~$X_w^-$ (cf. \cite[Theorem~8.2.2(b)]{Kumar:KacMoody}, we see that the restriction map~$$ \Ho(X_{w_1}^-\times X_{w_2}^-\times\{\uo\}, \Li) \to \Ho(\ccdX_\bs \cap (X_{w_1}^-\times X_{w_2}^-\times\{\uo\}), \Li) $$
is an isomorphism. Taking limits over~$w_1, w_2$, we get the lemma.
\end{proof}

As observed earlier, $\tilde\Chi'$~is irreducible and hence so is $\tilde\Chi_\bs$  by the isomorphism \eqref{eqn4.3.1} and $\tilde\Chi_\bs\cap
\ccChi_\bs$~is open in~$\tilde\Chi_\bs$. It follows thus  that the map~$$
i_6^*:\Ho(\tilde\Chi_\bs,\Li)^B\longto\Ho(\tilde\Chi_\bs\cap \ccChi_\bs,\Li)^B\\
$$
is injective.

We now prove that the maps~$\eta_2^*$ and $i_7^*$ are isomorphisms.

\begin{lemma} \label{newlemma10}The  map~$\Ho(\ccdX_\bs,\Li) \to \Ho(\ccChi_\bs,\Li)$ induced from~$\eta_2$  is an isomorphism and hence so is $\eta_2^*$ .
\end{lemma}
\begin{proof}
It is easy to see that 
the map~$\eta_2$~is proper. Moreover, it is  birational by  \cite[Lemma~10]{R:KM1}. In particular, it is surjective. 
If $\ccdX_\bs$ and $\ccChi_\bs$ are finite-dimensional, the lemma
follows from Zariski's main theorem (see, e.g.,  \cite[Chap. III, Corollary~11.4]{Hart}). 
The argument used to prove
\cite[Lemma~11]{R:KM1} allows us to prove that the above map ${\Ho(\ccdX_\bs,\Li) \to \Ho(\ccChi_\bs,\Li)}$ is an
isomorphism. 
To prove this we first show that $\ccChi_\bs$ is ind\yh-irreducible. To prove this, since $\ccChi_\bs$~is an open subset of~$\Chi_\bs$, by the isomorphism~\eqref{eqn6.1.2}, it suffices to show that~$\Chi$~is ind-irreducible. Further, since $\Chi=G\cdot(P/P, \bar{C}^+)$ (see above Lemma~\ref{lem:Chiopen}), it suffices to show that~$\overline{Pw_i^{-1}\uo^-}$ and $\overline{Pv^{-1}\uo}$ are
ind-irreducible. But, as observed earlier in the proof of Lemma~\ref{lemma6.2.1}  equality~\eqref{eq:89}, $\overline{Pw_i^{-1}\uo^-}=\overline{Bw_i^{-1}\uo^-}$. So, it is ind\yh-irreducible. Similarly, $\overline{Pv^{-1}\uo} =
\bigcup_{v_n\in W_P}\,X_{v_nv^{-1}}$, where $v_n$~is an increasing cofinal sequence in~$W_P$. This shows that~$\overline{Pv^{-1}\uo}$~is also 
 ind\yh-irreducible. Thus, $\Chi$~is  ind-irreducible. 
Take an open set $\Omega_\bs'\subset \ccChi_{\bs}$ and open set
$\Omega_\bs\subset {\dcirc\mcX}_\bs$ such that
$\eta_2\,:\,\Omega'_\bs\longto\Omega_\bs$ is an isomorphism.
Let $\xi\,:\,\Omega_\bs\longto\Omega'_\bs$ be its inverse.
Now, take an increasing cofinal sequence~$\{u_n\}_{n\geq 0}$ in $W$
and take the filtration $(Y_n)_{n\geq 0}$ of $ {\dcirc\mcX}_\bs$ by
$$
Y_n:={\dcirc\mcX}_\bs\cap (X^-_{u_n}\times X^-_{u_n}),
$$
where $X^-_{u_n}:=\overline{B^-u_n \uo^-}\subset G/B^-$.
Set
$Z_n:=\overline{\xi(\Omega_\bs\cap Y_n)}\subset \ccChi_{\bs}$.
Then, $Y_n$ is irreducible and normal (since so is $X^-_{u_n}$) and
$Z_n$ is closed and irreducible in $\ccChi_{\bs}$ (by definition).
Moreover, $\cup_nZ_n\supset\Omega'_\bs$.
Thus, by \cite[Lemma~3]{R:KM1}, $(Z_n)_{n\geq 0}$ provides an
irreducible filtration for~$\ccChi_{\bs}$.
Apply now Zariski's main theorem to the morphism $Z_n\longto Y_n$ to
conclude that $\Hd^0(Y_n,\Li)\longto \Hd^0(Z_n,\Li)$ is an isomorphism
for all $n\geq 0$. Taking the inverse limit, we get the lemma. 
\end{proof}

\begin{lemma}
\label{lem:isomi7}
 The restriction map~$ \Ho(\ccChi_\bs,\Li) \to \Ho(\tilde\Chi_\bs\cap \ccChi_\bs,\Li) $~is an isomorphism and hence so is $i_7^*$.
\end{lemma}
\begin{proof}
 As earlier, consider the action of~$U$ on~$X_v^P$:
$$
\theta\,:\,U\longto\Aut (X_v^P).
$$
Then, $\Image\, \theta$~is a finite-dimensional unipotent group $U_v$. 
 As a consequence,  $\Ker\,\theta$~is a normal subgroup of~$U$ of finite-codimension. 

Consider now the group 
$$
U_1=\Ker\,\theta\cap \left(\bigcap_{\alpha\in \Delta}s_\alpha Us_\alpha\right).
$$
Then, $U_1$~is again a normal subgroup of~$U$ of finite-codimension (i.e., $U/U_1$~is a finite-dimensional group). There exists a closed subgroup~$\mcU$ of~$U_1$ of finite-codimension such that~$\mcU$~is normal in~$U$, $\mcU^2:=\mcU\times \mcU$ acts freely and properly on~$\ccdX_\bs$ 
(under the action $(u_1, u_2)\cdot (x_1, x_2, \uo)=  (u_1x_1, u_2x_2, \uo)$) and the quotient
map~$\pi_{\mathcal{X}}: \ccdX_\bs\longto \mcU^2\backslash \ccdX_\bs$~is a principal $\mcU^2$\yh-bundle
(cf.  \cite[Lemma~6.1]{Kumar:positivity}). Moreover, since $\eta_2$~is proper (cf. Proof of Lemma~\ref{newlemma10}),  $\mcU^2$~acts freely and properly on~$ \ccChi_\bs$.

Consider the action of~$\mcU^2$  on~$X_v^P\times \X_\bs$ given by 
\begin{equation} \label{neweqn4.4.2}(u_1,u_2).(y,g_1\uo^-,g_2\uo^-, \uo)=(y,u_1g_1\uo^-,u_2g_2\uo^-, \uo).
\end{equation}
Since $\mcU$ acts trivially on~$X_v^P$ and $y\in X_v^P$, the condition
$y\in u_ig_iX^{w_i}_P$~is equivalent to~$y\in g_iX^{w_i}_P$. In
particular, $\Chi_\bs$, $\ccChi_\bs$ and $\tilde \Chi_\bs$ are all stable by the
action of~$\mcU^2$. Moreover, $\eta_2:  \ccChi_\bs \to \ccdX_\bs$~is  $\mcU^2$\yh-equivariant.

We consider the associated quotients:

\begin{center}
\begin{tikzpicture}
\matrix (m) [matrix of math nodes,row sep=1.2cm,column sep=1.2cm] {
\ccdX_\bs&\ccChi_\bs&\tilde\Chi_\bs\cap \ccChi_\bs\\
\mcU^2\backslash\ccdX_\bs& \mcU^2\backslash\ccChi_\bs&  \mcU^2\backslash\left(\tilde\Chi_\bs\cap \ccChi_\bs\right).\\
};
\path [->]     (m-1-2) edge  node[above]     {$\eta_2$} (m-1-1);
\path [commutative diagrams/.cd, every arrow, every label]     (m-1-3)
edge [commutative diagrams/hook']    (m-1-2);
\path [->]     (m-2-2) edge  node[above]     {$\bar \eta_2$} (m-2-1);
\path [commutative diagrams/.cd, every arrow, every label]     (m-2-3) edge [commutative diagrams/hook']    (m-2-2);
\path [->]     (m-1-1) edge  node[left]     {$\pi_\X$}    (m-2-1);
\path [->]     (m-1-2) edge  node[left]     {$\pi_\Chi$} (m-2-2);
\path [->]     (m-1-3) edge      (m-2-3);
\end{tikzpicture} 
\end{center}

Let $\Omega_\X$~be an open subset of~$ \mcU^2\backslash\ccdX_\bs$
such that the quotient~$\pi_\X$~is trivial over~$\Omega_\X$. (It can be seen that~$\pi_\X$~is locally trivial.) Set~$\Omega_\Chi=\bar \eta_2\inv(\Omega_\X)$.  Choosing a section of~$\pi_\X$ over~$\Omega_\X$ and taking the induced section 
of~$\pi_\Chi$ over~$\Omega_\Chi$, we get 
\begin{equation} \label{eqn4.4.1}
\pi_\X^{-1}(\Omega_\X)\simeq \mcU^2\times \Omega_\X\,\,\,\text{and}\,\, \pi_\Chi^{-1}(\Omega_\Chi)\simeq \mcU^2\times \Omega_\Chi
\end{equation}
such  that~${\eta_2}\restrict{\pi_\Chi^{-1}(\Omega_\Chi)}$ under the above isomorphism is given by~$$\eta_2(\tilde{u}, x)= (\tilde{u}, \bar{\eta}_2(x)), \,\,\,\text{for $\tilde{u}\in \mcU^2$ and $x\in \Omega_\Chi$.}$$

Since $\Li$~is $G^3$\yh-equivariant with $G^3$ acting on~$\X$ componentwise, we get that~$\Li\restrict{\ccdX_\bs}$ and  $\Li\restrict{\ccChi_\bs}$ are 
$\mcU^2$\yh-equivariant. Since $\mcU^2$ acts freely on~$\ccdX_\bs$ (resp.\ $\ccChi_\bs$), $\Li\restrict{\ccdX_\bs}$ (resp.\  $\Li\restrict{\ccChi_\bs}$) descends to a unique line bundle~$\bar\Li$ over~$\mcU^2\backslash\ccdX_\bs$ (resp.\ $\mcU^2\backslash\ccChi_\bs$). Hence, under the decompositions
\eqref{eqn4.4.1},
\begin{equation} \label{eqn4.4.2}
\Li\restrict{\mcU^2\times \Omega_\X}=\mathscr{O}_{\mcU^2}\boxtimes \bar{\Li}\restrict{\Omega_\X},\,\,\,\text{and}\,\, \Li\restrict{\mcU^2\times \Omega_\Chi}=\mathscr{O}_{\mcU^2}\boxtimes \bar{\Li}\restrict{\Omega_\Chi}.
\end{equation}
Now, the map~$$\bar\eta_2: \mcU^2\backslash\ccChi_\bs \to \mcU^2\backslash\ccdX_\bs$$
is proper. To prove this, consider the projection~$$\pi_2:\mcU^2\backslash (X_v^P\times \ccdX_\bs) =X_v^P\times  (\mcU^2\backslash\ccdX_\bs) \to \mcU^2\backslash\ccdX_\bs$$
with $\mcU^2$ acting on~$X_v^P\times \ccdX_\bs$ as in \eqref{neweqn4.4.2}. This is clearly a projective morphism. Now, 
$$\bar\eta_2 =( \pi_2)\restrict{\mcU^2\backslash\ccChi_\bs}.$$
Moreover, $\mcU^2\backslash\ccChi_\bs$~is a closed subset of~$\mcU^2\backslash (X_v^P\times \ccdX_\bs)$ (as can easily be seen) and hence
$\bar\eta_2 $~is a projective morphism. 

Further, $\bar\eta_2 $~is  a birational map since  so is $\eta_2$ (cf. Proof of Lemma~\ref{newlemma10}).

By the following lemma, $\bar\eta_2\left(\mcU^2\backslash (\ccChi_\bs\backslash
\tilde\Chi_\bs)\right)$~is of codimension~$\geq 2$ in~$\mcU^2\backslash\ccdX_\bs$. Moreover, $\mcU^2\backslash\ccdX_\bs$~is normal (cf. \cite[Proposition~3.2]{KS:Gpol}). In fact, it is smooth (cf. \cite[$\S$10]{Kumar:positivity}). Thus, by Proposition~\ref{prop:Zariski}, the restriction map 
\begin{equation}\label{eqn4.4.3}\Ho(\Omega_\Chi, \bar\Li)\to \Ho(\Omega_\Chi', \bar\Li)\,\,\,\text{is an isomorphism},
\end{equation}
for any open subset~$\Omega_\X \subset \mcU^2\backslash\ccdX_\bs$ over which $\pi_\X$ admits a section and $\Omega_\Chi := \bar\eta_2^{-1}(\Omega_\X)$, where $\Omega_\Chi':= \Omega_\Chi \cap \left(\mcU^2\backslash (\ccChi_\bs\cap \tilde\Chi_\bs)\right)$. But, by the decomposition \eqref{eqn4.4.2}
\begin{align}\label{eqn4.4.4}\Ho(\pi_\Chi^{-1}(\Omega_\Chi),\Li)&\simeq \Ho(\mcU^2\times \Omega_\Chi, \mathscr{O}_{\mcU^2}\boxtimes \bar\Li)
\notag\\
&=\varprojlim_n\, \CC[\mcU_n^2]\otimes \Ho( \Omega_\Chi,  \bar\Li),
\end{align}
where  $\{\mcU_n\}_{n\geq 0}$~is a filtration of~$\mcU$ giving the ind-variety structure. Similarly,
\begin{equation}\label{eqn4.4.5}\Ho(\pi_\Chi^{-1}(\Omega_\Chi'),\Li)
=\varprojlim_n\, \CC[\mcU_n^2]\otimes \Ho( \Omega_\Chi',  \bar\Li).
\end{equation}
Combining the Equations~\eqref{eqn4.4.3} -  \eqref{eqn4.4.5}, we get that the restriction map~$$\Ho(\pi_\Chi^{-1}(\Omega_\Chi),\Li) \to \Ho(\pi_\Chi^{-1}(\Omega_\Chi'),\Li)$$
is an isomorphism (use \cite[Chap.~II, Proposition~9.1]{Hart}). Since $\{\pi_\Chi^{-1}(\Omega_\Chi)\}$ provides an open cover of~$\ccChi_\bs$, we get that the restriction map~$$\Ho(\ccChi_\bs,\Li) \to \Ho(\tilde\Chi_\bs\cap \ccChi_\bs,\Li) $$ is
an isomorphism. This proves the lemma modulo Lemma~\ref{newlemma13} below.
\end{proof}

\subsection{Smallness of the boundary of~$\tilde\Chi_\bs$}

The goal of this subsection is to prove the following lemma. We refer to \cite[$\S$7]{BKR} for some parallel arguments. 

\begin{lemma}\label{newlemma13}
 With the notation as in the proof of Lemma~\ref{lem:isomi7}, the image $\bar\eta_2\left(\mcU^2\backslash (\ccChi_\bs\backslash
\tilde\Chi_\bs)\right)$ is of codimension~$\geq 2$ in~$\mcU^2\backslash\ccdX_\bs$.
\end{lemma}

This lemma will be a consequence of the nontransversality
Corollary~\ref{cor:nontransv}, which in turn is a  consequence of
Proposition~\ref{prop:freg}. 

\bigskip
Set, for~$i\in\mathbb{N}:=\{0, 1, 2, \dots \}$,
\begin{equation}
  \label{eq:3}
  (\lg/\lp)_i:=\{\xi\in \lg/\lp\,:\,\ad(x_P)\cdot \xi=-i\xi\},\,\,\,\text{where $x_P:=\sum_{\alpha_j\in \Delta\backslash \Delta(P)}\,x_j$}
\end{equation}
and
\begin{equation}
  \label{eq:5}
  (\lg/\lp)_{\leq i}:=\bigoplus_{j\leq i}(\lg/\lp)_j.
\end{equation}
Note that the $(\lg/\lp)_{\leq i}$'s  form a $P$\yh-stable filtration of~$\lg/\lp$.

Let $Z\subset G/P$~be a locally closed finite-dimensional subvariety
of~$G/P$ and let $z$~be a point of~$Z$. Write $z=gP/P$. Set, for~$i\in\NN$,
\begin{equation}
  \label{eq:4}
  d_i(z,Z):=\dim\left(
T_{\dot{e}}(g\inv Z)\cap (\lg/\lp)_{\leq i}
\right),\,\,\,\text{where $\dot{g}:=gP/P\in G/P$}.
\end{equation}

This indeed does not depend on the choice
of~$g$ such that~$z=gP/P$. Observe that\break$d_0(z,Z) = 0$,  $d_n(z,Z)=\dim T_zZ$ for~$n$ large enough, and that~$i\mapsto d_i(z,Z)$~is non-decreasing.
Define, for any~$i\in \mathbb{N}$,
$$\bar{d}_i(z, Z)= d_i(z,Z) - d_{i-1}(z,Z),$$
where we declare $ d_{-1}(z,Z)=0.$ Thus, $\bar{d}_m(z, Z) = 0$, for~$m>n$. 

Similarly, let~$z\in Z\subset G/P$, where $Z$~has finite-codimension. Write $z=gP/P$. Set, for~$i\in\NN$,
\begin{equation}
  \label{eq:41}
  d^i(z,Z):=\dim\left(
\frac{T_{\dot{e}}(g\inv Z)+(\lg/\lp)_{\leq i}}{T_{\dot{e}}(g\inv Z)}
\right).
\end{equation}

Again this  does not depend on the choice
of~$g$ such that~$z=gP/P$. Observe that\break$d^0(z,Z)=0$, that~$d^n(z,Z)$
is the codimension of~$T_zZ$ for~$n$ large enough, and that ${i\mapsto d^i(z,Z)}$ is non-decreasing.

\begin{PROP}
\label{prop:freg}
  Let~$v\in W^P$ and $\beta$ be a positive real root such that~$w=s_\beta v \in W^P$.

  \begin{enumerate}[label={\rm(\roman*)}]
  \item If $\ell(w)=\ell(v)-1$, then 
$$
 d_i(\dot{w},X_v^P)\geq d_i(\dot{v},X_v^P),\,\,\,\forall i\in\mathbb{N}.
$$
Moreover, if $\beta$~is {\bf not} a simple root,
$$
 d_{i_o}(\dot{w},X_v^P)>d_{i_o}(\dot{v},X_v^P), \,\,\, \text{for some}\,\, i_o\in\mathbb{N}.
$$
\item If $\ell(w)=\ell(v)+1$, then 
$$
 d^i(\dot{w},X^v_P)\leq d^i(\dot{v},X^v_P),\,\,\,\forall i\in\mathbb{N}.
$$
Moreover, if $\beta$~is {\bf not} a simple root,
$$
 d^{i_o}(\dot{w},X^v_P)<d^{i_o}(\dot{v},X^v_P),\,\,\,\text{for some}\,\, i_o\in\mathbb{N}.
$$
  \end{enumerate}
\end{PROP}

\begin{proof}
We first translate the first assertion in a combinatorial statement in
terms of roots. Given a $T$\yh-vector space~$E$, we denote by~$\Phi(E)$
the set of weights of~$T$ acting on~$E$.

Let~$\Phi^+$ (resp.\ $\Phi^-$) be the set of positive (resp. \ negative) roots. Since $T_{\dot{e}}(v\inv X_v^P)$~is multiplicity free as a $T$\yh-module, and
$\Phi(T_{\dot{e}}(v\inv X_v^P))=\{\theta\in\Phi^-\,:\,v\theta\in\Phi^+\}$, 
we have
\begin{equation}
  \label{eq:42}
   d_i(\dot{v},X_v^P)=\sharp 
\{\theta\in\Phi^-\,:\,v\theta\in\Phi^+{\rm\ and\ }-\theta(x_P)
\leq i\}, \,\,\,\forall i\geq 1.
\end{equation}
Consider the unique $T$\yh-stable curve~$\bl$ containing both $\dot{v}$ and
$\dot{w}$. Observe that~$\bl$~is isomorphic to~$\PP^1$,
$\Phi(T_{\dot{v}}\bl)=\{\beta\}$, $\Phi(T_{\dot{w}}\bl)=\{-\beta\}$ and
$\bl$~is contained in~$X_v^P$. 
Moreover, $X_w^P$~is contained in~$X_v^P$ and 
\begin{equation}
  \label{eq:46}
  T_{\dot{w}}X_v^P=T_{\dot{w}}X_w^P\oplus T_{\dot{w}}\bl.
\end{equation}
After translating by~$w\inv$, equality~\eqref{eq:46} implies that~$$
\Phi(T_{\dot{e}}(w\inv X_v^P))=\Phi(T_{\dot{e}}(w\inv X_w^P))\cup \{-w\inv \beta\}.
$$
It follows that 
\begin{equation}
  \label{eq:47}
  d_i(\dot{w},X_v^P)=\sharp 
\{\theta\in\Phi^-\,:\,w\theta\in\Phi^+{\rm\ and\ }
-\theta (x_P)\leq i\}+\delta^{(w^{-1}\beta)(x_P)}_i,\,\,\, \forall i\geq 1,
\end{equation}
where $\delta^m_i = 1 $ if $m\leq i$ and $0$ otherwise.

 We deduce that the
first assertion of the proposition is equivalent to~$\forall i\geq 1$:
\begin{equation}
  \label{eq:48}
  \begin{array}{ll}
 & \sharp 
\{\theta\in\Phi^-\,:\,w\theta\in\Phi^+{\rm\ and\ }-
\theta (x_P)\leq i\}+\delta^{(w^{-1}\beta)(x_P)}_i
   \\
&\geq\sharp 
\{\theta\in\Phi^-\,:\,v\theta\in\Phi^+{\rm\ and\ }   -
\theta (x_P)\leq i\},
\end{array}
\end{equation}
and the existence of~$i_o$ with a strict inequality~\eqref{eq:48} if
$\beta$~is not simple.

We now translate the second assertion of the proposition in a
combinatorial statement. 
First observe that, since $v\in W^P$,  
$$
\Phi\left(\frac{T_{\dot{e}}(G/P)}{T_{\dot{e}}(v\inv X^v_P)}\right)=
\{\theta\in\Phi^-\,:\,v\theta\in\Phi^+\}.
$$
We deduce
that 
\begin{equation}
  \label{eq:49}
\quad 
d^i(\dot{v},X^v_P)=\sharp 
\{\theta\in\Phi^-\,:\,v\theta\in\Phi^+{\rm\ and\ }    -
\theta (x_P)
\leq i\},\,\,\,  \forall i\geq 1.
\end{equation}
Now,  since $\ell(w)=\ell(v)+1$, $X_P^v\supset X_P^w$, $\bl$~is contained in~$X_P^v$, $\Phi(T_{\dot{w}}\bl)=\{\beta\}$ and
$\Phi(T_{\dot{v}}\bl)=\{-\beta\}$. 
Moreover, we have the following exact sequence~$$
0\longto T_{\dot{w}}\bl\longto \frac{T_{\dot{w}}(G/P)}{T_{\dot{w}} X^w_P}\longto
\frac{T_{\dot{w}}(G/P)}{T_{\dot{w}} X^v_P}\longto 0.
$$
After translation by~$w\inv$, we obtain that~$$
\Phi\left(\frac{T_{\dot{e}}(G/P)}{T_{\dot{e}} (w\inv X^w_P)}\right)=
\Phi\left(\frac{T_{\dot{e}}(G/P)}{T_{\dot{e}} (w\inv X^v_P)}\right)
\sqcup\{w\inv \beta\}.
$$
This implies
that
\begin{equation}
  \label{eq:50}
  d^i(\dot{w},X^v_P)=\sharp 
\{\theta\in\Phi^-\,:\,w\theta\in\Phi^+{\rm\ and\ }-\theta(x_P)
\leq i\}-\delta^{-(w^{-1}\beta)(x_P)}_i,\,\,\,\forall i\geq 1.
\end{equation}
With \eqref{eq:49} and \eqref{eq:50}, the second assertion of the
proposition is equivalent to~$  \forall i\geq 1$:
\begin{equation}
  \label{eq:51}
  \begin{array}{l}
\sharp 
\{\theta\in\Phi^-\,:\,v\theta\in\Phi^+{\rm\ and\ }-\theta(x_P)
\leq i\}
 \\
\geq\sharp 
\{\theta\in\Phi^-\,:\,w\theta\in\Phi^+{\rm\ and\ }-\theta(x_P)\leq i\}-\delta^{-(w^{-1}\beta)(x_P)}_i,
  \end{array}
\end{equation}
with a strict inequality for some~$i_o$, if $\beta$~is not simple.

Now, observe that given~$(v,w)$ such that~$w=s_\beta v$ and
$\ell(w)=\ell(v)+1$, one gets $(v',w')$ such that~$w'=s_\beta v'$ and
$\ell(w')=\ell(v')-1$ by setting $w'=v$ and $v'=w$. 
By \eqref{eq:48} and \eqref{eq:51}, the first assertion for~$(v',w')$
implies the second one for~$(v,w)$ (note that~${w'}\inv \beta=-w\inv
\beta$).
It is now sufficient to prove the first assertion.\\
 
From now on, we assume that~$\ell(w)=\ell(v)-1$.
Recall that we denote $v'\to v$ if $v'\in W^P$, $\ell(v')=\ell(v)-1$ and $v'\leq v$. 
Set~$$
\hat{X}_v^P=\cX^P_v\cup\left(\bigcup_{v'\to v}\cX^P_{v'}\right),\,\,\,\text{where $\cX^P_v:= BvP/P$}.
$$ 
Then, $\hat{X}_v^P$~is a smooth open subset of~$X_v^P$. Set~$$
\hat{Y}_v^P=\pi\inv (\hat{X}_v^P),
$$
where $\pi\,:\,G\longto G/P$~is the natural projection.
Define two vector bundles over~$\hat{Y}_v^P$:
$$
{\mathcal V}:=\bigcup_{g\in \hat{Y}_v^P}(\{g\}\times T_{\dot{e}} (g\inv
X_v^P))\longto \hat{Y}_v^P,
$$
and the trivial bundle~$$
\eps_i:=\hat{Y}_v^P\times\frac{T_{\dot{e}}(G/P)}{(\lg/\lp)_{\leq i}},
$$
for any fixed $i\in\mathbb{N}$.
The inclusion  $T_{\dot{e}} (g\inv X_v^P)\subset T_{\dot{e}}(G/P)$ induces a bundle
map~$$
\varphi_i\,:\, {\mathcal V}\longto\eps_i.
$$
  On the open subset~$BvP$, the rank of~$\varphi_i$~is constant since
$$
T_{\dot{e}}((bvp)\inv X_v^P)=T_{\dot{e}}(p\inv v\inv b\inv X_v^P)=p\inv
T_{\dot{e}}(v\inv X_v^P),
$$
and $(\lg/\lp)_{\leq i}$~is $P$\yh-stable. 

On the other hand, the subset of points in~$\hat{Y}_v^P$, where the rank of~$\varphi_i$~is maximum is  open.
Hence, this rank is maximum at any point of~$BvP\subset \hat{Y}_v^P$;
in particular, at~$v$. This shows the inequalities of the first
assertion of the proposition.

Note that~$$
\rho-v\inv \rho=-\sum_{\theta\in\Phi^-\cap v\inv \Phi^+}\theta.
$$
Hence,
$$
(\rho-v\inv \rho)(x_P)=\sum_{\theta\in\Phi^-\cap v\inv
  \Phi^+}
 -\theta (x_P).
$$
But by the Equation~\eqref{eq:42},
$$
 \sharp 
\{\theta\in\Phi^-\,:\,v\theta\in\Phi^+{\rm\ and\ }-\theta(x_P)
=i\}=d_i(\dot{v},X_v^P)-d_{i-1}(\dot{v},X_v^P), \,\,\,\forall i\geq 1.
$$
Hence,
$$
\begin{array}{ll}
  (\rho-v\inv \rho)(x_P)&=\sum_{j\geq
                                       1}j\bar{d}_j(\dot{v},X_v^P)\\
&=\ell(v)+ \sum_{j\geq
                                       2}(j-1) \bar{d}_j(\dot{v},X_v^P),
\end{array}
$$
since, by  the Equation~\eqref{eq:42},  $d_m=\ell(v)$ for large enough $m$. Similarly,
$$
  (\rho-w\inv \rho)(x_P)=\ell(w)+ \sum_{j\geq
                                       2}(j-1) \bar{d}_j(\dot{w},X_w^P).
$$
Since
$\ell(w)=\ell(v)-1$, we get
\begin{equation}
  \label{eq:53}
  (\rho-w\inv \rho-(\rho-v\inv \rho))(x_P)=
-1+\sum_{j\geq  2}\,(j-1)\left(\bar{d}_j(\dot{w},X_w^P)-\bar{d}_j(\dot{v},X_v^P)\right).
\end{equation}
On the other hand, since $w=s_\beta v$, we get 
\begin{align}
  \label{eq:53'}
  \rho-w\inv \rho-(\rho-v\inv \rho)&=-w\inv \rho+w\inv s_\beta \rho\notag\\
&=w\inv(s_\beta\rho-\rho)\notag\\
&=-\langle \rho, \beta^\vee \rangle w\inv \beta.
\end{align}
Combining the Equations~\eqref{eq:53} and  \eqref{eq:53'}, we get
$$
1+\sum_{j\geq  2}\,(j-1)\left(\bar{d}_j(\dot{v},X_v^P)-\bar{d}_j(\dot{w},X_w^P)\right)=\langle
\rho, \beta^\vee\rangle 
(w\inv \beta)(x_P).
$$
But by the Equation~\eqref{eq:42} (for $v$ replaced by~$w$) and the Equation~\eqref{eq:47}, we have
$$
d_i(\dot{w},X_v^P)=d_i(\dot{w},X_w^P)+
\delta^{(w^{-1}\beta)(x_P)}_i,\,\,\,\forall i\geq 1
$$
and hence (for $k:= (w\inv \beta)(x_P)$)
\begin{align}
  \label{eq:54}
1+&\sum_{j\geq 2, j\neq k}\,(j-1)\left(\bar{d}_j(\dot{v},X_v^P)-\bar{d}_j(\dot{w},X_v^P)\right)+ (k-1)\notag\\
& \left(\bar{d}_k(\dot{v},X_v^P)
-\bar{d}_k(\dot{w},X_v^P) +1\right) =
\langle \rho, \beta^\vee \rangle (w\inv \beta)(x_P).
\end{align}
If possible, assume that 
\begin{equation}
  \label{eq:97}
 d_j(\dot{v},X_v^P)=d_j(\dot{w},X_v^P),\,\,\,  \forall j\geq 1.
\end{equation}
Equivalently,
$$
 \bar{d}_j(\dot{v},X_v^P)=\bar{d}_j(\dot{w},X_v^P),\,\,\,  \forall j\geq 1.
$$
Then, the Equation~\eqref{eq:54} implies that~$$k=\langle\rho, \beta^\vee\rangle k.$$
But, $\langle \rho,  \beta^\vee
\rangle\geq 1$. 
Hence,
$$
\langle \rho, \beta^\vee\rangle=1.
$$
Observe that $k\ne0$ since $v=ws_{w^{-1}\beta}$ and $v,w\in W^P$. We deduce that~$\beta$~is simple if \eqref{eq:97} holds. This ends the
proof of the proposition.
\end{proof}

\begin{CORO}
\label{cor:nontransv}
  Let $w_1,w_2,v\in W^P$~be as in Theorem~\ref{th:restCisom}. In particular,  $\ell(v)=\ell(w_1)+\ell(w_2)$. 
Let~$x\in G/P$ and 
  $g,g_1,g_2$ in~$G$ be such that~$x$ belongs to~$g \hat{X}_v^P\cap g_1
 \hat{X}^{w_1}_P\cap g_2 \hat{X}^{w_2}_P$, where $ \hat{X}_v^P$~is as defined in the proof of Proposition~\ref{prop:freg} and
$$\hat{X}_P^w:=\cX_P^w\cup\left(\bigcup_{w\to w'}\cX_P^{w'}\right),\,\,\,\text{where $\cX_P^{w'}:= B^-w'P/P$}. $$
We assume that there exists a non-simple real
  root $\beta$ such that one of the  following two conditions holds:
  \begin{enumerate}[label={\rm(\roman*)}]
  \item $\ell(s_\beta v)=\ell(v)-1$, $s_\beta v\in W^P$ and $x\in g\cX^P_{s_\beta
    v}$.
\item  $\ell(s_\beta w_1)=\ell(w_1)+1$, $s_\beta v\in W^P$ and $x\in g_1\cX_P^{s_\beta
    w_1}$.
  \end{enumerate}
Then, the intersection~$g \hat{X}^P_v\cap g_1
  \hat{X}^{w_1}_P\cap g_2 \hat{X}^{w_2}_P$~is {\it not} transverse at~$x$.
\end{CORO}

\begin{proof}
  It suffices to prove that the standard linear map~$$
\theta\,:\,T_x (gX_v^P)\longto\frac{T_x (G/P)}{T_x(g_1X^{w_1}_P)}\oplus
\frac{T_x (G/P)}{T_x(g_2X^{w_2}_P)}
$$
is not an isomorphism.
Write $x=hP/P$. 
Up to changing $(g,g_1,g_2)$ by~$(h\inv g,h\inv g_1,h\inv g_2)$, we
may assume that~$h=e$.

Observe that~\begin{equation}\tag{$*$}
\theta\left(T_{\dot{e}}(gX_v^P)\cap (\lg/\lp)_{\leq i}
\right)\subset
\frac{(T_{\dot{e}}(g_1 X^{w_1}_P))+(\lg/\lp)_{\leq i}}{T_{\dot{e}}(g_1 X^{w_1}_P)}
\oplus
\frac{(T_{\dot{e}}(g_2 X^{w_2}_P))+(\lg/\lp)_{\leq i}}{T_{\dot{e}}(g_2 X^{w_2}_P)}.
\end{equation}
Moreover, since $\eps_P^v$ occurs with coefficient~$1$ (in particular, nonzero)  in 
the deformed product $\eps_P^{w_1}\odot_0 \eps_P^{w_2}$ by assumption, $\forall i\in \mathbb{N}$,
\begin{align*}\dim\left(T_{\dot{e}}(v^{-1}X_v^P)\cap (\lg/\lp)_{\leq i}
\right)&=
\dim\left(\frac{T_{\dot{e}}(w_1^{-1} X^{w_1}_P)+(\lg/\lp)_{\leq i}}{T_{\dot{e}}(w_1^{-1} X^{w_1}_P)}\right)\\
&\quad+\dim\left(
\frac{T_{\dot{e}}(w_2^{-1} X^{w_2}_P)+(\lg/\lp)_{\leq i}}{T_{\dot{e}}(w_2^{-1} X^{w_2}_P)}\right)
\end{align*}
(cf. \cite[Lemma~19]{R:KM1}).
But, Proposition~\ref{prop:freg} implies that, for some~$i_o$,
 the dimension of the first space $T_{\dot e}(gX^P_v)\cap(\mathfrak{g}/\mathfrak{p})_{i_o}$ in ($*$) is greater than 
that of the direct sum.
Hence, the restriction of~$\theta$ to~$T_{\dot{e}}(gX_v^P)\cap (\lg/\lp)_{\leq
  i}$ can not be injective. Thus, $\theta$ can not be an isomorphism.
\end{proof}

\begin{lemma}
  \label{lem:divcodimtwo}
Let $f\,:\,Y\longto X$~be a dominant morphism between two
quasi-projective irreducible varieties of the same dimension. 
Let $D\subset Y$~be an irreducible proper closed subset.

Then, if $\overline{f(D)}$~has codimension one in~$X$, then, for~$x\in D$
general, $f\inv(f(x))$~is finite.
\end{lemma}

\begin{proof}
  Otherwise, the general fibers of the restriction of~$f$ to~$f\inv(\overline{f(D)})$ would have positive dimension. Since
  $\overline{f(D)}$~has codimension one, this implies that~$\dim(f\inv(\overline{f(D)}))=\dim(Y)$ and hence
  $f\inv(\overline{f(D)})=Y$. But, $f$~is assumed to be dominant. A
  contradiction. 
\end{proof}

\begin{proof}[Proof of Lemma~\ref{newlemma13}]
 For~$(w_1',w_2',v')\in (W^P)^3$, we set~$$
\ccChi_\bs(w_1',w_2',v'):=\{(x,g_1\uo^-,g_2\uo_-, \uo)\in X_{v'}^P\times
\ccdX_\bs\,:\,x\in g_1X^{w_1'}_P\cap g_2X^{w_2'}_P\}
$$
and  
$$
\Chi_\bs(w_1',w_2',v'):=\{(x,g_1\uo^-,g_2\uo_-, \uo)\in X_{v'}^P\times
\mathcal{X}_\bs\,:\,x\in g_1X^{w_1'}_P\cap g_2X^{w_2'}_P\}.
$$ 
The set~$\ccChi_\bs\backslash
\tilde\Chi_\bs$~is the union of finitely many subsets of one of the
following types:

\subsubsection*{Type I} $\ccChi_\bs(w_1',w_2',v')$, where
$(w_1',w_2',v')\in (W^P)^3$,
$w_1'\geq w_1$, $w_2'\geq w_2$, $v'\leq v$ and
$\ell(w_1')+\ell(w_2')-\ell(v')\geq 2$.

\subsubsection*{Type II} $\ccChi_\bs(w_1,w_2,v')$, where
$v'\in W^P$,
$v'\leq v$, 
$\ell(v')=\ell(v)-1$ and $v'v\inv$~is not a simple reflection.

\subsubsection*{Type III} $\ccChi_\bs(w_1',w_2,v)$, where
$w_1'\in W^P$,
$w_1'\geq w_1$, 
$\ell(w_1')=\ell(w_1)+1$ and $w_1'w_1\inv$~is not a simple reflection.

\subsubsection*{Type IV} Like type III after exchanging $w_1$ and
$w_2$. 

\medskip

It is sufficient to prove that the image by~$\bar\eta_2$ of each one
of these subsets has codimension at least two in~$\mcU^2\backslash \ccdX_\bs$. 

Consider $(w_1',w_2',v')$ as in type I. There exists
$(w_1'',w_2'',v'')$ such that~$w_1'\geq w_1''\geq w_1$, ${w_2'\geq
w_2''\geq w_2}$, $v'\leq v''\leq v$ and $\ell(w_1'')+\ell(w_2'')-\ell(v'')=1$. 

The point~$(\dot{v}'',v''(w_1'')\inv\uo^-, v''(w_2'')\inv\uo^-)$ belongs
to~$\Chi_\bs(w_1'',w_2'',v'')$ and does not belong to~$\Chi_\bs(w_1',w_2',v')$. 
Hence, $\Chi_\bs(w_1'',w_2'',v'')\backslash \Chi_\bs(w_1',w_2',v')$~is open and
nonempty in~$\Chi_\bs(w_1'',w_2'',v'')$.

To prove the lemma in this type, we can assume that~$\ccChi_\bs(w_1',w_2',v') $~is nonempty, then  so is
$\ccChi_\bs(w_1'',w_2'',v'')$.
Since $\Chi_\bs(w_1'',w_2'',v'')$~is irreducible (cf. $\S$\ref{subsec4.1}), we deduce that the intersection $\left(\Chi_\bs(w_1'',w_2'',v'')\backslash\Chi_\bs(w_1',w_2',v')\right)\cap (G/P\times  \ccdX_s)$~is nonempty.

Thus, we have a strict inclusion $\ccChi_\bs(w_1',w_2',v')\subset
\ccChi_\bs(w_1'',w_2'',v'')$. 
Similarly, we have the  strict inclusion:
 $$\ccChi_\bs(w_1'',w_2'',v'') \subset \ccChi_\bs(w_1,w_2,v).$$
Combining the above two, we get the strict inclusions:
$$
\ccChi_\bs(w_1',w_2',v')\subset \ccChi_\bs(w_1'',w_2'',v'') \subset \ccChi_\bs(w_1,w_2,v).
$$
Since these varieties are irreducible and $\mcU^2$\yh-stable, we deduce
that~$\mcU^2\backslash \ccChi_\bs(w_1',w_2',v')$~is of codimension at
least two in~$\mcU^2\backslash \ccChi_\bs$. 

The lemma follows in this case since $\dim(\mcU^2\backslash
\ccChi_\bs)=\dim(\mcU^2\backslash \ccdX_\bs)$ since  $\bar\eta_2$~is a birational map (cf. Proof of Lemma~\ref{lem:isomi7}).

Let now $(w_1,w_2,v')$ be as in type II.

Assume, for contradiction, that~$\bar\eta_2(\mcU^2\backslash
\ccChi_\bs(w_1,w_2,v'))$~is a divisor.
By Lemma~\ref{lem:divcodimtwo}, there exists $(g_1,g_2)\in G^2$ such
that~$X_v^P\cap g_1X^{w_1}_P\cap g_2X^{w_2}_P$~is finite and there exists
${x\in \cX_{v'}^P\cap g_1\cX_P^{w_1}\cap g_2\cX_P^{w_2}}$; in particular, ~$(x,g_1\uo^-,g_2\uo^-, \uo)\in
\ccChi_\bs(w_1,w_2,v')$.

By Corollary~\ref{cor:nontransv}, the intersection~$\hat{X}_v^P\cap g_1\hat{X}^{w_1}_P\cap g_2\hat{X}^{w_2}_P$~is not transverse
at~$x$. 
Hence, the multiplicity of~$x$ in~$X_v^P\cap g_1X^{w_1}_P\cap g_2X^{w_2}_P$
is at least $2$. 
Since this intersection is finite, this implies that the coefficient of~$\eps_P^v$ in~$\eps_P^{w_1}\cdot \eps_P^{w_2}$: $n_{w_1,w_2}^v\geq
2$ (cf. \cite[Proof of Proposition~3.5]{BrownKumar}). A contradiction!

The last types III and IV work similarly.
\end{proof}

\subsection{Conclusion of the proof of Theorem~\ref{th:restCisom}}
Observe that~$\tilde\Chi_\bs \cap \ccChi_\bs$ being open in the  irreducible $\tilde\Chi_\bs$, $i_6^*$~is injective. 
Combining the results from Subsections~\ref{subsec4.2} - \ref{subsec4.4}, we get that~$$i_6^*\circ \gamma^*\circ\alpha^*\circ \eta^*: \Ho(\X, \Li)^G\to \Ho(\tilde\Chi_\bs \cap \ccChi_\bs,\Li)^B\,\,\,\text{is injective}$$
and
$$i_7^*\circ \eta_2^*\circ i_4^*\circ i_1^*: \Ho(\X, \Li)^G\to \Ho(\tilde\Chi_\bs \cap \ccChi_\bs,\Li)^B\,\,\,\text{is an isomorphism}.$$
From the commutative diagram $(\diamond)$ of Subsection~\ref{subsec4.1}, these two composite maps are equal forcing $\alpha^*$ to be an isomorphism. Thus, we get (from the top horizontal line of the commutative diagram $(\diamond)$) that the restriction map~$$\Ho(\X, \Li)^G \to \Ho(C, \Li)^L\,\,\,\text{is an isomorphism.}$$
This ends the proof of the theorem.\qed

\section{Proof of Theorem~\ref{th:codimface}}

In this section, $P$~is still a standard parabolic subgroup (and not necessarily maximal). We fix $(w_1,w_2, v)\in
(W^P)^{3}$  such that~$\eps_P^v$ occurs with coefficient~$1$ in 
the deformed product
$$\eps_P^{w_1}\odot_0 \eps_P^{w_2}
\in \bigl(\Hd^*(X_P,\ZZ), \odot_0\bigr).$$
In particular, $w_1, w_2\leq v$. Recall the definition of~$\mcD$ and $E_{\alpha, i}$ (for $(\alpha, i)\in \mcD$) from   
  subsection~\ref{subsection5.4}. 

\subsection{On the relative position of~$E_{\alpha,i}$ and $C$}

\begin{PROP} \label{prop:NC} For any~$(\alpha, i) \in \mcD$,
   the ind-variety
$$C:=Lw_1^{-1}\uo^-\times Lw_2^{-1}\uo^-\times Lv^{-1}\uo$$
is not contained in~$E_{\alpha,i}$.
\end{PROP}

To prove Proposition~\ref{prop:NC} we need the following lemma.

\begin{lemma}\label{newlemma12}
Let~$x\in G/P$ and $(g,g_1,g_2)\in G^3$ be such that the intersection~$$
 g_1\cX^{w_1}_P\cap g_2\cX^{w_2}_P\cap g\cX_v^P
$$ 
contains $x$ and is transverse at this point.
Such a choice is possible by \cite[Lemmas~6 and~7]{R:KM1}. 
Then, 
$$
 g_1X^{w_1}_P\cap g_2X^{w_2}_P\cap gX_v^P=\{x\}.
$$ 

In fact, the lemma remains true if we replace  $\cX^{w_i}_P$ (for any~$i=1,2$) by any~$B^-$\yh-stable open subset of~$\cX^{w_i}_P \cup \left(\bigcup_{w_i\to w_i'\in W^P}\,\cX^{w_i'}_P\right)$ and  $ \cX_v^P$ by any~$B$\yh-stable open subset of\break ${\cX_v^P\cup\left(\bigcup_{v'\to v, v'\in W^P}\, \cX_{v'}^P\right)}$. 
\end{lemma}

\begin{proof}
The strategy of the proof is to reduce the problem to a finite-dimensional situation (by quotient), and then to apply Zariski's main
theorem.

Up to a translation, we may assume that~$g$~is trivial. 
Since $G/B^-=\bigcup_{w\in W}wU\uo^-$, there exists, for~$i=1,2$, $u_i\in
W$ such that~$g_i\uo^-\in u_iU\uo^-$. Consider now
$$
\cChi_\bs=\{(y, h_1\uo^-,h_2\uo^-)\in  X_v^P\times u_1U\uo^-\times
u_2U\uo^- \,:\,
 y\in h_1X^{w_1}_P\cap h_2X^{w_2}_P\}
$$  
and its projection~$\eta$ to~$u_1U\uo^-\times u_2U\uo^-$.

Consider $\theta\,:\,U\longto \Aut(X_v^P)$ obtained by the action as before.
Fix $i\in\{1,2\}$.
Then, $\Ker\,\theta$~has finite-codimension in~$U$ and $U\cap
u_iUu_i\inv$~has finite-codimension in~$u_iUu_i^{-1}$.
It follows that there exists a closed normal subgroup~$\mcU_i$ of~$u_iUu_i\inv$ of finite-codimension such that~$$
\mcU_i\subset u_iUu_i\inv\cap \Ker\,\theta.
$$
Such a $\mcU_i$ can be obtained as a closed subgroup of~$U$ with Lie algebra
$$\Lie\, \mcU_i = \bigoplus_{\beta\in \Phi^+, |\beta|> N}\,\mathfrak{g}_\beta,$$
for large enough $N$ (depending upon $v$ and $u_i$), where, for~$\beta = \sum_j\,n_j\alpha_j$, $|\beta|:=\sum n_j$. 

The group $\mcU_1\times \mcU_2$ acts freely and properly on~$u_1U\uo^-\times u_2U\uo^-$ (and hence on~$\cChi_\bs$). Moreover,  $\eta$~is $(\mcU_1\times
\mcU_2)$\yh-equivariant. After quotient, one gets
$$
\bar\eta\,:\,
(\mcU_1\times \mcU_2)\backslash \cChi_\bs
\longto 
(\mcU_1\times \mcU_2)\backslash (u_1U\uo^-\times u_2U\uo^-).
$$
Observe that~$\mcU_i$ being closed subgroups of finite-codimension in~$u_iUu_i^{-1}$ and $X_v^P$ being finite-dimensional, the domain and the range of~$\bar\eta$ are finite-dimensional varieties and the range of~$\bar\eta$~is smooth and irreducible.

Since the coefficient of~$\eps_v^P$ in~$\eps_{w_1}^P\cdot \eps_{w_2}^P:$ $n_{w_1,w_2}^v=1$, the general fiber of~$\eta$~is one point
(see \cite[$\S$4.2]{R:KM1}). Further, as observed below the Equation~\eqref{eqn6.1.2}, $\Chi_\bs$~is irreducible and hence so is $(\mcU_1\times \mcU_2)\backslash \cChi_\bs$. 
Since the base field is $\CC$, this implies that~$\bar \eta$~is
birational.
Since $X_v^P$~is projective and $X^{w_1}_P$ and $X^{w_2}_P$ are closed in~$G/P$, it is easy to see that  the
map~$\bar\eta$~is proper. 
Now, we can apply Zariski's main theorem \cite[Chap. III, Corollary~11.4]{Hart} to conclude that  the fibers of~$\bar \eta$
are connected. 
But,  by assumption, $[x,g_1\uo^-,g_2\uo^-]$~is isolated in the fiber 
$\bar\eta\inv[g_1\uo^-,g_2\uo^-]$, where $[g_1\uo^-,g_2\uo^-]$ denotes the $(\mathcal{U}_1\times\mathcal{U}_2)$\yh-orbit of $(g_1\uo^-,g_2\uo^-)$.
Then, $\bar\eta\inv[g_1\uo^-,g_2\uo^-]=\{[x,g_1\uo^-,g_2\uo^-]\}$, that
is
$$
 g_1X^{w_1}_P\cap g_2X^{w_2}_P\cap X_v^P=\{x\}.
$$ 
 This proves the first part of the lemma.

The proof for the `In fact' statement in the lemma  is identical.
\end{proof}

\begin{proof}[Proof of Proposition~\ref{prop:NC}]
Since $\eps_P^v$ occurs with coefficient~$1$ in 
the deformed product
$\eps_P^{w_1}\odot_0 \eps_P^{w_2}$, by the proof of 
\cite[Lemma~19]{R:KM1}, there exist $l_1, l_2,
l_3\in L$ such that the intersection
\begin{equation}
(l_1w_1^{-1}\cX_P^{ w_1})\cap (l_2w_2^{-1}\cX_P^{ w_2})\cap (l_3 v^{-1}\cX^P_{v}) 
\end{equation}
is transverse at~$P/P$.
Then, Lemma~\ref{newlemma12} implies that the intersection $$(l_1w_1^{-1}X_P^{ w_1})\cap (l_2w_2^{-1}X_P^{ w_2})\cap (l_3 v^{-1}X^P_{v})$$  is reduced
to~$\{P/P\}$. 
In particular, if $w_1\leq s_\alpha w_1$ and $s_\alpha w_1\in W^P$,
\begin{equation}
(l_1w_1^{-1}X_P^{s_\alpha w_1})\cap (l_2w_2^{-1}X_P^{ w_2})\cap (l_3 v^{-1}X^P_{v}) =\emptyset.
\end{equation}
Then,
\begin{equation}\label{eqn1001}
(l_1w_1^{-1}\uo^-,  l_2w_2^{-1}\uo^-,  l_3v^{-1}\uo)\notin 
G\cdot\left(\overline{Pw_1^{-1}s_\alpha \uo^-}\times
\overline{Pw_2\inv \uo^-}\times \overline{Pv\inv \uo}\right).
\end{equation}
This proves that~$(l_1w_1^{-1}\uo^-,  l_2w_2^{-1}\uo^-,
l_3v^{-1}\uo)$ does not belong to~$E_{\alpha,1}$. 
The proposition follows for~$(\alpha, 1)$. 
The proof for~$(\alpha,i) \in \mcD$ for~$i=2,3$~is identical. 
\end{proof}

\subsection{The line bundles~$\Ni_{\alpha,i}$}
  
The goal of this subsection is to prove that~$\Ni_{\alpha,i}$ belongs
to the face considered in Theorem~\ref{th:codimface}:

\begin{PROP}
  \label{prop:Niface}
For any~$(\alpha, i) \in \mcD$, the center $Z(L)$  of~$L$ acts trivially on the restriction of~$\Ni_{\alpha,i}$
to~$C$, where $C$ is as in Proposition \ref{prop:NC}. 

In fact, for any~$L$\yh-equivariant line bundle~$\Li$ over~$C$ with
$\Ho(C, \Li)^L\neq 0$, $Z(L)$  acts trivially on~$\Li$.
In particular, if we write
$\Ni_{\alpha,i}=\Li^-(\lambda_1)\otimes\Li^-(\lambda_2)\otimes\Li(\mu)$, then
 for all $\alpha_j\not\in\Delta(P)$,
\[
\lambda_1(w_1x_{j})+ \lambda_2(w_2x_{j})-\mu(vx_j)=0. \tag{$I^j_{(w_1, w_2,v)}$}
\]
\end{PROP}

\begin{proof}
Consider a $G$\yh-invariant section~$ \mu_{\alpha,i}$ of~$\Ni_{\alpha,i}$ as guaranteed
 by Corollary \ref{coro6.6.1}.
For any $(\alpha,i)\in\mcD$, by Corollary \ref{coro6.6.1},
$Z(\mu_{\alpha,i})=E_{\alpha,i}$. Then 
 Proposition~\ref{prop:NC} implies that~$ \mu_{\alpha,i}$ 
restricts to a nonzero $L$\yh-invariant section on~$C$.

Since $Z(L)$ acts trivially on~$C$, it acts by a character on any line
bundle over~$C$. The existence of a nonzero $Z(L)$\yh-invariant section
implies that this character is trivial for the restriction of~$\Ni_{\alpha,i}$.  

Write
$\Ni_{\alpha,i}=\Li^-(\lambda_1)\otimes\Li^-(\lambda_2)\otimes\Li(\mu)$
and fix 
$\alpha_j\not\in\Delta(P)$.
There exists $d>0$, such that~$dx_j$~is the differential at~$1$ of a
one parameter subgroup  of~$Z(L)$. This one parameter subgroup acts
with weight~$\lambda_1(w_1x_{j})$,  $\lambda_2(w_2x_{j}) $ and
$\yh-\mu(vx_j)$ on the fiber over~$w_1\inv\uo^-$, $w_2\inv\uo^-$ and $v\inv\uo$
in~$\Li^-(\lambda_1)$,  $\Li^-(\lambda_2)$ and $\Li(\mu)$ respectively. Thus, the
equality $I^j_{(w_1, w_2,v)}$ follows proving Proposition~\ref{prop:Niface}. 
\end{proof}

\subsection{The line bundles~$\Ni_{\alpha,i}$ and the lines
  $\bl_{\beta,j}$}

Recall the definition of the line   $\bl_{\beta,j}$  from~$\S$1. We now study the restriction of the line bundle~$\Ni_{\alpha,i}$ to the lines
  $\bl_{\beta,j}$. This will be used to apply
  Theorem~\ref{th:restCisom}.

\begin{lemma}\label{lem:Nipos}
  Let~$(\alpha,i)\in \mcD$ and $(\beta,j)\in \mcD$ be two distinct elements.
Then, 
\begin{enumerate}[label={\rm(\roman*)}]
\item the degree of the restriction of~$\Ni_{\alpha,i}$ to~$\bl_{\alpha,i}$~is positive.
\item the degree of the restriction of~$\Ni_{\alpha,i}$ to~$\bl_{\beta,j}$~is nonnegative.
\end{enumerate}
\end{lemma}

\begin{proof}
   Take $(\alpha,1) \in \mcD$. Then, as in Section~\ref{sec1},
$$\bl_{\alpha,1} = (w_1^{-1}P_\alpha^-\uo^-, w_2^{-1}\uo^-, v^{-1}\uo).$$
Since the line bundle~$\Ni_{\alpha,i}$~has the form~$\Li^-(\lambda_1)\boxtimes \Li^-(\lambda_2)\boxtimes \Li(\mu)$ for some~$ 
(\lambda_1, \lambda_2, \mu)\in P_+^{3}$ (cf. Corollary \ref{coro6.6.1}),
$${\Ni_{\alpha,1}}\restrict{\bl_{\alpha,1}}\simeq \Li^-(\lambda_1)\restrict{w_1^{-1}P_\alpha^-\uo^-},$$
which is of degree
$$(w_1^{-1}\lambda_1)(w_1^{-1}\alpha^\vee)=\lambda_1(\alpha^\vee)\geq 0.$$
Assume, if possible,  that~$\lambda_1(\alpha^\vee)= 0$. Then, the zero set~$Z(\mu_{\alpha, 1})$ would be of the form~$\pi_\alpha^{-1}(S)$ for some~$S\subset G/P_\alpha^-\times G/B^-\times G/B$, where 
$$\pi_\alpha: G/B^-\times G/B^-\times G/B \to G/P_\alpha^-\times G/B^-\times G/B$$
is the projection. 

Then, by Corollary \ref{coro6.6.1} and 
Equation~\eqref{eq:defE},
$$Z(\mu_{\alpha, 1}) = E_{\alpha,1}=G\cdot \bar{C}^+_{s_\alpha w_1, w_2, v},$$
and hence we would have 
$$Z(\mu_{\alpha, 1}) \supset G\cdot \bar{C}^+_{w_1, w_2, v} =\X,$$
where the last equality follows from \cite[Proposition~3.5]{BrownKumar} since $\eps_P^v$ occurs with nonzero coefficient in~$\eps_P^{w_1}\cdot \eps_P^{w_2}$. This contradicts the nonvanishing of~$\mu_{\alpha, 1}$. Thus, 
$\lambda_1(\alpha^\vee) > 0$, proving (i) for~$(\alpha,1) \in \mcD$. 
The same proof works for any~$(\alpha,i) \in \mcD$ to prove (i).

To prove (ii), we still take  $(\alpha,1) \in \mcD$ and $(\beta,j)  \in \mcD$ for~$j=1,2$. Then, 
$${\Ni_{\alpha,1}}\restrict{\bl_{\beta,j}}\simeq \Li^-(\lambda_j)\restrict{w_j^{-1}P_\beta^-\uo^-},$$
which is of degree
$$(w_j^{-1}\lambda_j)(w_j^{-1}\beta^\vee)=\lambda_j(\beta^\vee)\geq 0.$$
For~$(\beta, 3) \in \mcD$,
$${\Ni_{\alpha,1}}\restrict{\bl_{\beta,3}}\simeq \Li(\mu)\restrict{v^{-1}P_\beta\uo},$$
which is of degree
$$(v^{-1}\mu)(v^{-1}\beta^\vee)=\mu (\beta^\vee)\geq 0.$$
This proves (ii) for~$(\alpha,1) \in \mcD$. The same proof gives (ii) for any~$(\alpha,i) \in \mcD$.
\end{proof}

\subsection{Conclusion of the proof of Theorem~\ref{th:codimface}}

Let $w_1,w_2, v\in W^P$~be as in Theorem~\ref{th:codimface}, i.e.,  $\eps_P^v$ occurs with coefficient~$1$ in 
the deformed product
$\eps_P^{w_1}\odot_0 \eps_P^{w_2}.$

Set~$d= 2 \dim \lh +\sharp \Delta(P)$. 
Let $\face=\face^P_{w_1, w_2, v}$~be the convex cone generated by the weights~$(\lambda_1,\lambda_2,\mu)\in\Gamma(\mathfrak{g})$ satisfying Identity $I^j_{(w_1,w_2,v)}$ for all $\alpha_j\not\in\Delta(P)$ as in Theorem~\ref{th:codimface}. 
Since the linear forms~$\{I^j_{(w_1, w_2,v)}\}_{\alpha_j\in \Delta \backslash \Delta(P)}$  restricted to~$E_\lg$ (cf. Proposition~\ref{prop:span}) defining $\face$ are
linearly independent, the dimension of~$\face$~is at most~$d$.

\bigskip
We now have to produce `enough' points in~$\face$. To do this we consider the
restriction map~$\Pic^{G^3}(\X)\longto\Pic^{L^3}(C)$ and we apply
Theorem~\ref{th:restCisom} to sufficiently many line bundles~$\Li$
such that~$\Ho(C,\Li\restrict{C})^L\neq\{0\}$. 

\bigskip 
Observe that, for any~$w\in W^P$, the map~$$L/B_L^- \to Lw^{-1}\uo^-\subset G/B^-, \,\,\, lB_L^- \mapsto lw^{-1}\uo^-\,\,\,\text{is an $L$-equivariant isomorphism}$$
and also the map~$$L/B_L \to Lw^{-1}\uo\subset G/B, \,\,\, lB_L \mapsto lw^{-1}\uo\,\,\,\text{is an isomorphism},$$
where $B_L:= B\cap L$~is the standard Borel subgroup of~$L$ and $B_L^-:= B^-\cap L$~is the standard opposite Borel subgroup of~$L$. (To prove the above two isomorphisms, use the fact that~$w\Delta_P\subset \Phi^+$.)
Thus,  the restriction map~$\Pic^{G^3}(\X)\simeq (\lh^*_\ZZ)^3
\longto\Pic^{L^3}(C)\simeq (\lh^*_\ZZ)^3$~is an isomorphism.
Let~${\mathfrak l}$ denote the Lie algebra of~$L$.

\begin{lemma} \label{lemma5.2}
There exist $\Li_1,\dots,\Li_d\in
\Pic^{G^3}(\X)$ such that 
\begin{enumerate}[label={\rm(\roman*)}]
\item $\Li_1,\dots,\Li_d\in
\Pic^{G^3}(\X)\otimes\QQ$ are linearly independent;
\item The restriction of each~$\Li_i$ to~$C$ belongs to~$\Gamma(\Liel)$.  
\end{enumerate}
\end{lemma}

\begin{proof}
By Proposition~\ref{prop:span}, $\Gamma({\mathfrak l})$~has dimension
$d$. (Observe that Proposition~3.1 remains valid for~$\mathfrak{l}$ by the same proof.) Hence, $\Gamma({\mathfrak l})\subset \Pic^{L^3}(C)\otimes_{\mathbb{Z}}\mathbb{Q}$ contains $d$~linearly independent elements. Then, the lemma follows from the isomorphism
${\Pic^{G^3}(\X)\simeq\Pic^{L^3}(C)}$.
\end{proof}

\begin{proof}[Proof of Theorem~\ref{th:codimface}]
Up to taking tensor powers, we may assume that the restriction of~$\Li_i$ to~$C$ admits a nonzero $L$\yh-invariant section~$\sigma_i$ (cf. \cite[Proof of Theorem~3.2]{BrownKumar}).

By Lemma~\ref{lem:Nipos}, there exists
$(a_{\alpha,i})_{(\alpha,i)\in\mcD}\in \NN^\mcD$ such that~$\Ni:=\sum_{(\alpha,i)\in\mcD}a_{\alpha,i} \Ni_{\alpha,i}$ satisfies:
$$\Li_k\otimes\Ni \,\,\,\text{is nonnegative for all $k$ when restricted to
  any~$\bl_{\beta,j}$ for~$(\beta,j)\in\mcD$. }$$
Moreover, up to changing $\Ni$ by~$2\Ni$ if necessary, we may assume
that $$\Li_1\otimes\Ni,\dots,\Li_d\otimes\Ni\in
\Pic^{G^3}(\X)\otimes\QQ$$ are linearly independent.

By Corollary ~\ref{coro6.6.1}, $\Ni$~has a $G$\yh-invariant section~$\sigma_\Ni$
that does not vanish identically on~$C$. Then,
$$
\tilde\sigma_i \in \Ho(C,\Li_i\otimes\Ni)^L\backslash\{0\}, \,\,\,\text{where $\tilde\sigma_i :=(\sigma_i\otimes {\sigma_\Ni})\restrict{C}$}.
$$
Moreover, since $\tilde\sigma_i$  is not identically zero on~$C$, by Proposition~\ref{prop:Niface}, each~$\Li_i\otimes \Ni$ satisfies the identity $I^j_{(w_1, w_2, v)} $  of Theorem~\ref{th:codimface} for all $\alpha_j\in \Delta\backslash \Delta(P)$.
 
By Theorem~\ref{th:restCisom}, each~$\tilde\sigma_i$
can be extended to a $G$\yh-invariant section~$\tilde\sigma_i$ of~$\Li_i\otimes\Ni$. In particular, $\Li_i\otimes\Ni$ belongs to~$\Gamma(\lg)$. 
Thus, the dimension of~$\face$~is at least $d$ and hence it is exactly $d$. This proves the theorem.\end{proof}

\bibliographystyle{smfalpha}
\bibliography{irredKM}
\end{document}